\documentclass[11pt,reqno]{amsart}
\usepackage{amsthm,amssymb,amsmath}

\newtheorem{theorem}{Theorem}[section]
\newtheorem*{theorem*}{Theorem}
\newtheorem{lemma}[theorem]{Lemma}
\newtheorem{corollary}[theorem]{Corollary}
\newtheorem*{corollary*}{Corollary}
\newtheorem{proposition}{Proposition}[section]

\theoremstyle{definition}

\newtheorem*{definition*}{\sc Definition}

\newtheorem{example}{\bf Example}

\newtheorem*{remark*}{\bf Remark}

\newtheorem*{example*}{\bf Example}
\newtheorem*{remarks}{\bf Remarks}

\newcommand{\loc}{{\rm loc}}

\newcommand{\Real}{{\rm Re\,}}
\newcommand{\Imag}{{\rm Im\,}}

\newcommand{\clos}{{\rm clos}}
\newcommand{\sgn}{{\rm sgn\,}}

\expandafter\def\expandafter\normalsize\expandafter{%
    \normalsize
    \setlength\abovedisplayshortskip{8pt}
    \setlength\belowdisplayshortskip{8pt}
}

\include{diagxy}

\begin{document}

\title[Theory of Kolmogorov operator in spaces $L^p$ and $C_\infty$]{On the theory of the Kolmogorov operator \\ in the spaces $L^p$ and $C_\infty$. I}

\author{D.\,Kinzebulatov and Yu.\,A.\,Sem\"{e}nov} 

\address{Universit\'{e} Laval, D\'{e}partement de math\'{e}matiques et de statistique, 1045 av.\,de la M\'{e}decine, Qu\'{e}bec, QC, G1V 0A6, Canada}

\email{damir.kinzebulatov@mat.ulaval.ca}

\thanks{The research of D.K. is supported by the Natural Sciences and Engineering Research Council of Canada and the Fonds de recherche du Qu\'{e}bec -- Nature et technologies}

\address{University of Toronto, Department of Mathematics, 40 St.\,George Str, Toronto, ON, M5S 2E4, Canada}

\email{semenov.yu.a@gmail.com}

\begin{abstract}
We establish the basic results concerning the problem of constructing operator realizations of the formal differential expression $\nabla \cdot a \cdot \nabla - b \cdot \nabla$ with measurable matrix $a$ and vector field $b$ having critical-order singularities as the generators of Markov semigroups in $L^p$ and $C_\infty$.
\end{abstract}



\maketitle

\section*{Introduction}

Let $\mathcal L^d$ be the Lebesgue measure on $\mathbb R^d$, $d \geq 3$. Let $\Omega$ be an open set in $\mathbb R^d$. 
Our object of study is the formal differential operator
$$
(-\nabla \cdot a \cdot \nabla + b \cdot \nabla)u(x) =  - \sum_{i,k=1}^d \partial_{ x_i} \big (a_{ik}(x) \partial_{x_k} u(x) \big) + \sum_{i=1}^d b_i(x) \partial_{x_i} u(x), \quad x \in \Omega.
$$
Under rather general assumptions on the matrix $a$,
\begin{align*}
& a = a^* : \Omega \rightarrow \mathbb{R}^d \otimes \mathbb{R}^d , \quad a \in \big[ L^1_\loc (\Omega, \mathcal{L}^d) \big]^{d\times d},\\
& \sigma I \leq a(x) \text{ for some constant } \sigma > 0 \text{ and $\mathcal{L}^d$ a.e. } \;\;x \in \Omega,
\tag{$H_1$}
\end{align*}
we will construct by means of the theory of the quadratic forms in Hilbert spaces, three operator realizations $A_D, A_{iD}, A_N$ of $-\nabla \cdot a \cdot \nabla$ in $L^2 = L^2(\Omega,\mathcal L^d).$ Each of these realizations is the (minus) generator of a symmetric Markov semigroup and inherits some basic properties of the classical Dirichlet and Neumann extensions of $-\Delta.$ Let $A$ denote one of these operators and let $\{e^{-tA_r}, t \geq 0 \}_{1 \leq r < \infty }$ be the collection of consistent $(e^{-tA_p} \upharpoonright L^p \cap L^q = e^{-tA_q} \upharpoonright L^p \cap L^q, 1 \leq p, q < \infty, \; A_2 \equiv A )$ $C_0$ semigroups on the scale $[1,\infty[$ of $L^r$ spaces.

For any $\mathcal L^d$ measurable $b: \Omega \rightarrow \mathbb R^d$ we define in $L^r$ the maximal operator $B_r \supset b \cdot \nabla$ of domain $\{f \in L^r \cap W^{1,1}_\loc(\Omega) \mid b\cdot \nabla f \in L^r \}.$

$\mathbf 1.$ Assuming that
\[
b_a := \sqrt{b \cdot a^{-1} \cdot b} \in L^d + L^\infty, \tag{$C_d$}
\]
we will prove, essentially using specific properties of the symmetric Markov semigroup $e^{-tA}$ and the structure of $B_r,$ that $B_r$ \textit{is $A_r$ bounded (with relative bound zero) in $L^r$ for all} $r \in ]1,\frac{2d}{d+2}].$

The interval $]1,\frac{2d}{d+2}]$ cannot be enlarged to $[1,\frac{2d}{d+2}]$ under the assumption $(C_d).$

 By means of the standard tools of the perturbation theory for linear operators one concludes that the algebraic sum $A_p + B_p$ of domain $D(A_p) \cap D(B_p)$, $p \in ]1, \frac{2d}{d+2}]$ is the (minus) generator of a quasi bounded $C_0$ semigroup in $L^p,$ say $\Lambda_p \equiv \Lambda_p(a,b).$ Moreover, essentially using specific properties of the symmetric Markov semigroup $e^{-tA}$ we will prove that $T^t_p := e^{-t\Lambda_p}$ is a Markov semigroup (i.e.\,positivity preserving, quasi contraction, $L^\infty$ contraction $C_0$ semigroup), so the whole family $\{T^t_r\}_{1 < r < \infty }$ is well defined. Let $-\Lambda_r$ denote the generator of $T^t_r.$ Then
 
 1) $\Lambda_r \supset A_r + B_r, \; 1 < r < \infty$ and $\Lambda_r = A_r +B_r$ only for $r \in ]1,\frac{2d}{d+2}].$
 
 2)  In place of $A$ one can substitute the form-sum $A \dot{+} V$ provided that $Vf(x)=V(x)f(x),$ $V$ is $\mathcal L^d$ measurable, $0 \leq V$ and $D(A^\frac{1}{2})\cap D(V^\frac{1}{2})$ is dense in $L^2.$ The simplest sufficient condition is $V \in L^1_{\loc}$, see, however, the example in \cite{StV}.

\smallskip

$\mathbf 2.$ In order to treat more singular $b$'s we introduce the class $\mathbf F_\delta(A)$ of form-bounded vector fields. We say that $b:\Omega \rightarrow \mathbb R^d$ belongs to $\mathbf F_\delta(A)$, and write $b \in \mathbf F_\delta(A)$ if and only if $b$ is $\mathcal L^d$ measurable, $b_a^2 \in L^1_\loc$ and there are constants $\delta > 0, \; 0 \leq \lambda < \infty$ such that $$\|b_a (\lambda + A)^{-\frac{1}{2}} \|_{2 \rightarrow 2} \leq \sqrt{\delta}.$$

Under the following assumptions on $(a,b):$
 \[
a\in (H_1), \; b \in \mathbf F_\delta(A), \; \delta < 4, \text{ and } b_a^2 \in L^1 + L^\infty  \tag{$\ast$}
 \]
we will construct an operator realization $\Lambda_r(a,b)$ of $-\nabla \cdot a \cdot \nabla + b \cdot \nabla$ in $L^r$ for every $r \in I_c:=[\frac{2}{2-\sqrt{\delta}},\infty[$ such that $-\Lambda_r(a,b)$ is the generator of a Markov semigroup. This semigroup is holomorphic for $r \in I^o_c :=]\frac{2}{2-\sqrt{\delta}},\infty[.$ 
 For every $r \in I^o_c,$ 
\begin{equation}
\label{conv_prop}
\tag{$\ast\ast$} 
 e^{-t\Lambda_r(a,b)}:=s \mbox{-} L^r \mbox{-} \lim_{n\uparrow \infty} e^{-t\Lambda_r(a,b_n)},
 \end{equation}
  where $b_n := \mathbf 1_n b,$ $\mathbf 1_n$ denotes the characteristic function of the set $\{x \in \Omega \mid b_a (x) \leq n \}$;
 \[ 
   \|e^{-t\Lambda_r(a,b)}\|_{r\rightarrow q}\leq c e^{t\omega_r} t^{-(\frac{1}{r}-\frac{1}{q})\frac{d}{2}}, \qquad \omega_r= \frac{\lambda \delta}{2(r-1)}, \; r < q \leq \infty, \;r \in I_c;
\]
The interval $I_c$ is called the interval of contraction solvability.

Note that any $b$ satisfying the assumption $(C_d)$ belongs also to $\mathbf F_0 (A) := \bigcap_{\delta > 0} \mathbf F_\delta(A).$ In particular, for each $n=1,2,\dots,$ $b_n$ satisfies $(C_d)$ and belongs to $\mathbf F_\delta(A).$ Again in place of $A$ one can substitute $A \dot{+} V$ provided that $V$ is $\mathcal L^d$ measurable, $0 \leq V$ and $D(A^\frac{1}{2})\cap D(V^\frac{1}{2})$ is dense in $L^2.$

For $0<\delta<1$, \eqref{conv_prop} can be viewed as a (fundamental) property of the semigroup $e^{-t\Lambda_r(a,b)}$, however, for $1 \leq \delta < 4$, \eqref{conv_prop} becomes the principal means of construction of $e^{-t\Lambda_r(a,b)}$.
   
Next, we will prove that, for more regular matrices  the constraint $ b_a^2 \in L^1 + L^\infty$ in $(\ast)$ is \textit{superfluous}. This is true in particular for any $a\in (H_u),$ the class of uniformly elliptic matrices (i.e.\,$a\in (H_1)$ and $a(x) \leq \xi I$ for some constant $\xi$ and $\mathcal L^d$ a.e. $x\in \Omega).$

Moreover, it will be shown that if
\[
(a,b)=(a \in (H_u), \;b \in \mathbf F_\delta(A), \; \delta < 4),
\]
then, $e^{-t\Lambda_p(a,b)} \upharpoonright L^r \cap L^p, \; p \in I_c^o,$ extends to a bounded holomorphic semigroup in $L^r$ for each $r \in I_m-I_c,$ where  $I_m := ]\frac{2}{2- \frac{d-2}{d} \sqrt{\delta}}, \infty[.$
 \footnote{The maximal interval of quasi bounded solvability for $A - V$, with $0 \leq V \leq \delta A+c(\delta)$, $0<\delta<1,$ is $\hat I_m:=]r(\delta), r'(\delta) [,$ $r'(\delta) := \frac{2}{1 - \sqrt{1-\delta}}\frac{d}{d-2}.$ This was proved in 1995 by Yu.A.\;Sem\"enov, based on ideas set forth in \cite{SV}. The fact that the semigroup associated with the Schr\"{o}dinger  operator $-\Delta - V$, $V \in L^{\frac{d}{2},\infty}$, can be extended to a $C_0$ semigroup on $L^r(\mathbb R^d)$ for every $r\in \hat I_m$  was first observed in \cite{KPS}.}

 We will present examples of $(a,b)$ which show that $I_c, I_m$ are maximal, as well as examples of $(a,b)=(a \in (H_u), \; b \in \mathbf F_\delta(A), \; \delta > 4)$ which show that the constraint $\delta < 4$ has direct bearing on the subject matter. 

The class $\mathbf F_\delta(A)$ contains vector fields $b$ having critical-order singularities: the basic properties of $\Lambda_r(a,b)$ (including the smoothness of $D(\Lambda_r(a,b))$) exhibit  quantitative dependence on the value of $\delta$ (note that $b \in \mathbf{F}_\delta(A)$ if and only if $cb \in \mathbf{F}_{c^2\delta}(A)$, $c>0$, so $\delta$ effectively plays the role of a ``coupling constant'' for $b\cdot\nabla$). See examples in sections \ref{fbd_sect} and \ref{weak_fbd_sect}.

Now consider the following assumption on $(a,b):$
\[
a \in (H_1) \text{ and } b \in \mathbf F_\delta(A), \; \delta < 1.
\]
(removing in $(\ast)$ the constraint $ b_a^2 \in L^1 + L^\infty$ but restricting the range of $\delta$).

Using old ideas of J.-L.\,Lions and E.\,Hille we will construct in $L^2$ a Markov semigroup $e^{-t\Lambda(a,b)},$ which possesses some important properties:
\[
e^{-t\Lambda_r(a,b)} = s \mbox{-} L^r \mbox{-} \lim_{n\uparrow \infty} e^{-t\Lambda_r(a,b_n)}, \quad (2 \leq r < \infty)
\]
whenever $\{b_n \} \subset \mathbf F_\delta(A)$ and $b_n \rightarrow b \; \; \mathcal L^d$ a.e.
\[ 
   \|e^{-t\Lambda_r(a,b)}\|_{r\rightarrow q}\leq c e^{t\omega_r} t^{-(\frac{1}{r}-\frac{1}{q})\frac{d}{2}}, \qquad \omega_r= \frac{\lambda \delta}{2(r-1)}, \; 2 \leq r < q \leq \infty.
\]
 $\Lambda \supset A + B, \; D(\Lambda)\subset D(A^\frac{1}{2}),$ the resolvent set of $-\Lambda$ contains $\mathcal O := \{\zeta \mid \Real \zeta > \lambda \}.$  
\[
(\zeta + \Lambda)^{-1} = (\zeta + A)^{-\frac{1}{2}}(1+T_\zeta)^{-1} (\zeta + A)^{-\frac{1}{2}}, \; \|T_\zeta\|_{2\rightarrow 2}\leq \sqrt{\delta} \qquad (\zeta \in \mathcal O).
\]

Under more restrictive assumption on $(a,b)$,
\[
a \in (H_u) \text{ and } b \in \mathbf F_\delta(A), \; \delta < 1,
\]
the results obtained above by different techniques are unified. In particular, the interval $[2, \infty[$ extends to $I_c$ and $I_m.$

It is useful (in some cases necessary) to have the convergence
\[
e^{-t\Lambda_r(a,b)} = s \mbox{-} L^r \mbox{-} \lim_{n\uparrow \infty} e^{-t\Lambda_r(a_n,b_n)},
\]
where $a,a_n \in (H_1),$ $b \in \mathbf F_\delta(A), \; b_n \in \mathbf F_\delta(A_n), \; \delta < 4,$ $a_n, b_n$ bounded and smooth, and $a_n \rightarrow a, \; b_n \rightarrow b $ $\mathcal L^d$ a.e.

We prove this convergence, although under
the additional constraints $a \in (H_u)$ and $\delta <1$. To make this result \textit{unconditional} we 
need to address
 the following problem: Given $(a,b)=(a \in (H_u), b \in \mathbf F_\delta(A))$, to construct $(a_n,b_n)$ with the claimed properties.
In the simplest case $\Omega = \mathbb R^d, \; a = a_n = I$ we solve the problem simply putting $b_n = E_n(b\mathbf 1_{\Omega_n}),$ where $\Omega_n = B(0,n)$ and $E_nf=e^{\varepsilon_n \Delta}f$ or $E_n f:= \gamma_{\varepsilon_n} * f, \; \gamma_{\varepsilon_n}$ is the K. Friedrichs mollifier,  with evident modifications for $\Omega \subset \mathbb R^d.$ $(\Omega = \bigcup_n \Omega_n, \; \Omega_n \Subset \Omega_{n+1}, \; \Omega_n$ are open and bounded). 

We also mention the following result. Set $\nabla_i u(x):=\partial_{x_i}u(x)$, $(\nabla a)_k = \sum_{i=1}^d (\nabla_i a_{ik})$ and $b:=(\nabla a)$,
so formally $-\nabla \cdot a \cdot \nabla + b \cdot \nabla = - a \cdot \nabla^2 \; (\equiv -\sum_{i,k}^d a_{ik} \nabla_i \nabla_k)$. Set $a_n = E_n a$, $b_n:=(\nabla a_n)$. Fix $\delta < \infty.$ Then $b \in \mathbf F_{\delta}(A) \Rightarrow b_n \in \mathbf F_\delta(A_n)$ provided that the matrix $a \in (H_u)$ is diagonal and $|b|\in L^2 + L^\infty.$

\smallskip

$\mathbf 3.$ Let $C_\infty=C_\infty(\mathbb R^d)$ denote the space of all continuous functions vanishing at infinity endowed with the $\sup$-norm. Our next concern is: to find a subclass of $(H_u)$ and constraints on $\delta$ which allow to construct in $C_\infty$ a $C_0$ semigroup associated with $-\nabla \cdot a \cdot \nabla + b \cdot \nabla$ whenever $b \in \mathbf F_\delta(A).$

As the first step we consider the case
\[
(a,b)=(I,b \in \mathbf F_\delta) \text{ with } \sqrt{\delta} < 1 \wedge \frac{2}{d-2}, \; d \geq 3.
\]
The constraint on $\delta$ allows to establish the following fundamental fact \cite{KS}:
\[
(\zeta + \Lambda_q)^{-1} L^q \subset W^{1,qj}, \; \zeta \in \rho(-\Lambda_q), \; j = \frac{d}{d-2}, \tag{$\star$}
\]
whenever $q \in \big]2 \vee (d-2), \frac{2}{\sqrt{\delta}}\big[\;\; ( \subset I_c^o).$
In particular, $e^{-t \Lambda_r} L^r \subset C^{0,1-\frac{d}{qj}}$ for any $r\in ]\frac{2}{2-\sqrt{\delta}}, q].$

Note that $(\star)$ is a ``trivial" fact only for $(d,q)=(3,2)$. Indeed, $(b \in \mathbf F_\delta, \; \delta < 1)$ implies that $\Lambda_2 = -\Delta + b \cdot \nabla, \; D(\Lambda_2) = W^{2,2}\subset W^{1,2j}.$ Thus, if $d=3,$ then $(\zeta + \Lambda_2)^{-1} L^2 \subset W^{1,6} \subset C^{0,\frac{1}{2}}.$ However, already for $d=4,$ $(\zeta + \Lambda_2)^{-1} L^2 \subset W^{1,p}, \; p=d,$ not $p > d.$

We emphasize that the assumption $b \in \mathbf{F}_\delta$ does not guarantee 
$W^{2,r}$ estimates on $(\zeta + \Lambda_r)^{-1}L^r$ for $r$ large enough to conclude that, for any $t>0$, $e^{-t \Lambda_r} L^r \subset C^{0,\alpha}.$

We will also discuss the analogue  of ($\star$) for $(a,b\in \mathbf F_\delta(A)),$ $a(x)=I+c|x|^{-2}x \otimes x$, $c>-1.$

Armed with $(\star)$ we will prove that
\[
 s \mbox{-} C_\infty \mbox{-} \lim_{n\uparrow \infty} e^{-t\Lambda_{C_\infty}(b_n)}, \quad b_n = E_n(b \mathbf 1_{B(0,n)})
\]
exists uniformly in $t\in [0,1],$ and hence determines a $C_0$ semigroup $e^{-t\Lambda_{C_\infty}}$ (a Feller semigroup), whose generator is an appropriate realization of $\Delta - b \cdot \nabla$ on $C_\infty$.
(This result has been established in \cite{KS} under the additional assumption $| b | \in L^2 + L^\infty$.)
We emphasize that in general there is no direct connection between $\Lambda_{C_\infty}$ and the algebraic sum $-\Delta + b \cdot\nabla$ even for $b \in [L^\infty]^d - [C]^d$. \textit{In particular, $C_c^\infty \not\subset D(\Lambda_{C_\infty})$}. The same remark applies to $\Lambda_r$, $r>2$, whenever $|b| \in L^q - L^r $, $q<r$.

\smallskip

\textbf{4.~}Next, consider the case ($\Omega=\mathbb R^d$)
\[
(a,b)=(I,b \in \mathbf F_\delta^\frac{1}{2} )\text{ with } \delta < 1, \; d \geq 3.
\] 
By definition, $b \in \mathbf F_\delta^{\scriptscriptstyle 1/2}$ if and only if $|b|\in L^1_\loc$ and there exists $0 \leq \lambda= \lambda_\delta$ and  $$\||b|^\frac{1}{2}(\lambda - \Delta)^{-\frac{1}{4}} \|_{2 \rightarrow 2} \leq \sqrt{\delta}.$$
Recalling the definition of the Kato class $\mathbf K^{d+1}_\delta$: $ \||b|(\lambda - \Delta)^{-\frac{1}{2}} \|_{1\rightarrow 1} \leq \delta,$ it is a simple matter to conclude that $\mathbf K^{d+1}_\delta \subset \mathbf F_\delta^{\scriptscriptstyle 1/2}$ while $\bigcap_{\delta> 0}\mathbf K^{d+1}_\delta - \bigcup_{\delta < \infty} \mathbf F_\delta \neq \varnothing.$ Taking into account that by interpolation, $\mathbf F_{\delta^2} \subset \mathbf F_\delta^{\scriptscriptstyle 1/2},$ it is clear that $\mathbf F_{\delta^2} \subsetneqq \mathbf F_\delta^{\scriptscriptstyle 1/2}.$

 To deal with such general class of vector fields we will again use ideas of E.\,Hille and J.-L.\,Lions (alternatively, ideas of E.\,Hille and H.F.\,Trotter) to construct the generator $-\Lambda \equiv -\Lambda(b)$ (an operator realization of $\Delta - b\cdot \nabla )$ of a quasi bounded holomorphic semigroup in $L^2.$ This operator has some remarkable properties. Namely, $\rho(-\Lambda) \supset \mathcal O := \{\zeta \mid \Real \zeta > \lambda \},$ and, for every $\zeta \in \mathcal O,$
\begin{align*}
(\zeta +\Lambda)^{-1} = & J_\zeta^3 (1 + H^*_\zeta S_\zeta))^{-1} J_\zeta \\
= & J^4_\zeta -  J_\zeta^3 H^*_\zeta (1 + S_\zeta H^*_\zeta)^{-1} S_\zeta J_\zeta; \\
\|H^*_\zeta S_\zeta \| \leq & \delta, \quad \|(\zeta + \Lambda)^{-1} \|_{2 \rightarrow 2} \leq |\zeta|^{-1} (1-\delta)^{-1}; \\
\|e^{-t \Lambda_r} \|_{r \rightarrow q} \leq & c \; e^{t\lambda} t^{-\frac{d}{2}(\frac{1}{r}-\frac{1}{q})}, \quad 2 \leq r < q \leq \infty;
\end{align*}
where $J_\zeta := (\zeta -\Delta)^{-\frac{1}{4}}, \; H_\zeta := |b|^\frac{1}{2} J_{\bar{\zeta}}, \; S := b^\frac{1}{2} \cdot \nabla J^3_\zeta, \; b^\frac{1}{2} := |b|^{-\frac{1}{2}}b$.

(In particular, $D(\Lambda) \subset \mathcal W^{\frac{3}{2},2}$, the Bessel potential space. In this regard, we note that the Kato-Lions-Lax-Milgram-Nelson Theorem applied to the operator $-\Delta + b \cdot \nabla$ requires $b \in \mathbf{F}_\delta$, $\delta<1$, a more restrictive assumption, while giving a weaker regularity result: $D(\Lambda) \subset W^{1,2}$).

As in the case $b \in \mathbf F_\delta$, it is reasonable to expect that there exists a 
quantitative
dependence between the value of $\delta$ and smoothness of the solutions to the equation $(\zeta + \Lambda_r)u = f, \; \zeta \in \rho(-\Lambda_r), \; f \in L^r.$ Such a dependence does exist. Set
\[
m_d:= \pi^\frac{1}{2}(2 e)^{-\frac{1}{2}} d^\frac{d}{2} (d-1)^{-\frac{d-1}{2}}, \quad \kappa_d := \frac{d}{d-1}, \quad r_\mp:= \frac{2}{1 \pm \sqrt{1-m_d \delta}}.
\]
It will be established that if $b \in \mathbf F_\delta^{\scriptscriptstyle 1/2}$ and $m_d \delta < 1,$ then $( e^{-t \Lambda_r(b)}, \; r \in [2,\infty[ )$ extends by continuity to a quasi bounded $C_0$ semigroup in $L^r$ for all $r \in ]r_-, \infty [.$ For every $r \in I_s := ]r_-, r_+[,$ the semigroup is holomorphic, the resolvent set $\rho(-\Lambda_r(b))$ contains the half-plane $\mathcal O := \{ \zeta \in \mathbb C \mid \Real \zeta > \kappa_d \lambda_\delta \},$ and the resolvent admits the representation
 \[
 (\zeta + \Lambda_r(b))^{-1} = (\zeta - \Delta)^{-1} - Q_r (1 + T_r)^{-1} G_r, \quad \zeta \in \mathcal O, \tag{$\star\star$}
 \]
 where $Q_r, G_r, T_r$ are bounded linear operators on $L^r$;
 $D(\Lambda_r(b)) \subset W^{1+\frac{1}{q},r} \;(q> r).$
 
 In particular, for $m_d \delta < 4 \frac{d-2}{(d-1)^2},$ there exists $r \in I_s, \; r > d - 1,$ such that $(\zeta + \Lambda_r(b))^{-1} L^r \subset C^{0, \gamma}, \gamma < 1 - \frac{d-1}{r}.$ 

The results above yield the following: Let $b \in \mathbf F_\delta^{\scriptscriptstyle 1/2}$ for some $\delta$ such that $m_d \delta < 4 \frac{d-2}{(d-1)^2}.$  Let $\{b_n\}$ be any sequence of bounded smooth vector fields, such that $b_n \rightarrow b$ strongly in $L^1_\loc,$ and, for a given $\varepsilon > 0$ and some $\delta_1 \in ]\delta, \delta + \varepsilon],$ $\{b_n\} \subset \mathbf F_{\delta_1}^{\scriptscriptstyle 1/2}$ with $\lambda \neq \lambda(n)$. Then
\[
 s \mbox{-} C_\infty \mbox{-} \lim_{n\uparrow \infty} e^{-t\Lambda_{C_\infty}(b_n)}  \tag{$\star\star\star$}
\]
exists uniformly in $t\in [0,1],$ and hence determines a $C_0$ semigroup $e^{-t\Lambda_{C_\infty}(b)}.$

The results $(\star\star)$, $(\star\star\star)$ can be obtained via direct investigation in $L^r$ of the operator-valued function $\Theta_r(\zeta,b)$ defined by the right hand side of $(\star\star)$ without appealing to $L^2$ theory (but again appealing to the ideas of E.\,Hille and H.\,F.\,Trotter) \cite{Ki2}.

\smallskip

There is an extensive literature on regularity of solutions to elliptic and parabolic equations having unbounded coefficients that are smooth outside of a discrete set, see, in particular, \cite{BGRT, CEF, FoL, LR, MPPS, MPSR, Met, MST, R} and references therein. In this work we adhere to the principle that the regularity properties of solutions should depend on the integral characteristics of the coefficients (here, on the relative bound $\delta$). Thus, as a by-product, we allow coefficients to be discontinuous (unbounded) e.g.\,on a dense set.

\smallskip

In the next parts of this work we will extend our regularity results 
in the spaces $L^p$ and $C_\infty$
to the operator
$A + b \cdot \nabla$, and with time-dependent coefficients $a$, $b$.

\smallskip

\noindent\textbf{Acknowledgments.} We are grateful to the anonymous referee for a number of valuable comments that helped to
improve the paper.

\smallskip

\setcounter{tocdepth}{2}
\tableofcontents

\section{Markov generators associated with $-\nabla \cdot a \cdot \nabla$}

\label{markov_gen_sect}

Throughout the paper we denote by $\mathcal B(X,Y)$ the space of bounded linear operators between complex Banach spaces $X \rightarrow Y$, endowed with the operator norm $\|\cdot\|_{X \rightarrow Y}$;  $\mathcal B(X):=\mathcal B(X,X)$. Set $\|\cdot\|_{p \rightarrow q}:=\|\cdot\|_{L^p \rightarrow L^q}$. 

\medskip

\noindent\textbf{1.~}Let $X$ be a set and $\mu$ a measure on $X.$ 
Fix $p\in [1,\infty[.$ A $C_0$ semigroup $T^t$, $t \geq 0$, of quasi contractions
on $L^p=L^p(X,\mu)$ (i.e.~$\|T^tf\|_p \leq e^{\omega_p t}\|f\|_p$, $f \in L^p$) is called Markov if, for each $t > 0$,
\begin{equation}\label{eqn:i}
	T^t L^p_+ \subset L^p_+, \tag{i}
\end{equation}
\begin{equation}\label{eqn:ii}
	(f \in L^p, |f| \leq 1) \Rightarrow |T^t f| \leq 1. \tag{ii}
\end{equation}

With each Markov semigroup $T^t$ we associate a collection $\{T_r^t\}_{p \leq r < \infty}$ of consistent quasi contraction $C_0$ semigroups on the scale $[p,\infty[$ of $L^r$ spaces as follows.

Since $\|T^t\|_{p \rightarrow p} \leq e^{\omega_p t}$ and $\|T^t f\|_\infty \leq \|f\|_\infty$ ($f \in L^p \cap L^\infty$), by the Riesz Interpolation Theorem,
$\|T^t f\|_r \leq e^{\frac{p}{r}\omega_p t}\|f\|_r$ $(f \in L^p \cap L^\infty)$  for all $r \in [p, \infty]$. Since $L^p \cap L^\infty$ is a dense subspace of $L^r$ for each $r \in [p, \infty[$,
$T^t\upharpoonright L^p \cap L^\infty : L^r \rightarrow L^r$ extends by continuity to a semigroup in $L^r$:
\[
T^t_r := (T^t \upharpoonright L^p \cap L^\infty)^{\text{\rm clos}}_{L^r \rightarrow L^r}.
\]
Next, by H\"older's inequality:
\[
	\|(T^t_r - 1) f\|_r \leq 2 \|f\|^{1 - \frac{p}{r}}_\infty \|(T^t - 1) f\|^{\frac{p}{r}}_p \;\;\;
	(f \in L^p \cap L^\infty),
\]
and hence we see that $T^t_r$ is strongly continuous. 

\smallskip

Set $T^t_p := T^t$. 
Then $T^t_r$ and $T^t_q$, $r, q \in [p, \infty[$ are \textit{consistent}: 
$$T^t_r \upharpoonright L^r \cap L^q = T^t_q \upharpoonright L^r \cap L^q.$$
Let $-A_r$ denote the generator of $T^t_r$ and $A:=A_p$. Then 
\begin{equation*}
	A_r \upharpoonright \mathcal{D}(A_r) \cap \mathcal{D}(A_q)
	= A_q \upharpoonright \mathcal{D}(A_r) \cap \mathcal{D}(A_q) \;\;\; (p \leq r, q < \infty).
\end{equation*}
because $\mathcal{E} := (1+A)^{-1} [L^1 \cap L^\infty]$ is a core of $A_r$ for all $r \in [p, \infty[ .$

\begin{definition*}Let $A \geq 0$ be a self-adjoint operator in $L^2(X, \mu)$. The semigroup $T^t = e^{-tA}$, $t \geq 0$, is called symmetric Markov if, for all $t > 0$,
\[
	T^t L^2_+ \subset L^2_+, \;\text{ and } 
	(f \in L^2, |f| \leq 1) \Rightarrow |T^t f| \leq 1.
\]
\end{definition*}

With each symmetric Markov semigroup $T^t$ we associate the collection $\{T^t_r\}_{1 \leq r \leq \infty}$ of contraction semigroups defined by
\begin{eqnarray*}
	T^t_r & := & (T^t \upharpoonright L^1 \cap L^\infty)^{\text{\rm clos}}_{L^r \rightarrow L^r} \;\;\; (r \in [1, \infty[), \\
	T^t_\infty & := & (T^t_1)^*.
\end{eqnarray*}
It is easily seen that the semigroup $T^t_r$ is strongly continuous for each $r \in ]1, \infty[$. (It is also known that $T^t_1$ is strongly continuous as well, see e.g.\,\cite[Prop.\,1.8]{LS}.)  We have $A_r \upharpoonright \mathcal{D}(A_r) \cap \mathcal{D}(A_q)
	= A_q \upharpoonright \mathcal{D}(A_r) \cap \mathcal{D}(A_q)$ ($1 \leq r, q < \infty$).

\smallskip

\textbf{2.~}Let $\Omega$ be an open set in $\mathbb{R}^d, \; d \geq 3,$
\begin{align*}
& a = a^* : \Omega \rightarrow \mathbb{R}^d \otimes \mathbb{R}^d , \quad a \in \big[ L^1_\loc (\Omega, \mathcal{L}^d) \big]^{d\times d},\\
& \sigma I \leq a(x) \text{ for some constant } \sigma > 0 \text{ and }\mathcal{L}^d \text{ a.e. }x \in \Omega.
\tag{$H_1$}
\end{align*}

One can define at least three realizations of the differential expression $-\nabla \cdot a \cdot \nabla$ in $L^2=L^2(\Omega,\mathcal{L}^d) $ as (minus) symmetric Markov generators: $A_D , A_{iD} , A_N ,$ the Dirichlet, intermediate Dirichlet, and generalized Neumann \cite{Se_RN}:

Let $\mathcal{T}$ denote a collection of all closed, symmetric non-negative quadratic forms in $L^2.$
Define the pre-Dirichlet form $\varepsilon$ by setting
\[
\varepsilon [u,v]:= \langle \nabla u \cdot a \cdot \nabla \bar{v} \rangle, \quad D(\varepsilon)=C^1_c(\Omega) \times C^1_c(\Omega), 
\]
where
$$
\langle f\rangle:=\int_\Omega f d\mathcal L^d, \quad \langle f, g\rangle := \langle f\bar{g}\rangle,
$$
and
\[
\nabla u \cdot a \cdot \nabla v (x) := \sum_{i,k=1}^d a_{ik}(x) \partial_{x_i} u(x)\partial_{x_k} v(x).
\]

Let
\begin{align*}
 \mathcal{T}_M := \{\tau \in \mathcal{T} \mid \tau \leftrightarrow T, \; e^{-sT}, s>0, \text{ is a symmetric Markov semigroup} \}
\end{align*}
($\tau \leftrightarrow T$ denotes the one to one correspondence (Appendix \ref{MCT_sect})).

Let
\begin{align*}
 \mathcal{T}_M (\varepsilon) := \{\tau \in \mathcal{T}_M \mid \tau \supset \varepsilon \}.
\end{align*}
It is said that $\tau \in \mathcal{T}_M$ is local if $\tau[f,g]=0$ whenever $(f,g) \in D(\tau), f, g \geq 0,$ and $f \wedge g =0.$

Consider the following extensions of $\varepsilon$ :
\begin{align*}
\tau_D & := \varepsilon^{\rm clos} \;\; (\text{the closure of } \varepsilon),\\
\tau_{iD} & \supset \tau_D , \quad D(\tau_{iD}):=\{u \in W^{1,2}_0 (\Omega) \mid \langle \nabla u \cdot a \cdot \nabla \bar{u} \rangle < \infty \}, \\
\tau_N & \supset \tau_{iD} , \quad D(\tau_N) :=\{u \in W^{1,2} (\Omega) \mid \langle \nabla u \cdot a \cdot \nabla \bar{u} \rangle < \infty \} .
\end{align*}

\begin{lemma}
$\tau_D , \tau_{iD} , \tau_N \in \mathcal{T}_M (\varepsilon)$, and are local.
\end{lemma}
\begin{proof}
Define
\[
a^n (\cdot) := \sigma I + (a(\cdot)- \sigma I) \bigg( I+  \frac{1}{n} a(\cdot) \bigg)^{-1} , \quad n=1,2,\dots 
\]
Clearly, $\sigma I \leq a^n(\cdot) \leq (n+\sigma) I$ and $a^n(\cdot) \leq a^{n+1}(\cdot) \leq a(\cdot) $ $\mathcal{L}^d$ a.e.

Let $E=W^{1,2}_0(\Omega)$ or $W^{1,2}(\Omega).$ Define
\begin{align*}
& \tau^n[u,v]:= \langle \nabla u \cdot a^n \cdot \nabla \bar{v} \rangle, \quad D(\tau^n) = E,\\
& \varepsilon^n[u,v]:= \langle \nabla u \cdot a^n \cdot \nabla \bar{v} \rangle, \quad D(\varepsilon^n)= D(\varepsilon).
\end{align*}
It is well known that $\tau^n \in \mathcal{T}_M (\varepsilon^n)$, and are local (by verifying the Beurling-Deny conditions, using properties of Sobolev's spaces, such as ``$C^\infty(\Omega) \cap W^{1,2}(\Omega)$ is dense in $W^{1,2}(\Omega)$", and Fukushima's test function $f_\epsilon = \epsilon \phi (\epsilon^{-1} f),$ $ \phi$ is a $C^\infty$ function from $\mathbb{R}$ to $[0,\infty[,$ $\phi(0)=0,$ $|\phi^\prime(x)|\leq 1$ and $\phi(x)=|x|-1$ if $|x| \geq 2$;  otherwise, see \cite[Sect.\,1.4]{FOT} or \cite[Theorems 1.3.5, 1.3.9]{Da1}).

Define $\tau$ by
\begin{align*}
\tau[u,v] & := \lim_n \tau^n [u,v], \\
D(\tau) & := \{u \in E \mid \sup_n \langle \nabla u \cdot a^n \cdot \nabla \bar{u} \rangle < \infty \}\\
& = \{u \in E \mid \langle \nabla u \cdot a \cdot \nabla \bar{u} \rangle < \infty \}.
\end{align*}
By the Monotone Convergence Theorem for a non-decreasing sequence of closed symmetric non-negative quadratic forms (Appendix \ref{MCT_sect}), $\tau \in \mathcal{T}$ and, obviously,  $\tau \in \mathcal{T}_M(\varepsilon)$ and is local.  Since $\tau \supset \varepsilon,$ $\varepsilon$ is closable. Thus $\tau_D$ is well defined. Now it is clear that $\tau_D$ belongs to $\mathcal{T}_M(\varepsilon)$ and is local.
\end{proof}
\begin{remarks} 1. The operators $A_D, \;A_{iD}, \; A_N$ associated with $\tau_D, \; \tau_i, \; \tau_N$, respectively, possess some nice properties. $C^\infty_c(\Omega)$ is a form core for $A_D.$ If $A_2(a)$ stands for $A_{iD}$  or $A_N,$ then $e^{-s A_p(a)} = s\mbox{-}\lim_n e^{-s A_p(a^n)}, p \in ]1, \infty [.$ At least $e^{-s A_D}$ and $e^{-s A_{iD}}$ have nice embedding properties $(L^p \rightarrow L^q).$

2. $\tau_D$ is the maximal element of $\mathcal{T}_M(\varepsilon)$ endowed with the semi-order
$\tau_1 \prec \tau_2 \Leftrightarrow D(\tau_1) \supset D(\tau_2)$, $\tau_1[u] \leq \tau_2[u], u\in D(\tau_2)$.
\end{remarks}

\bigskip

 \section{$b \cdot \nabla$ is $A_r$ bounded}

\label{strong_bdd_sect}

Let $\Omega$ be an open set in $\mathbb{R}^d, \; d \geq 3,$ and $L^p=L^p(\Omega,\mathcal{L}^d).$ Let $a:\Omega \rightarrow  \mathbb{R}^d \otimes \mathbb{R}^d$ be a symmetric $\mathcal{L}^d$ measurable strictly positive ($a \in (H_1)$) matrix.
Let $b: \Omega \rightarrow \mathbb{C}^d$ be $\mathcal{L}^d$ measurable and define $B_r = B_r(b)$ to be the maximal operator of $b \cdot \nabla$ in $L^r, \;\;1 \leq r < \infty,$ of domain
\[
D(B_r) := \{ h \in L^r \mid |\nabla h | \in L^1_\loc \; \text{ and } \; b \cdot \nabla h \in L^r \}.
\] 
Throughout this section $A \equiv A(a)$ denotes $A_D$, $A_{iD}$, $A_{N}$. In case of $A_N$ we also assume that $\partial \Omega \in C^{0,1}.$

\begin{definition*}
We say that $b \cdot \nabla$ is $A_r$ bounded if $\|B_r ( \lambda + A_r )^{-1} \|_{r \rightarrow  r}<\infty$ for some $\lambda>0$.
\end{definition*}

\noindent\textbf{1.~}Set $j := \frac{d}{d-2}$ and $j^\prime = \frac{d}{2}.$

\begin{proposition}
\label{thm:markwest}
Assume that $b^2_a := b \cdot a^{-1} \cdot \bar{b} \leq W + C$ $\mathcal L^d$ a.e.\;for a function $W \in L^{j^\prime}$ and a non-negative constant $C$. Then, for every $r \in ] 1, \frac{2d}{d+2} ],$ the operator $B_r$ is $A_r$ bounded and the inequality
\[
\|B_r ( \lambda + A_r )^{-1} \|_{r \rightarrow  r} \leq 4 \bigg ( \frac{C}{(r-1) \lambda} \bigg )^{1/2} + 4 c(r,j) \|W \|_{j^\prime}^{1/2} \quad \quad ( \lambda > 0 )
\]
holds with the constant $c(r,j) \rightarrow  \infty$ as $r \rightarrow  1.$ If $b^2_a \in L^\infty,$ then the assertion above is valid for all $1<r \leq 2.$ Moreover, for a given $\eta >0,$ we can choose $W_1$, $C_1$ and then $\lambda > 0$ such that $b_a^2 \leq W_1+C_1$ and
\[
4 \bigg( \frac{C_1}{(r-1) \lambda} \bigg )^{1/2} + 4 c(r,j) \|W_1 \|_{j^\prime}^{1/2} < \eta.
\]
\end{proposition}
\begin{proof}
Set $u = ( \lambda + A_r )^{-1} h , \; h \in L^1 \cap L_+^\infty.$ Let $1 < p \leq 2.$  Since $u \in D(A_p) \Rightarrow u_p := u^\frac{p}{2} \in D(A^{1/2})$ (Appendix \ref{markov_sect}, Theorem \ref{thm:markovest}), we conclude that $\nabla u_p \in L^2.$ (We use that $D(A^\frac{1}{2}) \subset E$ where $E=W_0^{1,2}$ if $A=A_D \text{ or } A_{iD},$ and $E=W^{1,2}$ if $A=A_N$.) Since $2 \geq p$ and $u \in L^\infty,$ 
\[
\nabla u = \nabla u_p^{2/p} =(2/p)u_p^\frac{2-p}{p} \nabla u_p = (2/p)u^\frac{2-p}{2} \nabla u_p.
\]
Now, $|b\cdot \nabla u|^2 \leq b^2_a \nabla u \cdot a \cdot \nabla \bar{u} = (2/p)^2 b^2_a u^{2-p} \nabla u_p \cdot a \cdot \nabla \bar{u}_p,$  and so by H\"older's inequality,
\begin{align*}
\|b \cdot \nabla u \|_r^2 & \leq (2/p)^2 \|b^2_a u^{2-p} \|_\frac{r}{2-r} \big \langle A^{1/2} u_p, A^{1/2} u_p \big \rangle \\
& \leq (p-1)^{-1} \|b^2_a u^{2-p} \|_\frac{r}{2-r} \big \langle A_p u, u^{p-1} \big \rangle. 
\end{align*}
If $b^2_a \leq C,$ then (take $p=r$) $\|b \cdot \nabla u \|_r^2 \leq \frac{C}{r-1} \|u\|_r^{2-r} \big \langle A_r u, u^{r-1} \big \rangle,$ and since $\|u\|_r \leq \lambda^{-1} \|h\|_r$ and $\langle A_r u, u^{r-1} \rangle \;(= \langle (\lambda + A_r)u,u^{r-1}\rangle - \lambda \langle u^r \rangle \leq \langle h, u^{r-1}\rangle )  \leq  \|h\|_r \|u\|_r^{r-1},$ 
\[
\|b \cdot \nabla u \|_r \leq \sqrt{\frac{C}{(r-1)\lambda}} \; \|h \|_r.
\tag{$\bullet$}
\]

If $b^2_a \leq W$, then the same argument (with $r<p \leq 2$,  $(p-1)^{-1} \geq r'\left(r^{-1}-j'^{-1}\right)$)
yields
\[
\|b \cdot \nabla u \|_r^2 \leq (p-1)^{-1} \|h\|_r \|u\|^{p-1}_{r^\prime (p-1)} \|W u^{2-p} \|_\frac{r}{2-r}.
\]
Therefore, by H\"older's inequality and the embedding $(\lambda + A_r)^{-1}: L^r \rightarrow  L^q , \; r^{-1} -q^{-1} \leq  j^{\prime -1},$ $\|b \cdot \nabla u \|_r \leq c(r,j) \|W \|_{j^\prime}^{1/2} \|h\|_r$, $h \in L^1 \cap L_+^\infty.$ It remains to pass in this inequality and in $(\bullet)$ to an arbitrary $h \in L_+^r.$ Using the facts that the weak gradient is closed in $L^r,$ and $L^1 \cap L^\infty$ is dense in $L^r,$ we have $\|\mathbf{1}_{\{|b|\leq n\}}b \cdot \nabla u \|_r \leq c(r,j) \|W \|_{j^\prime}^{1/2} \|h\|_r$ for all $h \in L_+^r.$ The use of Fatou's Lemma now completes the proof.
\end{proof}

\begin{remarks} 1. The implication $h \in D(A_r) \Rightarrow \nabla h \in \big[ L^r \big]^d$ for $r \in ]1, 2]$ follows now e.g. from a priori estimate $\| \nabla (1+A)^{-1} h \|_r \leq C(r) \|h \|_r, \; h \in L^1 \cap L^\infty,$ which in turn is a simple corollary of ($\bullet$). We also mention that, in general, even the condition $b^2_a \in L^\infty$ allows $b$ to be unbounded.

2. Let us recall a simple corollary of the Hille Perturbation Theorem (see e.g.\,\cite[Ch.\,IX, sect.\,2.2]{Ka}):

\textit{Let $e^{-t A}$ be a symmetric Markov semigroup, $K$ a linear operator in $L^r$ for some $r \in ]1,\infty[.$  If for some $\lambda > 0$ $\|K (\lambda + A_r)^{-1} \|_{r \rightarrow  r} < \frac{1}{2},$ then $-\Lambda_r := -A_r - K$ of domain $D(A_r)$ is the generator of a quasi bounded holomorphic semigroup on $L^r$.}

Thus, Proposition \ref{thm:markwest} implies that the (minus) algebraic sum $A_r + B_r$ is the generator of a holomorphic semigroup in $L^r$ for every $r \in ]1, \frac{2d}{d+2}].$

Sometimes one can employ the Miyadera Perturbation Theorem \cite{Mi} (see also \cite{Vo}):
\begin{theorem}
Let $e^{-t A}$ be a symmetric Markov semigroup, $K$ a linear operator in $L^r$ for some $r \in [1,\infty[.$ 
If for some $\lambda > 0$ $\|K (\lambda + A_r)^{-1} \|_{r \rightarrow  r} < \infty,$
and there exist $s>0$ and $\beta<1$ such that
$$
\int_0^s \|Ke^{-tA_r}\|_{r \rightarrow r}dt \leq \beta,
$$ 
then $-\Lambda_r := -A_r - K$ of domain $D(A_r)$ is the generator of a quasi bounded $C_0$ semigroup on $L^r$.
\end{theorem}

3.~Even in the case $a =I$, the assumption $W \in L^{j^\prime}$ cannot be weakened to $W \in L^p$ for some $p<j^\prime$. Also, this assumption does not guarantee that $B_1$ is $A_1$ bounded. Of course, if $a=I$ and $|b|\in L^d$, then $B_r$ is $A_r$ bounded for each $r\in ]1,d[$. 

If $a\in (H_u)$ and $|b|\in L^d$ ($\Omega=\mathbb R^d$), then $B_r$ is $A_r$ bounded for each $r\in ]1,2+\varepsilon]$, where $0<\varepsilon<d-2$ depends on the ellipticity constants $(\sigma,\xi)$. Indeed, for every  $r \in [\frac{2d}{d+2},2+\varepsilon]$ $$\|B_r (\mu+A_r)^{-1}f \|_r \leq \|b\|_{d} \|\nabla (\mu+A)^{-1}f\|_{s}, \quad s:=\frac{dr}{d-r}, \qquad (f \in C_c)$$ where
$
\|\nabla (\mu+A)^{-1}f\|_{s} \leq \|\nabla (\mu+A)^{-1}(1-\Delta)^{\frac{1}{2}}\|_{s \rightarrow s}\|(1-\Delta)^{-\frac{1}{2}}f\|_{s}.
$
By the N.\,Meyers Embedding Theorem, $\|\nabla (\mu+A)^{-1}(1-\Delta)^{\frac{1}{2}}\|_{s \rightarrow s}<\infty$ provided that $\xi/\sigma$ is sufficiently close to $1$ (Appendix \ref{meyers_sect}), and, clearly, $\|(1-\Delta)^{-\frac{1}{2}}f\|_{s} \leq C_S\|f\|_r$, which now yields the required.

\end{remarks}

\noindent\textbf{2.~}Now we specify the results of previous subsection for $b: \Omega \rightarrow  \mathbb R^d.$
Denote $T^t_r = e^{-t(\lambda + \Lambda_r(b))},$ where $\Lambda_r(b)= A_r + B_r$ of domain $D(\Lambda_r(b)) = D(A_r).$ Since $ b$ is real valued, $T^t_r \Real L^r \subset \Real L^r.$ We claim that $T^t_r$ is a positivity preserving $L^\infty$ contraction semigroup. We prove this here \textit{only for }$b_a \in L^\infty$ by verifying the criteria of R.\,\,Phillips and G.\,\,Stampacchia for $T^t_2$ (Appendix \ref{phillips_sect}).

Let $b_a^2 \leq C$, $C<4\lambda$.
Set $\Gamma = \lambda + \Lambda_2(b).$ It is seen that, for $ f \in D(A),$ $\Real \langle \Gamma f, f \rangle \geq 0,$ (Indeed, by the quadratic inequality, $\Real \langle \Gamma f,f\rangle \geq \lambda \|f\|_2^2 - \frac{1}{4} \langle b \cdot a^{-1} \cdot b, f^2 \rangle \geq \big(\lambda-\frac{C}{4}\big) \|f\|_2^2 \geq 0$ where at the last step we used $b_a^2 \leq C$). Thus, $T^t_2$ is a contraction. For any $f \in \Real L^2$ define $f_+ = f \vee 0.$ If $f = \Real f \in D(A),$ then
\[
\langle \Gamma f, f_+ \rangle = \lambda \langle  f_+^2 \rangle + \langle A f, f_+ \rangle + \langle b \cdot \nabla f, f_+ \rangle.
\]
Taking into account that $A$ is local and $f_+ \in D(A^\frac{1}{2}) \subset W^{1,2}_0,$ we obtain
\begin{align*}
\langle A f, f_+ \rangle  & = \langle \nabla f \cdot a \cdot \nabla f_+ \rangle = \langle \nabla f_+ \cdot a \cdot \nabla f_+ \rangle, \\
|\langle b \cdot \nabla f, f_+ \rangle | & \leq \langle b \cdot a^{-1} \cdot b f_+^2 \rangle^\frac{1}{2} \langle \nabla f_+ \cdot a \cdot \nabla f_+ \rangle^\frac{1}{2} \\
& \leq \frac{C}{4} \langle  f_+^2 \rangle + \langle \nabla f_+ \cdot a \cdot \nabla f_+ \rangle.
\end{align*}
Thus $\langle \Gamma f, f_+ \rangle \geq 0$ for all $f =\Real f \in D(A) = D(\Gamma).$ 

Next, let $f \in D(A).$ Then $f_\wedge := (1 \wedge |f|) \sgn f$ and $f-f_\wedge$ are from $D(A^\frac{1}{2}) $ (e.g.\,by the Beurling-Deny Theorem, see \cite[sect.\,1]{LS})
\[
\nabla (f-f_\wedge) = \mathbf 1_{|f|>1}\big[\big(1-\frac{1}{|f|} \big) \nabla f + \frac{f}{|f|^2} \nabla |f| \big] \text{ and } \Real \bar{f} \nabla f = |f| \nabla |f|.
\]
Setting $\psi = \mathbf 1_{|f|>1} (|f|-1)$ we have
\begin{align*}
\Real \langle \Gamma f, f-f_\wedge \rangle & = \lambda \langle |f|, \psi \rangle + \langle\nabla \psi \cdot a \cdot \nabla \psi \rangle + \langle b \cdot \nabla \psi, \psi \rangle\\
& \geq \big( \lambda - \frac{C}{4} \big) \langle |f|, \psi \rangle \geq 0.
\end{align*}
Thus $\Real \langle \Gamma f, f-f_\wedge \rangle \geq 0$ for all $f \in D(A)= D(\Gamma).$

\begin{corollary}
 $T^t_2$ is a Markov semigroup.
\end{corollary}

Now let $r \in ]1,2[.$ By the construction of $\Lambda_r$ and $\Lambda_2,$ $T^t_r f = T^t_2 f$ whenever $f \in L^r \cap L^2.$ Therefore, for each $r \in ]1, 2[,$  the semigroup $T^t_r$ preserves positivity and is $L^\infty$ contraction.

\bigskip

\section{$b\cdot \nabla$ is $A$ form-bounded}

\label{fbd_sect}

Throughout this section we are assuming that $\Omega$ is an open set of $\mathbb{R}^d, \; d \geq 3,$ $a:\Omega \rightarrow  \mathbb{R}^d \otimes \mathbb{R}^d,$ a symmetric $\mathcal{L}^d$ measurable strictly positive ($a \in (H_1)$) matrix. 

By $A$ we denote one of the $A_D, A_{iD}, A_N.$ In the results concerning $(L^p,L^q)$ estimates for $A_N$ we also assume that $\partial \Omega \in C^{0,1}.$

\begin{definition*}
We say that a $b:\mathbb R^d \rightarrow \mathbb{R}^d$ belongs to $\mathbf{F}_\delta(A)$, the class of $A$ form-bounded vector fields, and write $b \in \mathbf{F}_\delta(A)$, if $b_a:= \sqrt{b\cdot a^{-1} \cdot b} \in L^2_\loc$ and there exists a constant $0 \leq \lambda =\lambda_\delta<\infty$ such that 
$$\|b_a(\lambda + A)^{-1/2}\|_{2 \rightarrow 2} \leq \sqrt{\delta}.$$
(Equivalently, $b_a^2 \leq \delta A + c(\delta)$ in the sense of quadratic forms for some constant $c(\delta)$ $(=\lambda \delta)$.) 
 \end{definition*}

Write $\mathbf{F}_\delta \equiv \mathbf{F}_\delta(-\Delta)$. (Clearly, $b \in \mathbf{F}_{\sigma^2 \delta} \Rightarrow b \in \mathbf{F}_\delta(A)$.)

\begin{example} 
\label{ex0}
1. The inclusion 
$|b| \in L^d+L^\infty$ $\Rightarrow b \in \mathbf{F}_0:=\bigcap_{\delta>0}\mathbf{F}_\delta$
follows from the Sobolev Embedding Theorem.

2. For $|b| \in L^{d,\infty}$,
one can verify, using \cite[Prop.~2.5, 2.6, Cor.~2.9]{KPS}:
\begin{align*}
b \in \mathbf{F}_{\delta_1}, \quad \text{ with } \sqrt{\delta_1}&=\||b| (\lambda - \Delta)^{-\frac{1}{2}} \|_{2 \rightarrow 2} \\ & \leq
\|b\|_{d,\infty} \Omega_d^{-\frac{1}{d}} \||x|^{-1} (\lambda - \Delta)^{-\frac{1}{2}} \|_{2 \rightarrow 2} \\ & \leq \|b\|_{d,\infty} \Omega_d^{-\frac{1}{d}}2^{-1}\frac{\Gamma\bigl(\frac{d-2}{4} \bigr)}{\Gamma\bigl(\frac{d+2}{4} \bigr)}=\|b\|_{d,\infty} \Omega_d^{-\frac{1}{d}} \frac{2}{d-2}.
\end{align*}
where $\Omega_d=\pi^{\frac{d}{2}}\Gamma(\frac{d}{2}+1)$ is the volume of the unit ball in $\mathbb R^d$.

3. By Hardy's inequality, $b(x):=\sqrt{\delta} \frac{d-2}{2}|x|^{-2}x \in \mathbf{F}_{\delta}$, $\delta>0$. (And, of course, $b \not\in \mathbf{F}_{\delta_2}$ if $\delta_2<\delta$.)

4.~For every $\varepsilon>0$ one can find $b \in \mathbf{F}_\delta$ such that $|b| \not\in L^{2+\varepsilon}_{\loc}$, e.g.\,consider
$$
b^2(x)=C\frac{\mathbf{1}_{B(0,1+a)} - \mathbf{1}_{B(0,1-a)}}{\big| |x|-1\big|^{-1}(-\ln\big||x|-1\big|)^b}, \quad b>1, \quad 0<a<1.
$$

5.~Let $\mathbf{W}_s$ denote the class of vector fields $b$ such that $|b|^2$ is in the class of Chang-Wilson-Wolff ($s>1$):
$$
\left\{v \in L_{\loc}^s: \|v\|_{W_s}:=\sup_Q \frac{1}{|Q|}\int_Q |v(x)|\, l(Q)^2 \varphi\big(|v(x)|\,l(Q)^2 \big) dx<\infty\right\},
$$
where $|Q|$ and $l(Q)$ are the volume and the side length of a cube $Q$, respectively,
$\varphi:[0,\infty[ \rightarrow [1,\infty[$ is an increasing function such that
$
\int_1^\infty \frac{dx}{x\varphi(x)}<\infty.
$
By \cite{CWW}, if $b \in \mathbf{W}_s$, then $b \in \mathbf{F}_{\delta}$, $\delta=\delta\big(\|b^2\|_{W_s}\big)<\infty$.

The class $\mathbf{W}_s$ contains, in particular, the vector fields $b$ with $|b|^2$ in the Campanato-Morrey class ($s>1$)
$$
\left\{v \in L_{\loc}^s: \biggl(\frac{1}{|Q|}\int_Q |v(x)|^s dx \biggr)^{\frac{1}{s}} \leq c_s l(Q)^{-2} \text{ for all cubes $Q$}\right\}
$$
(write $b \in \mathbf{C}_s$).
(For the complete diagram of the spaces of vector fields $b:\mathbb R^d \rightarrow \mathbb R^d$ considered in this paper in connection with the operator $-\Delta + b \cdot \nabla$ see section \ref{weak_fbd_sect}.)
\end{example}

\begin{remark*}
If $K$ is a positivity preserving linear operator on $L^p$ or $C_\infty$ 
(e.g.~a Markov semigroup or the resolvent of its generator),
then $$|K f| \leqslant K|f|, \quad (f \in L^p \text{ or } C_\infty, \text{respectively}).$$
This well known fact will be extensively used below. (For a proof, if needed, see \cite[Prop.\,1.5]{LS}.)
\end{remark*}

\label{standard_sect}
\label{Lr_sect} 

If $b \in \mathbf{F}_\delta(A), \; 0<\delta < 1$, then the Hille-Lions theory (section \ref{unif_elliptic_2}) yields the following: In $L^2$, there exists an operator realization $\Lambda \equiv \Lambda(b)$ of the formal differential operator $A + b \cdot \nabla$   such that $D(\Lambda) \subset D(A^\frac{1}{2}), \; \rho(-\Lambda) \supset \mathcal{O} := \{\zeta \in \mathbb C \mid \Real \zeta > \lambda_o = \frac{\lambda\delta}{2(1+\sqrt{1-\delta})}\}.$ If $\zeta \in \mathcal O,$ then
 \[
 R_\zeta = R_{a,\zeta}^\frac{1}{2} (1 + T)^{-1} R_{a,\zeta}^\frac{1}{2}, \\
\tag{$i_0$} 
 \]
 where $R_\zeta \equiv R_\zeta(b) := (\zeta + \Lambda)^{-1}, \;\; R_{a,\zeta} :=(\zeta +A)^{-1}, \; T := R^\frac{1}{2}_{a,\zeta} b \cdot \nabla R^\frac{1}{2}_{a,\zeta}.$
 \[
 e^{-t\Lambda(b_n)} \overset{s}\rightarrow e^{-t\Lambda} \text{ as } n \uparrow \infty,\tag{$ii_0$}
 \]
 \[
(1 + A)^\frac{1}{2} \big(R_\zeta(b_n) - R_\zeta(b) \big) \overset{s} \rightarrow 0  \quad \text{ as } n \uparrow \infty, \tag{$iii_0$}
\]
 where $b_n:=\mathbf{1}_n b,$ $\mathbf{1}_n$ is the indicator of $\{x \in \Omega: b_a(x) \leq n\}$. 
\begin{remark*}
 \[
e^{-t(\lambda_o +\Lambda)} \text{ is a Markov semigroup. }
 \tag{$i$}
 \]
For all $2 \leq r < q \leq \infty$ and $t > 0,$ 
\[
\|e^{-t \Lambda_r} \|_{r \rightarrow q} \leq c \; e^{\frac{2}{r} \lambda_o t} t^{-\frac{d}{2} (\frac{1}{r}-\frac{1}{q})}. \tag{$ii$} 
\]
Indeed, $(i)$ follows from $(ii_0)$ and the results of section \ref{strong_bdd_sect}.2. Also, slightly modifying the arguments from section \ref{strong_bdd_sect}.2, it is not difficult to verify directly that $\lambda_o + \Lambda(b)$ obeys the conditions of R. Phillips and G. Stampacchia. In turn, $(i_0) + (i)$ entails $(ii)$ by modifying the arguments from the proof of Theorem \ref{thm:markweak} below. 
\end{remark*}

Now, we develop the semigroup theory of $-\nabla \cdot a \cdot \nabla + b \cdot \nabla$ in the $L^r \equiv L^r(\Omega, \mathcal{L}^d)$ spaces in the case $0 < \delta < 4.$ 

\begin{theorem}
\label{thm:markeast}
Assume that $b \in \mathbf{F}_\delta(A)$ for some $0 < \delta < 4.$ If $1 \leq \delta < 4$ also assume that $b_a \in L^2 + L^\infty.$ Set $r_\delta := \frac{2}{2-\sqrt{\delta}}.$ Then $\nabla \cdot a \cdot \nabla - b \cdot \nabla$ has an operator realization  $- \Lambda_r(b)$ in $L^r$ for any $r \in I_c := [r_\delta, \infty[$ as the generator of a positivity preserving, $L^\infty$ contraction, quasi contraction $C_0$ semigroup on $L^r$. In full:

Let $\mathbf 1_n$ denote the indicator of $\{x \in \Omega \mid  \; b_a(x) \leq n \}$ and set $b_n := \mathbf 1_n b.$ Then
\[
e^{-t \Lambda_r(b)} := s\mbox{-}L^r\mbox{-}\lim_{n\rightarrow \infty} e^{-t \Lambda_r(b_n)}  \quad (r \in I_c^o := ]r_\delta, \infty [ ); \tag{$\ast$}
\]
\[
e^{-t \Lambda_{r_\delta}(b)} := \bigg[e^{-t \Lambda_r(b)} \upharpoonright L^1 \cap L^r \bigg]^{clos}_{L^{r_\delta} \rightarrow L^{r_\delta} } \quad \quad (  r \in I_c^o ); \tag{$\ast\ast$}
\]
\[
\|e^{-t \Lambda_r(b)} \|_{r \rightarrow r} \leq e^{t \omega_r}, \quad \omega_r = \frac{\lambda \delta}{2(r-1)}. \quad (r \in I_c );
\]
There is a constant $c = c(\delta, d)$ such that the $(L^r, L^q)$ estimate
\[
\|e^{-t \Lambda_r(b)} \|_{r \rightarrow q} \leq c \; e^{t \omega_r} \; t^{-\frac{d}{2} (\frac{1}{r}-\frac{1}{q})} \tag{$\ast\ast\ast$} 
\]
is valid for all $r_\delta \leq r < q \leq \infty, \;t > 0.$

For each $r \in I_c^o, \; e^{-t\Lambda_r(b)}$ is a holomorphic semigroup of quasi contractions on the sector
\[
|\arg t | \leq \frac{\pi}{2}-\theta_r, \;0 < \theta_r < \frac{\pi}{2}, \; \tan \theta_r \leq \mathcal K(2- r^\prime \sqrt{\delta})^{-1}, \tag{$\ast\ast\ast\;\ast$}
\]
where $\mathcal K = \frac{|r-2|}{\sqrt{r-1}} +r^\prime \sqrt{\delta}$ if $r\leq 2r_\delta$ and $\mathcal K=\frac{r-2 +r\sqrt{\delta}}{\sqrt{r-1}}$ if  $ r>2r_\delta.$  
\end{theorem} 
  
\begin{remarks} 1. The additional assumption $b_a \in L^2 + L^\infty$ in case $1 \leq \delta <4 $, for a general  $a \in (H_1)$, seems to be close to optimal. For more regular $a$'s no such assumption is needed. See Theorem \ref{thm:markease} below and the remark after.

2. The interval $I_c = [\frac{2}{2-\sqrt{\delta}}, \infty[$ is called the interval of contraction solvability.

 The interval $I_m := ]\frac{2}{2-\frac{d-2}{d}\sqrt{\delta}}, \infty[$ is called the interval of bounded solvability. We will show below (Theorem \ref{thm:marksolve}) that if the matrix $a$ is uniformly elliptic, then, for any $r \in I_m - I_c,$ one can still define $e^{-t \Lambda_r}$ as a quasi bounded holomorphic semigroup. We also note that $I_m \ni d,$ whilst $I_c \ni d$ whenever $\sqrt{\delta} \leq 2\frac{d-1}{d}.$

3. The example of $\Lambda_r(b) \supset L=-\Delta +b\cdot\nabla, \; b(x)=c|x|^{-2}x, \; c=\frac{d-2}{2}\sqrt{\delta} \;(\delta <4)$ in $\mathbb R^d, d\geq 3,$ shows that both intervals $I_c$, $I_m$ are maximal. (Indeed, it is not difficult to see, appealing to Hardy's inequality, that $\Lambda_r(b)$ ceases to be quasi accretive for any $r \not\in I_c$. For the proof of the maximality of $I_m$, see remark after the proof of Theorem \ref{thm:marksolve}.)
 Any $c < d-2$ is admissible according to Theorem \ref{thm:markeast}. In turn, $c = d-2$ ($\Leftrightarrow$ $\delta=4$) makes $I_c =\varnothing,$ while $c = d$ makes $I_m =\varnothing$ even formally ($c=d \Rightarrow \delta = 4\big(\frac{d}{d-2}\big)^2 > 4).$
Note, however, that as $c \uparrow d-2$ the interval $I_m \downarrow \;\; ]\frac{d}{2},\infty[$.
 
In $\Omega =B(0,1):=\{x\in \mathbb R^d \mid |x|<1 \}$ consider the Dirichlet problem 
\[
\label{D}
L u = 0, \quad u= 0 \text{ on } \partial \Omega. \tag{$D$} 
\]
Obviously, $u_1 =0, \;u_2 = |x|^\alpha -1, \; \alpha = c - (d-2),$ are solutions of \eqref{D}; $u_2 \in L^r$ for all $0 < r < \frac{d}{- \alpha} = \frac{2}{2-\sqrt{\delta}}\frac{d}{d-2} \equiv r_\delta j.$ In particular, $u_2 \in L^r$ for all $ r_\delta < r < r_\delta j,$ and satisfies the maximum principle. If $u_2$ would be a solution to $\Lambda_r v =0$ (equivalently, $e^{-t\Lambda_r}v=v$) for some $r \in I_c^o,$ then $u_2$ should belong to $L^\infty.$ Thus $u_2 \notin D(\Lambda_r), r \in I_c.$

 If $\delta > 4,$ then $\alpha > 0,$ and so the problem $(D)$ has two bounded weak solutions. By the way, now $u_2$ does not satisfy the maximum principle. See also example in section \ref{nondiv_sect}.
 
 \textit{Thus, the assumption $\delta>4$ (even for $a=I$) destroys the uniqueness of (accordingly defined) weak solutions.}

We refer to \cite{FL} where the critical case $\delta=4$ in dimension $d=3$ is studied.  See also \cite{Met} concerning the case $\delta>4$. In \cite{FL} it is demonstrated that already for $\delta=4$ the properties of the constructed semigroup are drastically different from the properties of $e^{t\Delta}$ and $e^{t\Lambda_r}$, $\Lambda_r \supset -\Delta + b \cdot \nabla$, with $b \in \mathbf{F}_\delta(-\Delta)$, $\delta<4$. 
\end{remarks}

\begin{proof}
1. First of all, let us prove that for all $r \in I_c, \; n \geq 1, \; t \geq 0,$
\[
\|e^{-t \Lambda_r(b_n)} \|_{r \rightarrow r} \leq e^{\omega_r t}. \tag{$\star$}
\]
Set $u \equiv u_n := e^{-t \Lambda_r(b_n)} |h| = e^{-t \Lambda_{2 r}(b_n)} |h|, \; h \in L^1 \cap L^\infty.$ By the results of section \ref{strong_bdd_sect}.2, $u \in D(A) \cap L_+^\infty.$ 
If $r \geq 2,$ then by Theorem \ref{thm:markovest} (Appendix \ref{markov_sect}), $u, u^\frac{r}{2}, u^r \in D(A^\frac{1}{2}),$  and since $D(A^\frac{1}{2}) \subset E$ ($E=W_0^{1,2}$ if $A=A_D \text{ or } A_{iD},$ and $E=W^{1,2}$ if $A=A_N$), $\nabla u^\frac{r}{2} = \frac{r}{2} u^{\frac{r}{2}-1}  \nabla u$ because $r \geq 2$ and $u \in L^\infty$;  if $1 < r \leq2,$ then $u^{r-1} \nabla u = \frac{1}{r} \nabla u^r = \frac{2}{r} u^\frac{r}{2}\nabla u^\frac{r}{2}$ because $u \in D(A_r)$ and $u^{r-1}, u^\frac{r}{2} \in L^\infty.$ 

Noticing that $\Lambda_r(b_n)u=A_ru +b_n\cdot \nabla u$ if $1<r \leq 2,$ and $\Lambda_r(b_n)u=\Lambda_2(b_n)u =Au +b_n\cdot \nabla u,$ $Au \in L^2, u^{r-1} \in L^2, \; \langle Au,u^{r-1} \rangle \geq \frac{4}{r r^\prime}\|A^\frac{1}{2}u^\frac{r}{2}\|_2^2$ if $r \geq 2,$ we have
\[
-\big \langle \frac{d}{d t} u, u^{r-1} \big \rangle = \big \langle \Lambda_r(b_n) u, u^{r-1} \big \rangle = \big \langle A_r u, u^{r-1} \big \rangle + \big \langle b_n \cdot \nabla u, u^{r-1} \big \rangle, \;\; 1<r \leq 2,
\]
\[
-\big \langle \frac{d}{d t} u, u^{r-1} \big \rangle = \big \langle \Lambda_r(b_n) u, u^{r-1} \big \rangle = \big \langle A u, u^{r-1} \big \rangle + \big \langle b_n \cdot \nabla u, u^{r-1} \big \rangle, \;\; r \geq 2.
\]
Therefore,
\[
-\frac{d}{d t} \| u \|^r_r \geq \frac{4}{r^\prime} \big \langle A^\frac{1}{2} u^\frac{r}{2},A^\frac{1}{2} u^\frac{r}{2} \big \rangle + 2 \big \langle u^\frac{r}{2} b_n \cdot \nabla u^\frac{r}{2} \big \rangle .
\]
Using the conditions $r \in I_c, \; b \in \mathbf{F}_\delta(A)$ and completing the quadratic estimate
\begin{align*}
2 |\big \langle u^\frac{r}{2} b_n \cdot \nabla u^\frac{r}{2} \big \rangle | & \leq \varepsilon \|b_a u^\frac{r}{2} \|_2^2 + \varepsilon^{-1} \| A^\frac{1}{2} u^\frac{r}{2} \|^2_2 \\
& \leq (\varepsilon \delta + \varepsilon^{-1} )\| A^\frac{1}{2} u^\frac{r}{2} \|^2_2 +\varepsilon c(\delta) \|u\|_r^r,
 \end{align*}
we obtain (choosing $\varepsilon = \frac{r^\prime}{2}$ and taking into account that $\sqrt{\delta} \leq \frac{2}{r^\prime}$ for $r \in I_c$)
\[
2 |\big \langle u^\frac{r}{2} b_n \cdot \nabla u^\frac{r}{2} \big \rangle | \leq \frac{4}{r^\prime} \| A^\frac{1}{2} u^\frac{r}{2} \|^2_2 + \frac{c(\delta)r^\prime}{2} \|u\|_r^r.
\]
Thus $\frac{d}{d t} \|u\|_r^r \leq \frac{c(\delta)r^\prime}{2} \|u\|_r^r.$ The desired bound $(\star)$ follows.

Since the mapping $I_c \ni p \rightarrow \|f \|_p, \; f \in L^1 \cap L^\infty$ is continuous, $(\star)$ yields the bound
\[
\|e^{-t \Lambda_r(b_n)} h \|_\infty \leq \|h\|_\infty \quad \quad (h \in L^1\cap L^\infty ).
\]
Next, if $\delta <1,$ then $u_n \overset{s}\rightarrow u,$ and hence $u_n \rightarrow u$ strongly in $L^r, r \in I^o_c.$

2. To justify $(\ast)$ when $1 \leq \delta < 4,$ we use the direct method: setting $g = u_m - u_n,$ we will prove that, for each $r \in I_c^o, \; \|g \|_r \rightarrow 0$ as $n, m \uparrow \infty$ uniformly on $t\in [0,1].$ Obviously, the latter combined with $(\star)$ will yield $(\ast)$.

Since $-\frac{d}{d t} g = A g + b_m \cdot \nabla g + (b_m-b_n) \cdot \nabla u, \; u \equiv u_n,$  we have, multiplying both sides of the equation by $g |g|^{r-2},$ integrating over $\Omega,$ using the assumptions $b \in \mathbf{F}_\delta(A), \; b_a^2 = V_1 + V_\infty, \; V_1 \in L^1,$ and $ r > r_\delta,$
\begin{align*}
\frac{d}{d t} \|g\|_r^r & \leq \frac{c(\delta)}{\sqrt{\delta}} \|g\|^r_r + r \langle | \mathbf{1}_n -\mathbf{1}_m | b_a \sqrt{\nabla u \cdot a \cdot \nabla u} |g|^{r-1} \rangle \\
& \leq \frac{c(\delta)}{\sqrt{\delta}} \|g\|^r_r + r \langle |\mathbf{1}_n -\mathbf{1}_m | V_1 \rangle^\frac{1}{2} \|g\|_\infty^{r-1} \langle A u, u \rangle^\frac{1}{2}.
\tag{$\bullet$}
\end{align*}
To get a suitable bound on $\|g\|_r$ we need the following estimates
\[
\int_0^T e^{-c t} \langle A u, u \rangle d t \leq \|h\|_2^2 + T \|V_1 \|_1 \|h\|_\infty^2 \qquad (\;c := \|V_\infty \|_\infty).
\]
Evidently, $-\frac{d}{d t} \|u\|^2_2 = 2 \langle A u, u \rangle + 2 \langle u b_n \cdot \nabla u \rangle \geq \langle A u, u \rangle - \langle b_a^2 u^2 \rangle.$ Setting $\psi(t)= e^{-ct} \|u\|_2^2,$ and using the bound $\|u\|_\infty \leq \|h\|_\infty,$ we have $e^{-ct} \langle A u, u \rangle \leq - \frac{d}{d t}\psi(t) + e^{-c t} \|V_1 \|_1 \|h\|_\infty^2,$ and so  $\int_0^T e^{-c t} \langle A u, u \rangle d t \leq \psi(0) + \frac{1-e^{-cT}}{c}\|V_1 \|_1 \|h\|_\infty^2 \leq \|h\|_2^2 + T\|V_1 \|_1 \|h\|_\infty^2$. The required estimate is proved.
 
 Now using this estimate, the bound $\|g\|_\infty \leq 2 \|h\|_\infty$ and the equality $g |_{t=0}=0$ one can easily integrate $(\bullet)$, obtaining $\sup_{t \in [0, T]} \|g(t) \|_r^{2 r} \leq C \|(\mathbf{1}_n -\mathbf{1}_m )V_1\|_1,$ where the constant $C$ depends on $T, r, \delta, \|h\|_2, \|h\|_\infty,\|V_1\|_1, \|V_\infty \|_\infty$ only. The latter makes $\|g\|_r \rightarrow 0$ as $n, m \uparrow \infty$ uniformly in $t \in [0, T].$ Thus the $s\mbox{-} L^r\mbox{-} \lim_n e^{-t \Lambda_r(b_n)}$ does define a $C_0$ semigroup. We have proved $(\ast)$. Now it should be clear that
\[
\|e^{-t \Lambda_r(b)} f \|_{r_\delta} \leq e^{t \omega_{r_\delta}} \|f\|_{r_\delta}\quad \quad (r \in I^0_c, \; f \in L^1 \cap L^\infty )\\ 
\tag{$\star^\prime$}
\]
and $e^{-t \Lambda_{r_\delta}(b)}$ defined by $(\ast\ast)$ is indeed a semigroup.
The fact that it is strongly continuous follows from $(\star^\prime)$ and $ \|f - e^{-t \Lambda_r(b)} f \|_r \rightarrow 0$ as $t \downarrow 0$ for any $r \in I_c^o$ and any $f\in L^1 \cap L^\infty$ by employing Fatou's Lemma. 

3. A proof of $(\ast\ast\ast)$ presented below is based on the embedding property of $A^\frac{1}{2}$ and $(\star),$ the $L^r$ quasi contraction property of $e^{- t \Lambda_r(b_n)}.$ In view of $(\ast)$, it suffices to treat the case $\Lambda(b_n).$ We have for $u := e^{-t \Lambda(b_n)} f, \; f \in L^1_+ \cap L^\infty, \; r = 2 r_\delta,$
\[
-\frac{d}{d t} \|u\|^r_r \geq \bigg( \frac{4}{r^\prime} - 2 \sqrt{\delta} \bigg) \|A^\frac{1}{2} u^\frac{r}{2}\|_2^2 - \frac{c(\delta)}{\sqrt{\delta}} \|u\|_r^r.
\]
Note that $\frac{2}{r^\prime} > \sqrt{\delta}.$ Using the Nash inequality $\|A^\frac{1}{2} h \|_2^2 \geq C_N \|h\|_2^{2 + \frac{4}{d}} \|h\|_1^{-\frac{4}{d}}$ and $(\star^\prime),$ we have setting $w := \|u\|_r^r,$ 
\[
\frac{d}{dt} w^{-\frac{2}{d}} \geq - c_2 w^{-\frac{2}{d}} +c_1 e^{-c_3 t} \|f\|_{r_\delta}^{-\frac{2r}{d}},
\]
where $c_1 = C_N \frac{2}{d} \big( \frac{4}{r^\prime} - 2 \sqrt{\delta} \big), \; c_2 = \frac{2}{d}\frac{c(\delta)}{\sqrt{\delta}}, \; c_3 = \frac{2 }{d} \omega_{r_\delta} 2r_\delta.$

Integrating this inequality yields
\[
\|e^{-t \Lambda_{r_\delta}(b_n)} \|_{r_\delta \rightarrow 2 r_\delta} \leq c_1^{-\frac{d}{4 r_\delta}} e^{\omega_{r_\delta}t} t^{-\frac{d}{2} (\frac{1}{r_\delta} - \frac{1}{2 r_\delta})}, \quad t > 0. \tag{$\star\star$}
\]
By duality, $\|e^{-t \Lambda^*_{r_\delta}(b_n)} \|_{(2r_\delta)^\prime \rightarrow r_\delta^\prime} \leq$ the RHS of $(\star\star)$, and 
$$\|e^{-t \Lambda^*_{r_\delta}(b_n)} \|_{1 \rightarrow r_\delta^\prime} = \|e^{-t \Lambda_{r_\delta}(b_n)} \|_{r_\delta \rightarrow \infty} \leq C_1(d, r_\delta) e^{\omega_{r_\delta}t}t^{-\frac{d}{2} \frac{1}{r_\delta}}$$
by the Extrapolation Theorem (Appendix \ref{appendix_B}).  Interpolating the latter, we arrive at the desired  estimate.

4.~Proof of $(\ast\ast\ast\;\ast).$ Set $z :=2\lambda \sqrt{\delta}.$  We need to prove that
 \[
 |\Imag\, \langle \Lambda_r(b)u, u|u|^{r-2} \rangle|\leq \mathcal K (2-\sqrt{\delta}r^\prime)^{-1}\Real\, \langle(z+\Lambda_r(b))u, u|u|^{r-2}\rangle, \qquad u \in D(\Lambda_r(b)), \; r \in I_c^o. \tag{$\bullet$}
 \]
 a) First we prove $(\bullet)$ for $b=b_n$ and $u=u_n \in D(\Lambda_r(b_n)),$ where $b_n := b \mathbf 1_n, \mathbf 1_n$ is the indicator of $\{x\in \Omega \mid b_a(x)\leq n \}.$ Note that $D(A_r)\cap L^\infty$ is a core of $\Lambda_r(b_n)$ for $r \in \{r\leq 2\}\cap I_c^o,$ and $D(A)\cap L^\infty$ is a core of $\Lambda(b_n)$ for $r \in \{r\geq 2\}\cap I_c^o.$ Thus, we will take $u_n$ from these cores. Then $\Lambda_r(b_n)u_n = \tilde A u_n + b_n\cdot \nabla u_n,$ where $\tilde A v = A_r v$ if $v \in D(A_r)\cap L^\infty$ and $\tilde A v = A v$ if $v \in D(A)\cap L^\infty.$ 
 
Using the assumption $b_n \in \mathbf F_\delta(A),$ inequality $\|A^\frac{1}{2} |u|^\frac{r}{2} \|_2 \leq \|A^\frac{1}{2} u|u|^{\frac{r}{2}-1} \|_2,$ equality $X\equiv \langle \nabla u \cdot a \cdot \nabla \bar u, |u|^{r-2} \rangle = \Real\, \langle \tilde Au, u|u|^{r-2} \rangle -\frac{4(r-2)}{r^2}\|A^\frac{1}{2}|u|^\frac{r}{2} \|_2^2,$ and completing quadratic estimates we obtain 
\begin{align*}
|\Real\, \langle b_n\cdot \nabla u, u|u|^{r-2}\rangle| & = |\langle b_n\cdot \nabla |u|, |u|^{r-1}\rangle| \leq \frac{2}{r}  \|b_a |u|^\frac{r}{2}\|_2 \|A^\frac{1}{2} |u|^\frac{r}{2}\|_2\\
&\leq \frac{2\sqrt{\delta}}{r}\|A^\frac{1}{2} |u|^\frac{r}{2}\|_2^2 +\frac{\lambda \sqrt{\delta}}{r} \|u\|_r^r;\\
|\Imag\, \langle b_n\cdot \nabla u, u|u|^{r-2}\rangle| &\leq |\langle b_n\cdot \nabla u, u|u|^{r-2}\rangle| \leq \|b_a |u|^\frac{r}{2}\|_2 X^\frac{1}{2}\\
&\leq \frac{1}{4\varepsilon}\Real\, \langle \tilde Au, u|u|^{r-2} \rangle +\bigg(\delta \varepsilon-\frac{r-2}{r^2\varepsilon}\bigg)\|A^\frac{1}{2} |u|^\frac{r}{2} \|_2 ^2 + \lambda \delta \varepsilon \|u\|_r^r, \quad (\varepsilon>0);
\end{align*}
\begin{align*}
\text{RHS} &\equiv \Real \langle \big(z + \Lambda_r(b_n)\big)u,u|u|^{r-2}\rangle \\
&\geq \Real \langle \tilde Au, u|u|^{r-2} \rangle -\frac{2\sqrt{\delta}}{r}\|A^\frac{1}{2} |u|^\frac{r}{2}\|_2^2 + \bigg(z -\frac{\lambda\sqrt{\delta}}{r} \bigg)\|u\|_r^r;\\
\text{RHS} &\geq  (1-\frac{r^\prime\sqrt{\delta}}{2})\Real\, \langle \tilde Au, u|u|^{r-2} \rangle + \bigg(z -\frac{\lambda\sqrt{\delta}}{r} \bigg)\|u\|_r^r \tag{$\star$};
\end{align*}
\begin{align*}
\text{LHS} \equiv &|\Imag\, \langle \Lambda_r(b_n)u, u|u|^{r-2}\rangle| \\
\leq &\bigg[\frac{|r-2|}{2\sqrt{r-1}}+\frac{1}{4\varepsilon} \bigg]\Real\, \langle \tilde Au, u|u|^{r-2} \rangle 
+\big(\delta\varepsilon -\frac{r-2}{r^2\varepsilon}\big)\|A^\frac{1}{2} |u|^\frac{r}{2} \|_2 ^2 +\lambda\delta\varepsilon \|u\|_r^r;\tag{$\star\star$}\\
\text{LHS}\leq &\bigg[\frac{|r-2|}{2\sqrt{r-1}}+\frac{1}{4\varepsilon} \bigg] \text{RHS}+ \bigg\{\bigg[\frac{|r-2|}{2\sqrt{r-1}}+\frac{1}{4\varepsilon} \bigg]\frac{2\sqrt{\delta}}{r}+\delta\varepsilon -\frac{r-2}{r^2\varepsilon}\bigg\}\|A^\frac{1}{2} |u|^\frac{r}{2} \|_2 ^2\\
-&\bigg\{\bigg[\frac{|r-2|}{2\sqrt{r-1}}+\frac{1}{4\varepsilon} \bigg](z- \frac{\lambda\sqrt{\delta}}{r})-\lambda\delta\varepsilon \bigg\}\|u\|_r^r \tag{$\star\star\star$}\\
\end{align*}
Note that $r \leq 2r_\delta$ makes $\frac{1}{4\varepsilon}\frac{2\sqrt{\delta}}{r} + \delta \varepsilon -\frac{r-2}{r^2\varepsilon} > 0$ for any $\varepsilon >0.$ Thus,
setting $\varepsilon=\frac{1}{r\sqrt{\delta}}$ and $z=2 \lambda \sqrt{\delta},$ we obtain from $(\star\star\star)$ and $(\star)$
\begin{align*}
\text{LHS}\leq &\bigg[\frac{|r-2|}{2\sqrt{r-1}}+\frac{1}{4\varepsilon} \bigg] \text{RHS}+ \bigg\{\bigg[\frac{|r-2|}{2\sqrt{r-1}}+\frac{1}{4\varepsilon} \bigg]\frac{2\sqrt{\delta}}{r}+\delta\varepsilon -\frac{r-2}{r^2\varepsilon}\bigg\}\frac{rr^\prime}{4}\Real\, \langle \tilde Au, u|u|^{r-2} \rangle \\
-&\bigg\{\bigg[\frac{|r-2|}{2\sqrt{r-1}}+\frac{1}{4\varepsilon} \bigg](z- \frac{\lambda\sqrt{\delta}}{r})-\lambda\delta\varepsilon \bigg\}\|u\|_r^r\\
\leq & \bigg[\frac{|r-2|}{2\sqrt{r-1}}+\frac{1}{4\varepsilon} \bigg] \text{RHS}+\bigg\{\bigg[\frac{|r-2|}{2\sqrt{r-1}}+\frac{1}{4\varepsilon} \bigg]\frac{r^\prime\sqrt{\delta}}{2} +\big(\delta\varepsilon -\frac{r-2}{r^2\varepsilon}\big)\frac{rr^\prime}{4}\bigg\}(1-\frac{r^\prime\sqrt{\delta}}{2}\big)^{-1}\text{RHS}\\
= & \bigg(\frac{|r-2|}{2\sqrt{r-1}} +\frac{1}{4(r-1)}\big(\frac{1}{\varepsilon}+\delta r^2 \varepsilon \big) \bigg)2(2-r^\prime\sqrt{\delta})^{-1}\text{RHS}\\
=&\bigg[\frac{|r-2|}{\sqrt{r-1}}+r^\prime \sqrt{\delta}\bigg] (2-r^\prime \sqrt{\delta})^{-1}\text{RHS}. 
\end{align*}
If $r >2r_\delta,$ then setting $\varepsilon=\frac{\sqrt{r-1}}{r\sqrt{\delta}}$ and $z= 2\lambda \sqrt{\delta}$ we obtain from $(\star\star)$ and $(\star)$
\begin{align*}
\text{LHS}\leq &\bigg[\frac{r-2}{2\sqrt{r-1}}+\frac{1}{4\varepsilon} \bigg]\Real\, \langle \tilde Au, u|u|^{r-2} \rangle +\delta\varepsilon \|A^\frac{1}{2} |u|^\frac{r}{2} \|_2 ^2 +\lambda\delta\varepsilon \|u\|_r^r\\
\leq & \bigg[\frac{r-2}{2\sqrt{r-1}}+\frac{1}{4\varepsilon} +\delta \varepsilon\frac{rr^\prime}{4}\bigg]\Real\, \langle \tilde Au, u|u|^{r-2} \rangle +\lambda\delta\varepsilon \|u\|_r^r\\
\leq &\frac{r-2 +r\sqrt{\delta}}{\sqrt{r-1}}(2-r^\prime\sqrt{\delta})^{-1}\text{RHS},
\end{align*}
i.e.\,we have established ($\bullet$) for $b_n$ and $u_n$.

b) By a), the semigroups $e^{-t\Lambda_r(b_n)}$ are holomorphic and uniformly bounded on the sector
$S=:\{|\arg t | \leq \frac{\pi}{2}-\theta_r\}$. By ($\ast$), $u_n:=e^{-t\Lambda_r(b_n)}f \rightarrow u:=e^{-t\Lambda_r(b)}f$, $f \in L^r$, pointwise on the positive semi-axis.
Therefore, by the Privalov-Vitali Convergence Theorem $u$ is holomorphic in $S$
and $u(z)=s\mbox{-}L^r\mbox{-}\lim\mbox{-} u_n(z)$, $z \in S$. It follows that $u'(z)=s\mbox{-}L^r\mbox{-}\lim\mbox{-} u'_n(z)$, $z \in S$, and so $\Lambda_r(b_n)u_n \rightarrow \Lambda_r(b)u$ in $L^r$. Moreover, passing to a subsequence if needed, we have $g_n:=|u_n|^{r-2}u_n \rightarrow g:=|u|^{r-2}u$ $\mathcal L^d$ a.e. and $\|g_n\|_{r'}=\|u_n\|^{r-1}_r \rightarrow \|g\|_{r'}=\|u\|^{r-1}_r$. Thus, $g_n \rightarrow g$ in $L^{r'}$. It follows that in ($\bullet$) we can pass to the limit $n \rightarrow \infty$. This ends the proof of ($\ast\ast\ast\ast$). 
\end{proof}

\subsection{Uniformly elliptic case I}
\label{unif_elliptic_1}

We say that a matrix $a$ is uniformly elliptic and write $a \in (H_u)$ if $a\in (H_1)$ and there is a constant $\xi < \infty$ such that $a(x) \leq \xi I$ for $\mathcal L^d$ a.e $x \in \Omega.$

 In the uniformly elliptic case and $1 \leq \delta < 4$ the assumption $b_a \in L^2 + L^\infty$ is superfluous:

\begin{theorem}
\label{thm:markease}
Let $a \in (H_u), A=A_D$ or $A_N,$ and $b : \Omega \rightarrow \mathbb R^d, \; b \in \mathbf F_\delta(A)$ for some $1 \leq \delta < 4$.
Then the limit $$e^{-t\Lambda_r(b)}:=s\mbox{-} L^r \mbox{-} \lim_n e^{-t \Lambda_r(b_n)}, \quad r>r_\delta:=\frac{2}{2-\sqrt{\delta}}$$ exists uniformly in $t \in [0,T]$ for each $T<\infty$, and determines a positivity preserving, $L^\infty$ contraction, quasi contraction $C_0$ semigroup on $L^r$. 
\end{theorem}

\begin{proof} 1. Let $k > 2.$ Fix $o \in \Omega$. Define
$$
\eta(t):=\left\{
\begin{array}{ll}
0, & \text{ if } t< k, \\
\big( \frac{t}{k} - 1 \big)^k, & \text{ if } k \leq t \leq 2 k, \quad \quad \text{ and } \zeta(x) = \eta(\frac{|x-o|}{R}), \;\; R > 0.\\
1, & \text{ if } 2 k < t,
\end{array}
\right.
$$
Note that $ |\nabla \zeta | \leq R^{-1} \mathbf{1}_{\nabla \zeta} \zeta^{1-\frac{1}{k}}.$

By $u_n$ we denote the solution to $\big(\frac{d}{d t} + \Lambda_r(b_n) \big) u_n = 0, \; u_n(0) = f \in L^\infty \cap 
L^2_+.$ Set $v := \zeta u_n.$ Clearly,
\[
\left\langle \zeta (\frac{d}{d t} + A + b_n \cdot\nabla ) u_n, v^{r-1} \right\rangle =0.
\]
Set $[F, G]_- := F G - G F$ and $\langle A v, v^{r-1} \rangle := \langle \nabla v \cdot a \cdot \nabla v^{r-1} \rangle.$ Since the matrix is uniformly elliptic, 
$\langle A^\frac{1}{2} \phi, A^\frac{1}{2} \psi \rangle = \langle \nabla \phi \cdot a \cdot \nabla \bar{\psi} \rangle$ for all $\phi, \psi \in D(A^\frac{1}{2}) = E,$ where $E=W_0^{1,2}$ for $A=A_D$ and $E=W^{1,2}$ for $A=A_N.$ Also $u_n \in D(A) \cap L_+^\infty, \; D(A) \subset D(A^\frac{1}{2})= E.$ Thus $v \in E$ as well as $v^\frac{r}{2}$ and $\zeta v^{r-1}$ since $r \geq 2.$ Therefore $\langle A u_n, \zeta v^{r-1} \rangle = \langle \nabla u_n \cdot a \cdot \nabla(\zeta v^{r-1}) \rangle.$ Now it is easy to justify the following equation and equality.
\begin{align*}
\langle (\frac{d}{d t} + A + b_n \cdot \nabla) v, v^{r-1} \rangle & = \big\langle [A, \zeta]_-u_n + u_n b_n \cdot \nabla \zeta, v^{r-1} \big\rangle 
\tag{$\star$} \\
\langle [A, \zeta]_-u_n, v^{r-1} \rangle & =\frac{2}{r^\prime} \big\langle \nabla v^\frac{r}{2} \cdot a u_n v^{\frac{r}{2}-1} \cdot \nabla \zeta \big\rangle - \langle \nabla \zeta \cdot a v^{r-1} \cdot \nabla u_n \rangle \\
& = \frac{2}{r^\prime} \big\langle \nabla v^\frac{r}{2} \cdot \frac{a}{\zeta} \cdot \nabla \zeta, v^\frac{r}{2} \big\rangle - \frac{2}{r}\big\langle \nabla \zeta \cdot \frac{a}{\zeta} \cdot \nabla v^\frac{r}{2}, v^\frac{r}{2} \big\rangle + \big\langle \nabla \zeta \cdot \frac{a}{\zeta^2} \cdot \nabla \zeta, v^r \big\rangle.
\end{align*}
By the quadratic estimates
\begin{align*}
\big\langle u_n b_n \cdot \nabla \zeta, v^{r-1} \big\rangle & = \big\langle b_n \cdot \frac{\nabla \zeta}{\zeta}, v^r \big\rangle \\
& \leq \frac{\mu \sqrt{\delta}}{r} \|A^\frac{1}{2} v^\frac{r}{2}\|_2^2 + \frac{r \sqrt{\delta}}{4 \mu}\big\langle \nabla \zeta \cdot \frac{a}{\zeta^2}\cdot \nabla \zeta, v^r \big\rangle + \frac{\mu c(\delta)}{r \sqrt{\delta}} \|v\|_r^r \;\;( \mu > 0 ),\\
\frac{2(r-2)}{r}\big\langle \nabla v^\frac{r}{2} \cdot \frac{a}{\zeta} \cdot \nabla \zeta, v^\frac{r}{2} \big\rangle & \leq \frac{\mu \sqrt{\delta}}{r}\|A^\frac{1}{2} v^\frac{r}{2}\|_2^2 + \frac{(r-2)^2}{r \mu \sqrt{\delta}}\big\langle \nabla \zeta \cdot \frac{a}{\zeta^2}\cdot \nabla \zeta, v^r \big\rangle,
\end{align*}
we get from $(\star)$
\[
\frac{d}{d t} \|v\|_r^r + 2\bigg( \frac{2}{r^\prime} - (1 + \mu) \sqrt{\delta} \bigg) \|A^\frac{1}{2} v^\frac{r}{2}\|_2^2 \leq \bigg(\frac{(r-2)^2}{\mu \sqrt{\delta}} + \frac{r^2 \sqrt{\delta}}{4 \mu} +r \bigg)\big\langle \nabla \zeta \cdot \frac{a}{\zeta^2}\cdot \nabla \zeta, v^r \big\rangle + \frac{r + \mu}{\sqrt{\delta}} c(\delta) \|v\|_r^r.
\]
Recalling that $\frac{2}{r^\prime} > \sqrt{\delta},$ we can find $\mu > 0$ such that $\frac{2}{r^\prime} -(1 +\mu)\sqrt{\delta} \geq 0.$ Thus
\[
\frac{d}{d t} \|v\|_r^r \leq \bigg(\frac{4 (r-2)^2 + r^2 \delta}{4 \mu \sqrt{\delta}} +r \bigg)\big\langle \nabla \zeta \cdot \frac{a}{\zeta^2}\cdot \nabla \zeta, v^r \big\rangle + \frac{r + \mu}{\sqrt{\delta}} c(\delta) \|v\|_r^r \tag{$\star\star$}
\]
Next, $\big\langle \nabla \zeta \cdot \frac{a}{\zeta^2}\cdot \nabla \zeta, v^r \big\rangle \leq \xi R^{-2} \| \mathbf 1_{\nabla \zeta} \zeta^{-2 \theta} v^r \|_1,$ where $\theta = k^{-1}$ and $\mathbf 1_{\nabla \zeta}$ denotes the indicator of the support of $|\nabla \zeta|.$ Since $\|u_n\|_\infty \leq \|f\|_\infty, \; \|\mathbf 1_{\nabla \zeta}\|_\frac{r}{2\theta} \leq c(d,\theta) R^\frac{2\theta d}{r},$ and
\[
\| \mathbf 1_{\nabla \zeta} \zeta^{-2 \theta} v^r \|_1 \leq \|\mathbf 1_{\nabla \zeta}u_n^{2 \theta} \|_\frac{r}{2\theta} \|v\|_r^{r-2\theta}\leq \|\mathbf 1_{\nabla \zeta} \|_\frac{r}{2 \theta} \|u_n\|_\infty^{2\theta} \|v\|_r^{r-2\theta},
\]
we obtain, using the Young inequality, the crucial estimate (\textit{on which the whole proof rests})
\[
\big\langle \nabla \zeta \cdot \frac{a}{\zeta^2}\cdot \nabla \zeta, v^r \big\rangle \leq \frac{2 \theta}{r} [\xi c(d)]^\frac{r}{2 \theta} R^{d-\frac{r}{\theta}} \|f\|_\infty^r + \frac{r-2\theta}{r} \|v\|_r^r.
\]
Fix $\theta$ by $0 < \theta < \frac{r}{d+2r}.$ Now from $(\star\star)$ we obtain the inequality
\[
\frac{d}{d t} \|v\|_r^r \leq N(r, d, \delta) \|v\|_r^r + M(r, d, \delta) R^{-\gamma} \|f\|_\infty^r, \;\; \gamma = \frac{r}{\theta}-d > 0, \tag{$\star\star\star$}
\]
from which we conclude that, for given $T, f \in L^2 \cap L^\infty_+, \;\varepsilon > 0,$ there exists $R$ such that $$\sup_{t \in [0,T], n} \|\zeta u_n(t)\|_r \leq \varepsilon.$$

2. Let $k$ be as above. Define
$$
\eta(t):=\left\{
\begin{array}{ll}
1, & \text{ if } t< 2 k, \\
\big( 1 - \frac{1}{k} (t - 2 k) \big)^k, & \text{ if } 2 k \leq t \leq 3 k, \quad \quad \text{ and } \zeta(x) = \eta(\frac{|x-o|}{R}), \;\; R > 0.\\
0, & \text{ if } 3 k < t,
\end{array}
\right.
$$
Note that $ |\nabla \zeta | \leq R^{-1} \mathbf{1}_{\nabla \zeta} \zeta^{1-\frac{1}{k}}.$

Set $g := u_n - u_m$ and $v := \zeta g.$ Let $\langle A v, v |v|^{r-2} \rangle := \langle \nabla v \cdot a \cdot \nabla (v |v|^{r-2}) \rangle.$ 
Clearly
\[
\big \langle \zeta \big (\frac{d}{d t} + A +b_n \cdot \nabla \big ) g + \zeta (b_n - b_m) \cdot \nabla u_m, v |v|^{r-2} \big \rangle = 0,
 \]
 \[
 \big \langle \big (\frac{d}{d t} + A +b_n \cdot \nabla \big)v, v |v|^{r-2} \big \rangle = \big \langle [A, \zeta]_- g + v b_n \cdot \frac{\nabla \zeta}{\zeta}, v |v|^{r-2} \big \rangle + \langle \zeta (b_m - b_n) \cdot \nabla u_m, v |v|^{r-2} \rangle,
 \]
 \[
\langle [A, \zeta]_- g , v |v|^{r-2} \rangle = \frac{2(r-2)}{r} \langle \nabla |v|^\frac{r}{2} \cdot \frac{a}{\zeta} \cdot \nabla \zeta, |v|^\frac{r}{2} \rangle + \langle \nabla \zeta \cdot \frac{a}{\zeta^2} \cdot \nabla \zeta, |v|^r \rangle,
\]
\begin{align*}
\frac{d}{d t} \|v\|_r^r & \leq \bigg(\frac{4 (r-2)^2 + r^2 \delta}{4 \mu \sqrt{\delta}} +r \bigg)\big\langle \nabla \zeta \cdot \frac{a}{\zeta^2}\cdot \nabla \zeta, v^r \big\rangle + \frac{r + \mu}{\sqrt{\delta}} c(\delta) \|v\|_r^r\\
& + r \langle \zeta (b_m - b_n) \cdot \nabla u_m, v |v|^{r-2} \rangle \text{ with the same } \mu \text{ as in } (\star \star), \\
\langle \zeta (b_m - b_n) \cdot \nabla u_m, v |v|^{r-2} \rangle & \leq \langle \zeta (b_m - b_n) \cdot a^{-1} \cdot (b_m -b_n) \rangle^\frac{1}{2} \langle \nabla u_m \cdot a \zeta \cdot \nabla u_m \rangle^\frac{1}{2} (2\|f\|_\infty)^{r-1}.
\end{align*}
In order to estimate $\int_0^T \langle \nabla u_m(t) \cdot a \zeta \cdot \nabla u_m (t) \rangle d t$ note that $\langle \frac{d}{d t} u_m + A u_m + b_m \cdot \nabla u_m, \zeta u_m \rangle = 0,$ or
\[
\frac{1}{2} \frac{d}{d t} \langle \zeta u_m^2 \rangle + \langle \nabla u_m \cdot a \zeta \cdot \nabla u_m \rangle + \langle \nabla u_m \cdot a u_m \cdot \nabla \zeta \rangle + \langle b_m \cdot \nabla u_m, \zeta u_m \rangle = 0,
\]
and so
\[
\frac{d}{d t} \langle \zeta u_m^2 \rangle + \langle \nabla u_m \cdot a \zeta \cdot \nabla u_m \rangle \leq 2 \big(\big \langle \nabla \zeta \cdot \frac{a}{\zeta} \cdot \nabla \zeta \big \rangle + \langle \zeta b \cdot a^{-1} \cdot b \rangle \big) \|f\|^2_\infty,
\]
\begin{align*}
\int_0^T \langle \nabla u_m(t) \cdot a \zeta \cdot \nabla u_m (t) \rangle d t & \leq \| f \|_2^2  + 2 T \big( \big\langle \nabla \zeta \cdot \frac{a}{\zeta} \cdot \nabla \zeta \big \rangle + \langle \zeta b \cdot a^{-1} \cdot b \rangle \big) \|f\|^2_\infty \\
& \equiv \| f \|_2^2  + T L(R) \|f\|_\infty^2.
\end{align*}
Now it should be clear that the above is sufficient for concluding that the following inequality analogous to $(\star\star\star)$  holds for all $n, m$ and $t \in [0,T],$
\begin{align*}
e^{-N t} \|\zeta g\|_r^r & \leq t M R^{-\gamma} \|f\|_\infty^r \\ 
& + \sqrt{t} (2 \|f\|_\infty)^{r-1} \big(\|f\|_2^2 + t L(R) \|f\|_\infty^2 \big)^\frac{1}{2} \big\langle \zeta (b_n-b_m)\cdot a^{-1}\cdot(b_n-b_m) \rangle^\frac{1}{2}. 
\end{align*}

3. It follows from Step 1 and Step 2 that, for each $0 <T < \infty, \; r > r_\delta, f \in L^2_+ \cap L^\infty$ and $\epsilon > 0,$ we can find $R < \infty$ and $M < \infty$ such that
\[
\sup_{t \in [0, T], \; n, m \geq 1} \|(1-\mathbf{1}_{B(o, 2 k R)})(u_n(t) - u_m(t))\|_r < \epsilon; \quad
 \sup_{t \in [0, T], n, m \geq M}\|\mathbf{1}_{B(o, 2 k R)}(u_n(t) - u_m(t))\|_r < \epsilon.
\]
The proof of Theorem \ref{thm:markease} is completed. 
\end{proof}

\begin{remark*}
Let $a \in (H_1).$ Set $A := A_{iD}$ or $A_N.$ Fix any compact $K$ in $\Omega.$ The proof of Theorem \ref{thm:markease} shows that the matrix $a(x)$ admits the $(|x|^\varkappa,  \; \varkappa < 2)$-growth at infinity while on $K$ being just from $[L^1]^{d\times d}$. Indeed, for instance, let $A=A_{iD}.$ By the definition of $A_{iD},$ $\psi \in D(A^\frac{1}{2})$ if and only if $\psi \in W_0^{1,2}$ and $\langle \nabla \psi \cdot a \cdot \nabla \bar{\psi} \rangle < \infty.$ It is a simple matter to check that 
$$
\langle \nabla (\zeta v |v|^{r-2}) \cdot a \cdot \nabla (\zeta v |v|^{r-2}) \rangle < \infty,
$$ 
and hence to conclude that 
$$\langle A g, \zeta v |v|^{r-2} \rangle = \langle \nabla g \cdot a \cdot \nabla (\zeta v |v|^{r-2}) \rangle.$$
If $A=A_D$ one has to  modify $\zeta$ to $C^1_c$ function and use the fact that $C^1_c$ is a core of $A^\frac{1}{2}_D.$ 
\end{remark*}

\subsection{The maximal interval of bounded $L^p$ solvability for $-\nabla \cdot a \cdot \nabla + b \cdot \nabla$}
\label{max_solve_sect}

\begin{theorem}
\label{thm:marksolve}
Let $a \in (H_u)$. Assume that $b \in \mathbf F_\delta(A)$ for some $0 < \delta < 4.$

{\rm(a)}  $e^{-t \Lambda_r(b)}$, $r \in I_c$, extends to a positivity preserving, $L^\infty$ contraction, quasi bounded holomorphic semigroup on $L^r$ for every $r \in I_m - I_c$, $I_m:= ]\frac{2}{2- \frac{d-2}{d}\sqrt{\delta}}, \infty[$.

{\rm(b)} For every $r \in I_m$ and $q>r$ there are constants $c_i = c_i(\delta, r, q)$, $i=1,2$, such that
\[
\|e^{-t \Lambda_r(b)} \|_{r \rightarrow q} \leq c_1 \; e^{t c_2} \; t^{-\frac{d}{2} (\frac{1}{r}-\frac{1}{q})}, \quad t>0.
\]
\end{theorem}

\begin{proof}
Proof of (b). Suppose that (a) has been proved. Then (b) for $r\in I_m-I_c$ and $q\in I_c$  follows from (a) by Theorem \ref{thm:markeast}($\ast\ast\ast$) and the Extrapolation Theorem (Appendix \ref{appendix_B}). By interpolation, we obtain (b) for all $q>r$ and $r \in I_m$.

Proof of (a).
Let $\Omega=\mathbb R^d.$ Set $b_n:=\gamma_{\varepsilon_n}*(b\mathbf{1}_n ),$ where $\gamma_{\varepsilon_n}$ is the K.\,Friedrichs mollifier, $\mathbf{1}_n$ is the indicator of $\{x \in \Omega: |x| \leq n, |b(x)| \leq n\}.$ 

For any $\tilde{\delta}>\delta$ we can select a sequence $\varepsilon_n \downarrow 0$ such that
$b_n
 \in \mathbf{F}_{\tilde{\delta}}(A)$ with the same $\lambda=\lambda_\delta$.
(Indeed, we have for $f \in L^2$
\begin{align*}
\||b_n|(\lambda+A)^{-\frac{1}{2}}f\|^2_2 & \leqslant \||\mathbf{1}_n b|(\lambda+A)^{-\frac{1}{2}}f\|^2_2 + \||b_n-\mathbf{1}_n b|(\lambda+A)^{-\frac{1}{2}}f\|^2_2 \\
& \leqslant \delta \|f\|^2_2 + \||b_n-\mathbf{1}_n b|(\lambda+A)^{-\frac{1}{2}}f\|^2_2
\end{align*}
In turn, by H\"{o}lder's inequality,
$$
 \||b_n-\mathbf{1}_n b|(\lambda+A)^{-\frac{1}{2}}f\|^2_2 \leqslant \|b_n-\mathbf{1}_n b\|_{2d} \|(\lambda+A)^{-\frac{1}{2}}f\|^2_{\frac{2d}{d-1}} \leqslant C_d\|b_n-\mathbf{1}_n b\|_{2d} \|f\|^2_2,
$$
where $C_d>0$ is the constant in the uniform Sobolev inequality.
Since $\mathbf{1}_n b \in L^\infty$ and has compact support (and hence $\gamma_{\varepsilon_n}*\mathbf{1}_n b \rightarrow \mathbf{1}_n b$ in $L^{2d}$ as $\varepsilon \downarrow 0$), for every $\tilde{\delta}>\delta,$ we can select $\varepsilon_n$, $n=1,2,\dots$ sufficiently small so that
$\|b_n-\mathbf{1}_n b\|_{2d}<\frac{\tilde{\delta}-\delta}{C_d}$, and hence $ \||b_n-\mathbf{1}_n b|(\lambda+A)^{-\frac{1}{2}}f\|^2_2<(\tilde{\delta}-\delta)\|f\|^2_2$.
Therefore,
$
\||b_n|(\lambda+A)^{-\frac{1}{2}}f\|^2_2<\tilde{\delta} \|f\|^2_2,
$
as needed.)

Since our assumptions on $\delta$ involve strict inequalities only, we may assume that $b_n \in \mathbf{F}_\delta(A)$ for all $n$ with the same $\lambda=\lambda_\delta$.

First, we prove that
\[
\|(z+\Lambda_{q}^*(b_n))^{-1}\|_{qj \rightarrow qj} \leq c|z-z_0|^{-1},  \quad \Real z > z_0, \quad n=1,2,\dots, \quad j=\frac{d}{d-2},\tag{$\bullet$}
\] 
where $\Lambda_{q}^*(b_n)$ is the adjoint of $\Lambda_r(b_n)$, $q=\frac{r}{r-1} \in ]1,\frac{2}{\sqrt{\delta}}[$, for some $c$, $z_0>0$ independent of $n$.

For $1<p \leq 2$, let $B_{p,n}:=-\nabla \cdot b_n = -b_n \cdot \nabla - {\rm div\,} b_n$, $D(B_{p,n}):=D(A_p)$.
Then the operator $\Lambda_p^*=A_p + B_{p,n}$ of domain $D(A_p)$ is the (minus) generator of a quasi bounded  holomorphic semigroup on $L^p$.

Set $u \equiv u_n:=(z+\Lambda_p^*)^{-1}h, \;h \in L^1 \cap L^\infty, \;\Real z>\frac{\lambda \delta}{2(r-1)}$.
By consistency, $u=(z+\Lambda^*)^{-1}h$, so $u \in D(A^\frac{1}{2})$, and, clearly, $\|u\|_\infty<\infty$ for every $n$.

a) $0<\delta<1$. Fix any $2 \leq q < \frac{2}{\sqrt{\delta}}$ close to $\frac{2}{\sqrt{\delta}}$. Then 
\[
\langle (z+\Lambda^*)u,u|u|^{q-2}\rangle = \langle (z+A- \nabla \cdot b_n)u,u|u|^{q-2}\rangle,
\]
\[-\langle \nabla \cdot (b_n u),u|u|^{q-2}\rangle = \langle b_n \cdot u \nabla \bar{u},|u|^{q-2}\rangle  + (q-2)\langle b_n \cdot \nabla |u|, |u|^{q-1} \rangle,
\]
and so
$$
z\|u\|_q^q + \langle Au,u|u|^{q-2}\rangle+ \langle b_n \cdot u \nabla \bar{u},|u|^{q-2}\rangle + (q-2)\langle  b_n \cdot \nabla |u|, |u|^{q-1}\rangle = \langle h,u|u|^{q-2}\rangle.
$$
Taking the real and imaginary parts of this identity, we have
\begin{align*}
&\Real z \|u\|_q^q + \Real\langle Au,u|u|^{q-2}\rangle + (q-1)\langle b_n\cdot\nabla |u|,|u|^{q-1}\rangle \leq |\langle h,u|u|^{q-2}\rangle|\\
&\alpha \big(|\Imag z|\|u\|_q^q - |\Imag \langle Au,u|u|^{q-2}\rangle| - |\langle b_n \cdot u \nabla \bar{u},|u|^{q-2}\rangle| \big) \leq  \alpha |\langle h,u|u|^{q-2}\rangle| \quad (0<\alpha<1). 
\end{align*}
Adding these inequalities and using the inequality $|\Imag \langle Au,u|u|^{q-2}\rangle| \leq \frac{q-2}{2\sqrt{q-1}} \Real\langle Au,u|u|^{q-2}\rangle,$ we obtain
\begin{align*}
&\alpha |z| \|u\|_q^q + \big(1-\alpha \frac{q-2}{2\sqrt{q-1}}\big) \Real\langle Au,u|u|^{q-2}\rangle\\
&\leq \alpha|\langle b_n \cdot u \nabla \bar{u},|u|^{q-2}\rangle|+(q-1)|\langle b_n\cdot\nabla |u|,|u|^{q-1}\rangle| + (1+\alpha)\|h\|_q\|u\|_q^{q-1}.
\end{align*}
Recalling that $|\langle b_n \cdot u \nabla \bar{u},|u|^{q-2}\rangle|^2 \leq \langle b_a^2,|u|^q\rangle X,$ where
\[
X\equiv \langle \nabla u \cdot a \cdot \nabla \bar u, |u|^{q-2} \rangle = \Real\, \langle A u, u|u|^{q-2} \rangle -\frac{4(q-2)}{q^2}\|A^\frac{1}{2}|u|^\frac{q}{2} \|_2^2,
\]
(see the proof of Theorem \ref{thm:markeast} steps 1, 4)) and performing quadratic estimates, we conclude that, for a sufficiently small $\alpha$, there exist $z_0=z_0(\alpha)>0$ and a constant $C_\alpha <\infty$ such that
\[
|z-z_0| \|u\|_q^q + \|A^{\frac{1}{2}}|u|^{\frac{q}{2}}\|_2^2  \leq C_\alpha\|h\|_q \|u\|_q^{q-1}, \quad \Real z > z_0.
\]

The latter, Young's inequality ($|z-z_0|^{\frac{1}{q'}}\|u\|_q^{q-1}\|u\|_{qj} \leq \frac{|z-z_0|\|u\|_q^q}{q'} + \frac{\|u\|_{qj}^q}{q}$) and the Sobolev Embedding Theorem yield
\[
\|u\|_{qj} \leqslant c(q,d) |z-z_0|^{-\frac{1}{q'}}\|h\|_q,  \quad \Real z > z_0. \tag{$\star$}
\]

(b) Let $1 \leq \delta<4$, so the interval $\bigl]1,\frac{2}{\sqrt{\delta}}\bigr[$ does not contain 2. Fix any $1 < q <\frac{2}{\sqrt{\delta}}$ close to $\frac{2}{\sqrt{\delta}}$. Following the arguments above, but using Theorem \ref{thm:markovest}(\textit{i}),(\textit{ii}) (Appendix \ref{markov_sect}) in place of Theorem \ref{thm:markovest}(\textit{iv}),(\textit{v}), we obtain $(\star)$ in the case.

\smallskip

We are in position to complete the proof of Theorem \ref{thm:marksolve} for $\Omega=\mathbb R^d.$
Fix $1<q<\frac{2}{\sqrt{\delta}}$ close to $\frac{2}{\sqrt{\delta}}$. Set $R_n^*(z):=(z+\Lambda_q^*(b_n))^{-1}$.
Our goal is to prove that
$\|R_n^*(z)\|_{qj \rightarrow qj} \leq c|z-z_0|^{-1}$,  $\Real z > z_0$.
 Given $l = (l^1, \dots l^d) \in \mathbb Z^d$, define a cube
\[
Q_l := \{ x \in \mathbb R^d \mid | l^m |z-z_0|^{-\frac{1}{2}} -x^m | \leq (4 |z-z_0|)^{-\frac{1}{2}}, \; m = 1, \dots , d \}.
\]
Given $k \in \mathbb Z^d$, subdivide $\mathbb{R}^d$ into $Q_k + \sum_{i \in \mathbb Z^d - \{k\}} Q_i.$ Fix $k.$

1) Let $i \in \mathbb{Z}^d$ be such that $|k - i| \leq \alpha := \frac{\sqrt{d}}{2} (1 + \frac{1}{\gamma)}, \;  \frac{1}{2 d} < \gamma < \frac{1}{d}.$ Then
\[
\| \mathbf{1}_k R_n^*(z) \mathbf{1}_i \|_{q j \rightarrow q j} \leq c_1 |z-z_0|^{-1}, \;\; c_1 = c_1(q, d),
\]
where $\mathbf{1}_i$ denotes the indicator function of $Q_i$.
[Indeed, by H\"older's inequality, $\| \mathbf{1}_k R_n^*(z) \mathbf{1}_i h \|_{q j} \leq \|R_n^*(z) \|_{q \rightarrow q j} \|\mathbf{1}_i \|_1^\frac{1}{q j^\prime} \|h\|_{q j}, \; j^\prime = \frac{d}{2},$ so $(\star)$ and $\|\mathbf{1}_i \|_1 = c_d |z-z_0|^{-\frac{d}{2}}$ yield the required].

\medskip

2) Let $i \in \mathbb{Z}^d$ be such that $|k - i| \leq \alpha.$ Then ($\xi$ the ellipticity constant)
\[
\| \mathbf{1}_k R_n^*(z) \mathbf{1}_i \|_{q j \rightarrow q j} \leq c_2 |k - i|^{-\frac{1}{\gamma}} |z-z_0|^{-1}, \;\;c_2 =c_2(q, d, \gamma, \xi).
\]
(The proof of the inequality, which we call the separating property, is given below.)

\medskip

1) and 2) combined yield
\begin{align*}
|\langle R_n^*(z) h, g \rangle| & \leq \sum_{i,k \in \mathbb Z^d} \big | \langle \mathbf 1_k R_n^*(z) \mathbf 1_i h, \mathbf 1_k g \rangle \big | \leq \sup_{l \in \mathbb Z^d}\bigg \langle  \sum_{i\in \mathbb Z^d} | \mathbf{1}_l R_n^*(z) \mathbf{1}_i h|, \sum_{k \in \mathbb Z^d}\mathbf 1_k | g |\bigg\rangle ;\\
\|R_n^*(z) \|_{q j \rightarrow q j} & \leq \tilde{c}_1 |z-z_0|^{-1} + \tilde{c}_2 |z-z_0|^{-1} \sup_{k \in \mathbb Z^d} \sum_{ i \in \mathbb Z^d; |k-i| \geq \alpha } |k - i|^{-\frac{1}{\gamma}} \\
 & \leq \tilde{c}_1 |z-z_0|^{-1} + \tilde{c}_2 |z-z_0|^{-1} \sup_{k \in \mathbb Z^d} \int_\alpha^\infty t^{d-1 -\frac{1}{\gamma}} d t \\
 & \leq c_3 |z-z_0|^{-1},
\end{align*}
which yields ($\bullet$) in the case $\Omega = \mathbb R^d.$

By duality, ($\bullet$) yields ($r:=q' \in I^o_c$, $s:=(qj)' \in I_m$)
\[
\|(z+\Lambda_{r}(b_n))^{-1}\|_{s \rightarrow s} \leq c|z-z_0|^{-1},  \quad \Real z > z_0, \quad n=1,2,\dots
\] 
In view of Theorem \ref{thm:markease}, for every $f \in L^r \cap L^s$
\begin{equation}
\tag{$\bullet\bullet$}
\|(z+\Lambda_{r}(b))^{-1}f\|_{s} \leq c|z-z_0|^{-1}\|f\|_s,  \quad \Real z > z_0.
\end{equation}
By ($\bullet\bullet$), $\|e^{-t\Lambda_r(b)}f\|_{s} \leq M\|f\|_s$, $t \in [0,1]$, for $M<\infty$. 
The strong continuity of $e^{-t\Lambda_r(b)}$ in $L^r$, $r \in I_c^o$, and the following elementary result:

\begin{quote}
Let $S_k : L^{p_1} \cap L^{p_2} \rightarrow L^{p_1} \cap L^{p_2}, \; 1 \leq p_1 < p_2 \leq \infty, \; k=0, 1, 2, \dots,$ be such that $\|S_k f\|_{p_i} \leq M \|f\|_{p_i}, \; i = 1, 2,$ for all $f \in L^{p_1}\cap L^{p_2}, \; k$ and some $M < \infty.$ If $\|S_k f - S_0 f \|_{p_0} \rightarrow 0$ for some $p_0 \in ]p_1, p_2[,$ then $\|S_k f - S_0 f \|_p \rightarrow 0$ for every $p \in ]p_1, p_2[.$   
\end{quote}

\noindent yield $e^{-t\Lambda_{r}(b)}f \rightarrow f$ strongly in $L^{s_1}$ as $t \downarrow 0$, for all $s<s_1 \leq r$ (and ultimately for all $s_1 \in I_m - I_c^o$), which gives assertion (a) of the theorem.

\smallskip

Now we come to the proof of the inequality from 2). Let $i \in \mathbb Z^d - \{k\}$ be such that $|k-i| \geq \alpha.$ Define functions $\zeta_i(x)$ by
\[
\zeta_i(x):= \eta  \bigg( \frac{\big |  |z-z_0|^{-\frac{1}{2}} k- x \big |}{|k-i|} |z-z_0|^\frac{1}{2} \left(1+\frac{1}{\gamma}\right) \bigg),
\]
where
$$\eta (s) = \left\{ \begin{array}{cr} 1,&  s \leq 1 \\ \big(1-\frac{s-1}{m} \big)^m,&  1 < s < m+1 \\
 0,& m+1  \leq s \quad (m =\frac{1}{\gamma}). \end{array} \right.$$
We list the following properties of $\zeta_i: \; \zeta_i  \upharpoonright Q_k = \mathbf{1}_k, \; \zeta_i \upharpoonright Q_i = 0.$ Define
\[
\Gamma_i(x) := \nabla \zeta_i (x) \cdot \frac{a(x)}{\zeta_i(x)^2} \cdot \nabla \zeta_i (x), \;\; e(x) = \frac{k |z-z_0|^{-\frac{1}{2}}- x}{|k |z-z_0|^{-\frac{1}{2}}- x |}.
\]
Using the inequalities $-\partial_s \eta (s) \leq \eta(s)^{1-\gamma}$ and $e\cdot a \cdot e \leq \xi,$ we have
\[
\Gamma_i \leq \bigg(1+\frac{1}{\gamma}\bigg)^2 \xi |k-i|^{-2} |z-z_0| \zeta_i^{-2 \gamma}.\\
\tag{$\circ$}
\]

\begin{lemma} Define $v_i := R_n^*(z) \mathbf{1}_i f, \, u := \zeta_i v_i, \; f \in L^1 \cap L^\infty.$ Then
\[
|z-z_0| \|u\|_q^q + \|u\|^q_{q j} \leq c(q) \langle \Gamma_i |u|^q \rangle.\\
\tag{\textit{i}}
\]
\[
|z-z_0|^{-2 \gamma /q} \|u\|_q^{q-2\gamma} \|u\|_{q j}^{2 \gamma} \leq c(q,\gamma)\xi |k-i|^{-2} \langle |v_i|^{2 \gamma} |u|^{q-2 \gamma} \rangle.\\
\tag{\textit{ii}}
\]
\[
\|u\|_{q j} \leq c(q, \gamma) \xi^\frac{1}{2 \gamma} |k-i|^{-\frac{1}{\gamma}} |z-z_0|^{-1} \|f\|_{q j}.\\
\tag{\textit{iii}}
\]
\end{lemma}
Due to $\zeta_i \upharpoonright Q_k = \mathbf{1}_k,$ $(iii)$ implies 2).

\begin{proof}[Proof of Lemma] $(ii)$ follows from $(i)$, Young's inequality and $(\circ).$ In turn, $(iii)$ follows from $(ii)$ by applying H\"older's inequality to $\langle |v_i|^{2 \gamma} |u|^{q-2 \gamma}\rangle,$ so that $\|u\|_{q j} \leq c(q, \gamma) \xi^\frac{1}{2 \gamma}|k-i|^{-\frac{1}{\gamma}} |z-z_0|^\frac{1}{\gamma} \|v_i \|_q,$ then applying $\|R_n^*(z)\|_{q \rightarrow q} \leq C|z-z_0|^{-1}$ to $\|v_i \|_q : \; \|v_i \|_q \leq c |z-z_0|^{-1} \|\mathbf{1}_i f\|_q,$ and finally $\|\mathbf{1}_i f\|_q \leq \|\mathbf{1}_i \|_{q j^\prime} \|f\|_{q j}$ by H\"older's inequality. 

We are left to prove $(i).$ We have $(z+A_q -\nabla \cdot b_n) v_i = \mathbf{1}_i f$ (if $q \geq 2$, we write $A \equiv A_2$ in place of $A_q$) and, since $\zeta_i \upharpoonright Q_i =0,$
$\langle (z+A_q+\nabla \cdot b_n)v_i, \zeta_i u |u|^{q-2} \rangle =0$.
 One can easily check that $g \in D(A^\frac{1}{2}) \Rightarrow \zeta_i g \in D(A^\frac{1}{2}).$ Since $v_i \in D(A)$, both $u$ and $u|u|^{q-2}$ belong to $D(A^\frac{1}{2})$ (Appendix \ref{markov_sect}, Theorem \ref{thm:markovest}). Thus
\[
z \|u\|^r_r + \langle A_q u, u |u|^{q-2} \rangle - \langle \nabla \cdot b_n v_i, \zeta_i u |u|^{q-2} \rangle = \langle [ A_q, \zeta_i]_- v_i, u |u|^{q-2} \rangle,
\]
where
\begin{align*}
\langle [ A_q, \zeta_i]_- v_i, u |u|^{q-2} \rangle & := \big\langle \nabla \zeta_i \cdot a v_i \cdot \nabla \big(\bar{u} |u|^{q-2} \big) \big \rangle - \langle
 \nabla \zeta_i \cdot a \cdot \nabla v_i, u |u|^{q-2} \rangle \\
& = \langle \nabla \zeta_i \cdot \frac{a u}{\zeta_i} \cdot \nabla \big(\bar{u} |u|^{q-2} \big) \big \rangle - \big \langle \nabla \zeta_i \cdot \frac{a }{\zeta_i} \cdot \nabla u, u |u|^{q-2} \big \rangle + \big \langle \nabla \zeta_i \cdot \frac{a }{\zeta_i^2} \cdot \nabla \zeta_i, |u|^q \big \rangle.
\end{align*}
The rest of the proof resembles what we have already done. Taking the real and imaginary parts
of the last equation and performing quadratic estimates we arrive at $(i)$.
\end{proof}

The same proof works for an arbitrary open $\Omega \subset \mathbb R^d$ with $\zeta_i \upharpoonright \Omega$ in place of $\zeta_i.$ 
\end{proof}

\begin{remark*}
The example of $\Lambda_r \supset - \Delta + b \cdot \nabla$ with
$
b(x):=c |x|^{-2} x \in \mathbf{F}_\delta(-\Delta)$ in $\mathbb R^d$, $c = \frac{d-2}{2} \sqrt{\delta}, 
$ $\delta<4$,   can be used to show that the interval of bounded solvability $I_m=]\frac{2}{2- \frac{d-2}{d}\sqrt{\delta}}, \infty[$ can not be enlarged, i.e.\,the constructed $C_0$ semigroup $e^{-t\Lambda_r}$ can not be extended to a quasi bounded $C_0$ semigroup on $L^{s}$ for $s \not\in I_m$. 
Indeed, by duality it suffices to show that  $e^{-t\Lambda_q^*}$, $q \in ]1,\frac{2}{\sqrt{\delta}}\frac{d}{d-2}[$, can not be extended to a quasi bounded $C_0$ semigroup on $L^p$ for any $p \geq \frac{2}{\sqrt{\delta}}\frac{d}{d-2}$. Set $u(x):=|x|^{-c} \exp(-|x|^2)$, $x \in \mathbb R^d.$ Then $u \in D(\Lambda^*_q)$ for any $q \in ]1,\frac{d}{c+2}[$.  
Clearly, $]1,\frac{d}{c+2}[ \subset ]1,\frac{2}{\sqrt{\delta}}[$, the interval of contractive solvability for $e^{-t\Lambda_q^*}$. Now, suppose that $e^{-t\Lambda_q^*}$ admits extension to a semigroup of bounded linear operators $L^{p} \rightarrow L^{p}$. Then, using the analogue of Theorem \ref{thm:markeast}($\ast\ast\ast$) for the semigroup $e^{-t\Lambda_q^*}$ and then applying the Extrapolation Theorem (Appendix \ref{appendix_B}), we obtain that $e^{-t\Lambda_q^*} \in \mathcal B(L^q,L^p)$, $t>0$, and $\|e^{-t \Lambda^*_q(b)} \|_{q \rightarrow p} \leq c_1 e^{t c_2}  t^{-\frac{d}{2} (\frac{1}{q}-\frac{1}{p})}$. Next, it is seen that $(\lambda + \Lambda_q^*)u=f$, $\lambda>c_2$, $f:=\bigl[\lambda + 2(d-c)\bigr]|x|^{-c}e^{-|x|^2}-4|x|^{-c+2}e^{-|x|^2} \in L^q$, and so $(\lambda + \Lambda_q^*)^{-1}f=u$. The latter means, in view of $\|e^{-t \Lambda^*_q(b)} \|_{q \rightarrow p} \leq c_1 \; e^{t c_2} \; t^{-\frac{d}{2} (\frac{1}{q}-\frac{1}{p})}$, that $u \in L^p$, which is clearly false.
\end{remark*}

\subsection{Uniformly elliptic case II. The Hille-Lions approach}
\label{unif_elliptic_2}

Let $a \in (H_u), \;b \in \mathbf F_\delta(A), \;0<\delta < 4,$ and let $\Lambda_r(a,b)$ be the operator defined in Theorem \ref{thm:markease}. It is useful (in some cases necessary) to have the convergence
\[
e^{- t \Lambda_r (a, b)} = s \mbox{-} L^r \mbox{-} \lim_n e^{- t \Lambda_r (a_n, b_n )}, 
\]
where $a_n$ and $b_n$ have smooth and bounded entries.
\begin{theorem}
\label{thm:markfeast}
Fix $\delta < 1.$  Let $a, a_n \in (H_u),  \; b, b_n:\Omega \rightarrow \mathbb R^d, n=1,2,\dots$ Assume that
\[
 b \in \mathbf F_\delta(A(a)), \; b_n \in \mathbf F_\delta(A(a_n)) \text{ with fixed } \lambda =\lambda_\delta \text{ for all } n = 1,2,\dots\tag{$i$}
\]
\[
a_n \rightarrow a \text{ strongly in }[L^2_\loc]^{d\times d}, \quad b_n \rightarrow b \text{ strongly in } [L^2_\loc]^d.\tag{$ii$}
\]
Then $s \mbox{-} L^r \mbox{-} \lim_{n \uparrow \infty} e^{-t \Lambda_r(a_n, b_n)} = e^{-t \Lambda_r(a, b)}$ whenever $r  \in I^o_c= ]r_\delta, \infty[$ (recall $r_\delta = \frac{2}{2-\sqrt{\delta}}).$
\end{theorem}

\begin{proof}
$\mathbf{1}.$~Set $\mathcal H = L^2(\Omega, \mathcal L^d), \; \mathcal H_+ =\big ((D(A^\frac{1}{2}), \; \|f\|_+^2 = \lambda \|f\|_2^2 + \|A^\frac{1}{2} f\|_2^2 \big), \; \mathcal H_- =  \mathcal H_+^*.$ By $\langle g, f \rangle, \; g \in \mathcal H_+, f \in \mathcal H_-$ denote the pairing between $ (\mathcal H_+,  \mathcal H_-)$ which coincides with $\langle g, f\rangle_\mathcal H$ for $f \in \mathcal H.$ Then $\mathcal H_+ \subset \mathcal H \subset \mathcal H_-$ is the standard triple of Hilbert spaces w.r.t. $\langle , \rangle_{\mathcal H}.$ By $\hat{A}$ denote the extension by continuity of $A \equiv A(a)$ to the operator from $\mathcal H_+$ to $\mathcal H_-.$ Then $\hat A \in \mathcal B(\mathcal H_+, \mathcal H_-)$ and $|\langle f, (\zeta + \hat A)f \rangle | \geq \|f\|_+^2, \; f \in \mathcal H_+, \; \Real \zeta > \lambda.$ Thus $\zeta + \hat A$ is a bijection. Clearly $(\zeta + \hat A)^{-1} \upharpoonright \mathcal H = (\zeta + A)^{-1}.$

Consider $\hat B \equiv b\cdot\nabla : \mathcal H_+ \rightarrow \mathcal H_-.$ By $b \in \mathbf F_\delta(A), \hat B \in \mathcal B(\mathcal H_+, \mathcal H_-)$ and 
\[
|\langle f,(\zeta + \hat A + \hat B)f \rangle| \geq (1 - \sqrt{\delta}) \langle f,(\mu + \hat A )f \rangle, \; \mu = \frac{|2 \zeta -\lambda \sqrt{\delta}|}{2(1-\sqrt{\delta})} > 0
\]
whenever $\Real \zeta > \frac{\lambda \sqrt{\delta}}{2}.$

Thus, for $\hat \Lambda \equiv \hat \Lambda(a,b) := \hat A + \hat B$ and every  $\zeta \in \mathcal O_o :=\{z \mid \Real z > \frac{\lambda \sqrt{\delta}}{2} \},$ $\zeta + \hat \Lambda$ is a bijection.

In $\mathcal H$ define the operators
\[
H = b_a (\bar \zeta+A)^{-\frac{1}{2}}, \; S = \frac{b}{b_a} \cdot \nabla (\zeta+A)^{-\frac{1}{2}} \text{ and } H^* S.
\]
Clearly, due to $b \in \mathbf F_\delta(A),$ for each $\zeta$ with $\Real \zeta \geq \lambda,$
\[
\|H^*\|_{2 \rightarrow 2} = \|H\|_{2 \rightarrow 2}\leq \sqrt{\delta}.
\]
Since ($\mathcal L^d$ a.e.) $|S f| \leq \sqrt{\frac{b}{b_a}\cdot a^{-1} \cdot \frac{b}{b_a}}\sqrt{\nabla (\zeta +A)^{-\frac{1}{2}} f \cdot a \cdot \nabla (\zeta +A)^{-\frac{1}{2}} \bar f},$ and so $\|S f\|_2 \leq \|f\|_2, \;\;f \in L^2,$ we conclude that $\|H^*S\|_{2\rightarrow 2} \leq \sqrt{\delta}.$ 
\[
\text{Set } \hat R_\zeta \equiv \hat R_\zeta(a,b) := (\zeta + \hat \Lambda )^{-1}, \; \hat F_\zeta \equiv \hat F_\zeta(a) := (\zeta +\hat A)^{-1},\; F_\zeta  := (\zeta + A)^{-1}, \; P_\zeta \equiv P_\zeta(a,b) := H^* S. 
\]
Clearly, for all $\zeta, \eta \in \mathcal O := \{z \mid \Real z > \lambda \},$
\begin{align*}
& \hat R_\zeta - \hat R_\eta  = (\eta -\zeta) \hat R_\zeta \hat R_\eta;\tag{$p_1$}\\
& \hat R_\zeta  = \hat F_\zeta -\hat F_\zeta \hat B \hat R_\zeta = \hat F_\zeta -\hat F_\zeta \hat B \hat F_\zeta + \hat F_\zeta \hat B \hat F_\zeta \hat B \hat F_\zeta - \dots ;\\
&\hat F_\zeta \hat B \hat F_\zeta \upharpoonright \mathcal H  = F_\zeta^\frac{1}{2}  P_\zeta F_\zeta^\frac{1}{2},\;\; \|P_\zeta\|_{2 \rightarrow 2} \leq \sqrt{\delta}; \\
& \hat F_\zeta \hat B \hat F_\zeta \hat B \hat F_\zeta \upharpoonright \mathcal H = F_\zeta^\frac{1}{2}  H^*\, SH^*\, S  F_\zeta^\frac{1}{2}.
\end{align*}
Therefore,
\begin{align*}
& R_\zeta :=  \hat R_\zeta \upharpoonright \mathcal H =  F_\zeta^\frac{1}{2} (1+ P_\zeta)^{-1}  F_\zeta^\frac{1}{2} \tag{$p_2$};\\
& R_\zeta = F_\zeta -  F_\zeta^\frac{1}{2} H^* (1 + S H^*)^{-1} S  F_\zeta^\frac{1}{2} \tag{$p_2^\prime$};\\
& \|R_\zeta\|_{2\rightarrow 2} \leq |\zeta|^{-1} (1-\sqrt{\delta})^{-1} \tag{$p_3$}.
\end{align*}
Now, we employ Hille's theory of pseudo-resolvents. By $(p_1)$, $R_\zeta$ is a pseudo-resolvent on $\mathcal O.$ By $(p_2),$ the common null set of $\{R_\zeta \mid \zeta \in \mathcal O \}$ is $\{0 \}.$ Also, from $(p_2^\prime)$ it follows that $\nu R_\nu \overset{s} \rightarrow 1$ as $\nu \uparrow \infty.$ (Indeed, since $\|\nu R_\nu\|_{2 \rightarrow 2}$ is bounded in $\nu$, it suffices to prove $\nu R_\nu f \rightarrow f$ for $f \in D(A^\frac{1}{2}).$ In view of $(p_2^\prime)$, we only have to prove that $\nu M_\nu f \equiv \nu  F_\nu^\frac{1}{2} H_\nu^* (1 + S_\nu H_\nu^*)^{-1} S_\nu  F_\nu^\frac{1}{2} f \rightarrow 0$. Since
$\|S_\nu F_\nu^\frac{1}{2} f\|_2 \leq \|F_\nu (\lambda+ A)^\frac{1}{2} f\|_2 \leq \nu^{-1} \|(\lambda+A)^\frac{1}{2} f\|_2,$ it is seen that $\| M_\nu f\|_2 \leq \nu^{-\frac{3}{2}}\|(\lambda+A)^\frac{1}{2} f\|_2.$) 
Therefore, the range of   $R_\zeta$ is dense in $\mathcal H$, and $R_\zeta$ is the resolvent of a densely defined closed operator $\Lambda \equiv \Lambda(a,b)$ (Appendix \ref{hille_theory_sect}, Theorem \ref{hille_thm1}).

Finally, by ($p_3$), $-\Lambda$ is the generator of a quasi bounded holomorphic semigroup.

\begin{remarks}
1. The above construction of $\Lambda(a,b)$ works for $a \in (H_1), \; A=A_D, A_N \text{ or } A_{iD}.$ The use of $(p_2)$ leads to the convergence $(1 + A)^\frac{1}{2} \big(R_\zeta(b_n)-R_\zeta(b)\big) \overset{s}\rightarrow 0$ claimed in section \ref{Lr_sect} almost immediately.

2. If $a \in (H_u), \; A=A_D,$ then $\mathcal H_+ =W^{1,2}_0(\Omega), \; \mathcal H_-= W^{-1,2}(\Omega).$ If $a\in(H_u), \; A=A_N,$ then $\mathcal H_\pm =W^{\pm 1,2}(\Omega).$ \qed
\end{remarks}

\textbf{2.~}In view of $(p_3)$ it suffices to prove the convergence $R_\zeta(a_n, b_n) \overset{s}\rightarrow R_\zeta(a, b)$ for $\zeta =\lambda_\delta.$ For brevity, set $F_n= F_\lambda(a_n), \; F= F_\lambda (a)$.

Note, that $F_n \overset{s}\rightarrow F$ if and only if $\langle (F_n - F)f,F_nf\rangle \rightarrow 0$ and $\langle (F - F_n)f,F f\rangle \rightarrow 0$  $(f \in \mathcal H).$ In turn, 
$$|\langle (F_n - F)f,F_nf\rangle| = |\langle \nabla F f \cdot (a_n - a)\cdot \nabla F_n^2 \bar f\rangle| \leq \|\nabla F_n^2 \bar f\|_2 \|(a-a_n)\cdot\nabla F f\|_2,$$ 
$$\|\nabla F_n^2 f\|_2 \leq \sigma^{-\frac{1}{2}} \lambda^{-\frac{3}{2}} \|f\|_2$$ 
and $a_n \rightarrow a$ strongly in $L^2_\loc.$ Thus $F_n \overset{s}\rightarrow F.$

Next, employing the formula
\[
F_\lambda^\frac{1}{2}(a) = \pi^{-1} \int_0^\infty t^{-\frac{1}{2}} F_{t+\lambda}(a) d t,
\]
it is seen that
\[
\int_T^\infty t^{-\frac{1}{2}} \| \big(F_{t+\lambda}(a)-F_{t+\lambda}(a_n) \big)f \|_2 dt \leq 2 T^{-\frac{1}{2}} \|f\|_2^2
\]
and
\[
\lim_n \int_0^T t^{-\frac{1}{2}} \| \big(F_{t+\lambda}(a)-F_{t+\lambda}(a_n) \big)f \|_2 dt = 0.
\]
Thus $F_n^\frac{1}{2} \overset{s}\rightarrow F^\frac{1}{2}$ and, since $\nabla$ is a closed operator in $\mathcal H,$ $\nabla F_n^\frac{1}{2} \overset{s}\rightarrow \nabla F^\frac{1}{2}.$ In turn, the latter and the fact that $b_n \rightarrow b$ strongly in $L^2_\loc$ yield $H_n^*S_n \overset{s}\rightarrow H^*S.$ $\big($Indeed, $s \mbox{-} L^2 \mbox{-} H^*_n S_n = s \mbox{-} L^2 \mbox{-} F^\frac{1}{2}_n b_n \cdot \nabla F^\frac{1}{2}.$ Therefore, it suffices to establish the convergence $\|F^\frac{1}{2}_n b_n \cdot \nabla F^\frac{1}{2}f - H^* \frac{b}{b_a} \cdot \nabla F^\frac{1}{2} f\|_2 \rightarrow 0,$ or $\|F^\frac{1}{2}_n (b_n \cdot \varphi) - H^* (\frac{b}{b_a} \cdot \varphi) \|_2 \rightarrow 0,$ only for all $\varphi \in [C_c^\infty]^d.$ For such $\varphi$, we have $F^\frac{1}{2}_n (b_n \cdot \varphi) - H^*( \frac{b}{b_a} \cdot \varphi) = F_n^\frac{1}{2} (b_n \cdot \varphi) - F^\frac{1}{2} (b \cdot \varphi) \overset{s} \rightarrow 0.\big)$ Now the convergence $R_\zeta(a_n, b_n) \overset{s}\rightarrow R_\zeta(a, b)$ easily follows. The theorem is proved for $r = 2,$ and hence for all $r \in I^o_c.$
\end{proof}

\subsection{Non-divergence form operators} 
\label{nondiv_sect}

The following theorem is a by-product of Theorem \ref{thm:markeast} and Theorem \ref{thm:markfeast}.

\begin{theorem}
\label{thm:markbeast}
 Set $b := (\nabla a), (\nabla a)_k = \sum_{i=1}^d (\nabla_i a_{ik}),$ and $b^2_a = b\cdot a^{-1} \cdot b.$
 
 {\rm(\textit{i})} If $a \in (H_1), \; b \in \mathbf{F}_\delta(A)$ for some $\delta < 4$ and also $b_a^2 \in L^1 + L^\infty$ if $1 \leq \delta < 4,$ then $a \cdot \nabla^2$ has an operator realization $-\Lambda_r(a,b)$ in $L^r$ for every $r \in I^o_c$ as the generator of the positivity preserving, $L^r$ quasi contraction, $L^\infty$ contraction $C_0$ semigroup $e^{-t \Lambda_r(a,b)} = s\mbox{-}L^r\mbox{-}\lim_n e^{-t \Lambda_r(a,b_n)}$. 

{\rm(\textit{ii})} If $a, a_n \in (H_u)$, $n=1, 2, \dots$, $(\nabla a) \in \mathbf F_\delta(A)$, $(\nabla a_n) \in \mathbf F_\delta(A_n)$ {\rm($A_n \equiv A(a_n)$)} for some $\delta < 1$ and all $n$, and if 
$$
a_n \rightarrow a \text{ strongly in $[L^2_{\loc}]^{d \times d}$}, \quad (\nabla a_n) \rightarrow (\nabla a) \text{ strongly in $[L^2_\loc]^d$},
$$
 then, for every $r\in I^o_c,$ 
\[
e^{-t \Lambda_r(a,b)} = s\mbox{-}L^r\mbox{-}\lim_n e^{-t \Lambda_r(a_n,b_n)}, \qquad b=(\nabla a), \; b_n = (\nabla a_n).
\]

{\rm(\textit{iii})} If $\Omega = \mathbb R^d,$ $a \in (H_u)$, $b := (\nabla a) \in \mathbf F_{\delta_1} \equiv \mathbf F_{\delta_1}(-\Delta)$ with $\delta_1 = \sigma^2 \delta,$ $\delta < 1,$ and $|b| \in L^2+L^\infty,$ then, for every $r\in I^o_c,$
\[
e^{-t \Lambda_r(a,b)} = s\mbox{-}L^r\mbox{-}\lim_n e^{-t \Lambda_r(a_n,b_n)}, \qquad a_n := e^\frac{\Delta}{n} a, \;  b_n := (\nabla a_n).
\]

{\rm(\textit{iv})} If $\Omega = \mathbb R^d,$ $a$ is a uniformly elliptic diagonal matrix, $b := (\nabla a) \in \mathbf F_{\delta}(A)$ with $\delta < 1$, $|b| \in L^2+L^\infty,$ then, for each $r\in I^o_c,$
\[
e^{-t \Lambda_r(a,b)} = s\mbox{-}L^r\mbox{-}\lim_n e^{-t \Lambda_r(a_n,b_n)}, \qquad a_n := e^\frac{\Delta}{n}a, \; b_n := (\nabla a_n).
\]
\end{theorem}
\begin{proof}
The claimed convergence in (\textit{i}) (respectively, (\textit{ii})) is a direct consequence of Theorem \ref{thm:markeast} (Theorem \ref{thm:markfeast}).

The important thing in (\textit{iii}) is the fact that $b_n \in \mathbf{F}_\delta(A_n), \; 0<\delta < 1,$  uniformly in  $n,$ and so $\Lambda_r(a_n, b_n)$ are well defined for $r \in I^o_c.$

1. Set $E_n f := e^{\frac{\Delta}{n}} f.$ Alternatively, we may set $E_nf := \gamma_n * f, \gamma_n$ denotes the K. Friedrichs mollifier. Note the following elementary pointwise inequalities (below $b \cdot b=|b|^2=:b^2$)
\begin{align*}
(\nabla E_n a)^2 & \leq E_n (\nabla a)^2.\\
|E_n (f g)|^2 & \leq (E_n |f|^2) E_n |g|^2, \;\;\; f, g \in L^2 + L^\infty.
\end{align*}
Clearly, $b = (\nabla a) \in \mathbf F_{\delta_1} \Rightarrow b \in \mathbf F_\delta (A)$ with $c(\delta) = \frac{c_1(\delta_1)}{\sigma}.$ Thus we only need to show that $b_n \in \mathbf F_{\delta_1}$ in order to conclude that $b_n \in \mathbf F_\delta (A_n).$ Set $|f|_\varepsilon := |f| + \varepsilon e^{-x^2}.$ We have for $f \in W^{1,2},$
\begin{align*}
\|(\nabla E_n a) f \|_2^2 & = \|(E_n \nabla a)^2 |f|^2 \|_1 = \lim_{\varepsilon \downarrow 0} \|(E_n \nabla a)^2 |f|_\varepsilon^2 \|_1,\\
\|\big( E_n (\nabla a)^2\big) |f|_\varepsilon^2 \|_1 & = \|(\nabla a) \sqrt{E_n |f|_\varepsilon^2}\|_2^2\\
&  \leq \delta_1 \| \nabla \sqrt{E_n |f|_\varepsilon^2}\|_2^2 + c_1(\delta_1) \| E_n |f|_\varepsilon^2 \|_1 \\
& = \delta_1 \bigg \| \frac{E_n(|f|_\varepsilon\nabla|f|_\varepsilon)}{\sqrt{E_n(|f|_\varepsilon^2)}} \bigg \|_2^2 + c_1(\delta_1) \| E_n |f|_\varepsilon^2 \|_1 \\
&\leq \delta_1 \|E_n(\nabla|f|_\varepsilon)^2 \|_1 + c_1(\delta_1) \| E_n |f|_\varepsilon^2 \|_1 \\
& \leq \delta_1 \|\nabla |f|_\varepsilon \|_2^2 + c_1(\delta_1) \| |f|_\varepsilon \|_2^2.
\end{align*}
Since
\[
\lim_{ \varepsilon \downarrow 0} \big(\delta_1 \|\nabla |f|_\varepsilon \|_2^2 + c_1(\delta_1) \| |f|_\varepsilon \|_2^2 \big) = \delta_1 \|\nabla |f| \|_2^2 + c_1(\delta_1) \| f \|_2^2 \leq \delta_1 \|\nabla f \|_2^2 + c_1(\delta_1) \| f \|_2^2,
\]
we have proved that $b_n \in \mathbf F_{\delta_1}.$ 

2.~Now we claim that under assumptions on $a$ in $(iv)$
\[
b \in \mathbf{F}_\delta (A) \Rightarrow E_n b \in \mathbf F_\delta (A_n).
\]
Only for simplicity we treat the special case: $a = \kappa^2 I$, with $b = (\nabla a) = 2 \kappa \nabla \kappa \in \mathbf{F}_\delta(A)$. Since $b_a^2 = 4 (\nabla \kappa)^2,$ the assumption $b \in \mathbf{F}_\delta (A)$ means that 
\[
4 \langle (\nabla \kappa)^2 |f|^2 \rangle \leq \delta \| \kappa \nabla f \|_2^2 + c(\delta)\|f\|_2^2 \quad \quad f \in D(A^\frac{1}{2}).
\]
Set $a_n := E_n a.$ Then $(\nabla a_n) = 2 E_n (\kappa \nabla \kappa), \; (\nabla a_n) \cdot a_n^{-1} \cdot (\nabla a_n) = \frac{4 |E_n(\kappa \nabla \kappa)|^2}{E_n \kappa^2}.$ Note that
\[
|E_n(\kappa \nabla \kappa)|^2 \leq (E_n \kappa^2) E_n |\nabla \kappa|^2,
\] 
and so
\[
\langle (\nabla a_n) \cdot a_n^{-1} \cdot (\nabla a_n), |f|^2 \rangle\leq 4 \langle E_n |\nabla \kappa|^2, |f|^2 \rangle = 4 \langle |\nabla \kappa|^2, E_n |f|^2 \rangle.
\]
But $4 \langle |\nabla \kappa|^2, E_n |f|_\varepsilon^2 \rangle \leq \delta \|\kappa \nabla \sqrt{E_n |f|_\varepsilon^2} \|_2^2 + c(\delta) \| \sqrt{E_n |f|_\varepsilon^2} \|_2^2= \delta \big \langle \frac{\kappa^2 |\nabla E_n |f|_\varepsilon^2 |^2}{4 E_n |f|_\varepsilon^2} \rangle + c(\delta) \langle E_n |f|_\varepsilon^2 \rangle$ and $\frac{|\nabla E_n |f|_\varepsilon^2 |^2}{4 E_n |f|_\varepsilon^2} = \frac{|E_n (|f|_\varepsilon \nabla |f|_\varepsilon)|^2}{E_n |f|_\varepsilon^2} \leq E_n|\nabla |f|_\varepsilon|^2.$ Thus,
\[
4 \langle (\nabla \kappa)^2, E_n |f|^2 \rangle \leq \delta \langle \kappa^2 E_n |\nabla |f||^2 \rangle + c(\delta) \langle |f|^2 \rangle = \delta \langle (E_n \kappa^2)  |\nabla |f||^2 \rangle + c(\delta) \langle |f|^2 \rangle
\]
and
\[
\langle (\nabla a_n) \cdot a_n^{-1} \cdot (\nabla a_n), |f|^2 \rangle\leq \delta \| A_n^\frac{1}{2} |f| \|_2^2 + c(\delta) \|f\|_2^2 \leq \delta \| A_n^\frac{1}{2} f \|_2^2 + c(\delta) \|f\|_2^2.
\]
In other words $E_n b \in \mathbf F_\delta (A_n)$ as required. 

3.~By Theorem \ref{thm:markfeast}, steps 1 and 2 entail the claimed convergences in (\textit{iii}), (\textit{iv}).
\end{proof}

\begin{remarks}
1.  If  $a_n$, $\nabla a_n$ in (\textit{ii}) are smooth (e.g.\,in the assumptions of (\textit{iii}) or (\textit{iv})), then by the Krylov-Safonov a priori H\"older continuity of $(\lambda + \Lambda(a_n,b_n))^{-1} f, \; f \in C_\infty$ \cite[sect.\,4.2]{Kr}, for every $ r \geq d$, there exists a constant $0 <\alpha < 1$ such that
\[
(\lambda + \Lambda_r)^{-1} L^r \cap L^\infty \subset C_\infty^{0,\alpha}, \qquad \lambda > \omega_r.
\]

2. Let $\{\varepsilon_n \}$ be a sequence such that $\varepsilon_n \downarrow 0$ as $n\uparrow \infty.$ The proof of (\textit{iii}) yields the following.

 For any $b:\mathbb R^d \rightarrow \mathbb R^d, \; |b|\in L^2_\loc$ define $b_n \equiv E_n b := \gamma_{\varepsilon_n} * (b \mathbf 1_{B(0,n)})$ and choose $\{\varepsilon_n \}$ such that $b_n \rightarrow b$ strongly in $[L^2_\loc]^d$ as $n \uparrow \infty.$  Fix $\delta < \infty.$ Then
\[
b \in \mathbf F_\delta \Rightarrow \big (b_n \in \mathbf F_\delta, \; b_n \rightarrow b \text{ strongly in } [L^2_\loc]^d, \text{ and hence } e^{-t \Lambda(b_n)} \overset{s}\rightarrow e^{-t\Lambda(b)}\big). 
\]
3. The same arguments used in step 2 yield the following:

 If $a \in (H_1), \; a_{il} = \kappa_i^2 \delta_{il}, \; |\kappa_i| + |\nabla_i \kappa_i| \in L^2 + L^\infty ,\; i,l = 1,2,\dots, d,$ then 
\[
b \in \mathbf F_\delta(A) \Rightarrow b_n \in \mathbf F_\delta(A_n) \quad (b = (\nabla a), \; b_n = e^\frac{\Delta}{n} b).
\]
\end{remarks}

\begin{example*}In $\mathbb R^d, d \geq 3$ consider the matrix
\[
a(x) = I + c |x|^{-2} x \otimes x, \quad a_{ik}(x) = \delta_{ik} + c |x|^{-2} x_i x_k \text{ with }c = \frac{d-1}{1-\alpha}-1, \; \alpha < 1.
\]
Thus $a(x)$ is strictly positive, $(\nabla a)=(d-1) c |x|^{-2} x, \; a^{-1} = I - \frac{c}{c+1} |x|^{-2} x \otimes x$ and $b_a^2 = \frac{[(d-1) c ]^2}{c+1} |x|^{-2}.$ The following Hardy type inequality (with the best possible constant) will be proved below:
\[
(c+1) \frac{(d-2)^2}{4} \| |x|^{-1} h \|_2^2 \leq \langle \nabla h \cdot a \cdot \nabla \bar{h} \rangle \quad \quad ( h \in W^{1,2}(\mathbb R^d) ).\\
\tag{$\star$}
\]
$(\star)$ implies that $b \in \mathbf{F}_\delta (A)$ with $\delta = 4 \big( 1+\frac{\alpha}{d-2} \big )^2$ and $c(\delta)=0.$ In particular, $\delta < 4$ if and only if $ \alpha \in ]-2(d-2), 0[$ or equivalently $c \in ] - \frac{1}{2 + \frac{1}{d-2}}, 0 [ \; \cup \; ]0, d-2 [.$ Armed with $(\star)$  and Theorem \ref{thm:markeast}, one can reconsider the conclusions in \cite[Ch.\,I, sect.\,3, Example 4]{LSU}. 
\end{example*}

Consider the following problem in $L^p(\mathbb R^d, \mathcal L^d)$, $p > 0$, $d \geq 3$:
$$a \cdot \nabla^2 u = 0, \quad u(x) \upharpoonright \{|x|\geq 1\} =0.$$
If $\alpha \neq 0,$ then the problem has two solutions $u_1 = 0$ and $u_2 = |x|^\alpha -1.$ If $\delta < 4,$ then $\alpha < 0$ and the unbounded solution to $\Lambda_p u = 0, \; p > \frac{2}{2 - \sqrt{\delta}},$ is inadmissible according to Theorem \ref{thm:markeast}. ($\Lambda_p \supset -\nabla \cdot a \cdot \nabla + (\nabla a) \cdot \nabla$ only formally equals to $- a \cdot \nabla^2$). If $\alpha > 0$ (so that $\delta > 4 ),$ then the problem has two (bounded) solutions.

\smallskip

\noindent \textbf{Conclusion.} \textit{The condition $\delta < 4$ of Theorem \ref{thm:markeast} can not be substantially strengthened.}
 
\medskip

\begin{proof}[Proof of $(\star)$] 
Let $c > 0.$ Since $\langle \phi, x \cdot \nabla \phi \rangle = - \frac{d}{2} \langle \phi, \phi \rangle, \; \phi \in C_c^\infty,$ we have
\[
\langle \phi, -\nabla \cdot (a-1) \cdot \nabla \phi \rangle = c \big(\| x \cdot \nabla (|x|^{-1} \phi) \|_2^2 - (d-1) \| |x|^{-1} \phi \|_2^2 \big).
\]
Next, the following inequality (with the sharp constant) is valid:
\[
\| x \cdot \nabla f \|_2 \geq \frac{d}{2} \|f \|_2, \quad \quad ( f \in D(\mathcal{D}) ), \\
\tag{$\blacktriangledown$}
\]
where $\mathcal{D} \upharpoonright C_c^\infty = \frac{\sqrt{-1}}{2} ( x \cdot \nabla + \nabla \cdot x ).$ 

Indeed, since the operator $\mathcal{D} = (\mathcal{D} \upharpoonright C_c^\infty )^{{\rm clos}}_{L^2 \rightarrow L^2}$ is self-adjoint, $\|(\mathcal{D} - \lambda)^{-1} \| \leq \frac{1}{|\Imag \lambda |}$ for $\Real \lambda= 0.$ Therefore
\[
\| x \cdot \nabla f \|_2 = \| \frac{1}{2}( x \cdot \nabla + \nabla \cdot x -d )f \|_2 = \|(\mathcal{D} - \sqrt{- 1} \; \frac{d}{2}) f \|_2 \geq \frac{d}{2} \|f \|_2, \quad \quad (f \in C_c^\infty )
\]
and ($\blacktriangledown$) is proved. But then
\[
\langle \phi, -\nabla \cdot (a-1) \cdot \nabla \phi \rangle \geq c \frac{(d-2)^2}{4} \| |x|^{-1} \phi \|_2^2 \quad \quad ( \phi \in C^\infty_c ).
\]

$(\star)$ follows now from the equality $\langle \phi, - \nabla \cdot a \cdot \nabla \phi \rangle =  \langle \phi, - \nabla \cdot (a-1) \cdot \nabla \phi \rangle + \langle \phi, - \Delta \phi \rangle$ and Hardy's inequality $\langle \phi, -\Delta \phi \rangle \geq \frac{(d-2)^2}{4} \| |x|^{-1} \phi \|_2^2.$

Finally, the obvious inequality $(1+c) \langle \phi, -\Delta \phi \rangle \geq \langle \phi, - \nabla \cdot a \cdot\nabla \phi \rangle$ clearly shows that the constant in $(\star)$ is sharp. 

If $-1 < c \leq 0,$  $(\star)$ is a trivial consequence of Hardy's inequality.
\end{proof}

\begin{remark*}The Krylov-Safonov a priori estimates yield 
the uniqueness of a ``good"
 solution to $-a \cdot \nabla^2=f$ in $L^d$ provided that $a \in (H_u)$ 
is continuous outside of a ``sufficiently small'' set \cite{CEF}.

The assumption $(\nabla a) \in \mathbf{F}_\delta(A)$ \textit{ does not guarantee } $W^{2,r}$ estimates on $(\zeta + \Lambda_r(a,\nabla a))^{-1}L^r$ for some $r>(d-2)\vee 2.$ The same is true even for $(\nabla a) \in \mathbf F_0(A).$ (See also remark 5 in the next section.)
\end{remark*}

\subsection{$W^{1,s}$-estimates on solutions to  $(\mu - \Delta + b \cdot \nabla ) u = f$, $b \in \mathbf{F}_\delta$}
\label{est_sect}

Let $d \geq 3,$ $L^p=L^p(\mathbb R^d, \mathcal L^d ), \; b : \mathbb R^d \rightarrow \mathbb R^d, \; b \in \mathbf F_\delta,$ $ 0 < \delta < 4.$ Define $b_n \equiv E_n b := \gamma_{\varepsilon_n} * (b \mathbf{1}_{B(0,n)}), \; n = 1,2, \dots$   

Set $\Lambda_p(b_n) := -\Delta + b_n \cdot \nabla,$ $D(\Lambda_p(b_n)) = W ^{2,p}(\mathbb{R}^d)$, $1 < p < \infty.$ Clearly, $-\Lambda_p(b_n)$ is the generator of a holomorphic semigroup in $L^p.$ According to Theorem \ref{thm:markeast} and the fact $b_n \in \mathbf F_\delta$ (see Remark 2 after the proof of Theorem \ref{thm:markbeast}), for each $p \in I_c =[\frac{2}{2-\sqrt{\delta}}, \infty [,$
\[
\|e^{-t\Lambda_p(b_n)} \|_{p \rightarrow p} \leq e^{\omega_p t}, \; \omega_p = \frac{c(\delta)}{2(p-1)},
\]
and by Theorem \ref{thm:markease}, for each $p \in I_c^o, \; s \mbox{-}L^p \mbox{-} \lim_n e^{-t\Lambda_p(b_n)}$ exists and determines the $C_0$ semigroup $e^{-t\Lambda_p};\ \Lambda_p \equiv \Lambda_p(b)$ is an operator realization of $-\Delta + b \cdot \nabla$ in $L^p.$

\begin{theorem}
\label{thm:markgrad}
Let $d \geq 3.$ Fix any $\delta \in \big]0, 1 \wedge \big(\frac{2}{d-2}\big)^2 \big[.$ Assume that $b \in \mathbf{F}_\delta.$ Let $q \in [2, \frac{2}{\sqrt{\delta}} \big[.$ Set $u = (\mu + \Lambda_q(b) )^{-1} f,$ where $f \in L^q,  \; \mu > \omega_q.$ Then there exist constants $\lambda_0 =\lambda_0(\delta, q)$ and $K_l=K_l(\delta, q), \; l=1, 2,$ such that, for all $\mu > \lambda_0 \vee \omega_q,$ 
\begin{align*}
\|\nabla u \|_q & \leq K_1 \big(\mu -\lambda_0 \big)^{-\frac{1}{2}} \|f\|_q ;\\
\|\nabla u \|_{qj} & \leq K_2 \big(\mu -\lambda_0 \big)^{\frac{1}{q}-\frac{1}{2}} \|f\|_q, \quad\quad j=\frac{d}{d-2}.
\end{align*}
In particular, $(\mu + \Lambda_q(b))^{-1}: L^q \rightarrow C^{0,1- \frac{d}{qj}}$ whenever $d \geq 4$, $q \in \big]d-2, \frac{2}{\sqrt{\delta}} \big[$ and $\mu > \omega_q.$ For $d=3$,  $(\mu + \Lambda_q(b))^{-1}: L^q \rightarrow C^{0,1- \frac{d}{qj}}$ whenever $q \in \big[2, \frac{2}{\sqrt{\delta}} \big[$, $\mu > \omega_q.$
\end{theorem}
\begin{proof}
1. Let $q \in \big] \frac{2}{2-\sqrt{\delta}}, \infty \big[.$ Clearly $ b_n \cdot \nabla \upharpoonright W^{2,q}$ is the Miyadera perturbation of $-\Delta:$ $\|b_n\cdot \nabla \;(\lambda_n -\Delta)^{-1} \|_{q \rightarrow q} < 1, \; \sqrt{\lambda_n} \geq m_{n,d}$ for some constant $m_{n,d}.$ Therefore, by the resolvent identity, $(\mu + \Lambda_q(b_n))^{-1} L^q \subset W^{2,q}, \; \mu > \omega_q.$

2. Let $q \in [ 2, \frac{2}{\sqrt{\delta}} [.$ Set $u_n = (\mu + \Lambda_q(b_n))^{-1} f , \; 0 \leq f \in C_c^\infty.$ We will use the following notations 
\[
w := \nabla u_n, \quad w_i :=\nabla_i u_n, \quad w_{ik} := \nabla_i w_k, \quad \Delta u_n = \nabla \cdot w \equiv \sum_{i=1}^d w_{ii}; \; 
\]
\[
\phi := - \nabla \cdot (w |w|^{q-2}) \equiv - \sum_{i=1}^d \nabla_i (w_i |w|^{q-2}).
\]
Note that $u_n \in \mathcal S,$ the L. Schwartz space of test functions. We have
$$
\langle(\mu-\Delta)u_n,\phi\rangle = - \langle b_n\cdot \nabla u_n,\phi\rangle + \langle f,\phi \rangle.$$
Since
\[
\langle -\Delta u_n, \phi \rangle = \big\langle -\Delta w, w |w|^{q-2} \big\rangle = \sum_{i,k=1}^d \big \langle w_{ik}, w_{ik}  |w|^{q-2} + (q-2) |w|^{q-3} w_k \nabla_i |w| \big \rangle
\]
and
\[
\phi = - |w|^{q-2} \Delta u_n - (q-2)|w|^{q-3} w \cdot \nabla |w| =: \phi_1 + \phi_2, \quad \phi_1 = - |w|^{q-2} (\mu u_n + b \cdot w -f), 
\]
we arrive at
\[
\mu \langle |w|^q \rangle + I_q + (q-2)J_q = \langle b_n \cdot w -f, |w|^{q-2} (\mu u_n + b_n \cdot w -f) + (q-2)|w|^{q-3} w \cdot \nabla |w| \rangle,
\tag{$\star$}
\]
where
\[
I_q = \sum_{i=1}^d \big\langle |\nabla w_i |^2, |w|^{q-2} \big\rangle, \;\; J_q = \big\langle |\nabla |w| |^2, |w|^{q-2} \big\rangle.
\]
Now we bound the terms from the RHS of $(\star)$ by $J_q, \; B_q := \langle |b_n \cdot w|^2 |w|^{q-2} \rangle, \; \|w\|_q^{q-2}, \; \|f\|^2_q$ and $\big(\frac{\mu}{\mu -\omega_q}\big )^2 \|f\|_q^2$ as follows.

1) $\langle b_n \cdot w, |w|^{q-2} \mu u_n \rangle  \leq \frac{\mu}{\mu-\omega_q} B_q^\frac{1}{2} \|w\|_q^\frac{q-2}{2}  \|f\|_q. \;\; ( \frac{2}{2-\sqrt{\delta}} < q \Rightarrow \|u_n\|_q \leq (\mu - \omega_q )^{-1} \|f\|_q$ ).

2) $\langle b_n \cdot w, |w|^{q-2} b_n \cdot w \rangle = B_q .$

3) $|\langle b_n \cdot w, |w|^{q-2} (-f) \rangle |\leq B_q^\frac{1}{2} \|w\|_q^\frac{q-2}{2} \|f\|_q .$

4) $(q-2) \langle b_n \cdot w, |w|^{q-3} w \cdot \nabla |w| \rangle \leq (q-2) B_q^\frac{1}{2} J_q^\frac{1}{2} .$

5) $\langle -f, |w|^{q-2} \mu u_n \rangle \leq 0.$

6) $\langle -f, |w|^{q-2} b_n \cdot w \rangle \leq B_q^\frac{1}{2} \|w\|_q^\frac{q-2}{2} \|f\|_q .$

7) $ \langle f, |w|^{q-2} f \rangle \leq \|w\|_q^{q-2} \|f\|_q^2 .$

8) $ (q-2) \langle -f, |w|^{q-3} w \cdot \nabla |w| \rangle \leq (q-2) J_q^\frac{1}{2}\|w\|_q^\frac{q-2}{2} \|f\|_q .$

(Below we will get rid of the terms arising in the RHS of 1),\,3),\,5)-8) using Young's inequality, at expense of increasing the coefficients $K_1$, $K_2$ of $\|f\|_q$.)

By means of 1)-8) we have ($\varepsilon, \varepsilon_0 > 0$)
\begin{align*}
\text{ RHS of } (\star) \leq &(q-2) \varepsilon J_q + (q-2) \big(\varepsilon_0 J_q + \frac{1}{4 \varepsilon_0}B_q \big)\\
 + & (1 +3 \varepsilon)B_q + \bigg(1  + \frac{q}{4\varepsilon} +\frac{1}{4 \varepsilon}\frac{\mu^2}{(\mu-\omega_q)^2}\bigg) \|w\|_q^{q-2} \|f\|_q^2.
\end{align*}
By $b_n \in \mathbf F_\delta,$
\[
B_q \leq \|b_n |w|^\frac{q}{2} \|_2^2 \leq \delta \| \nabla |w|^\frac{q}{2} \|_2^2 + c(\delta) \|w\|_q^q = \frac{\delta q^2}{4} J_q + c(\delta) \|w\|_q^q.
\]
Setting $\varepsilon_0 = \frac{q\sqrt{\delta}}{4}$ we have
\begin{align*}
\text{ RHS of } (\star) &\leq \bigg[ (q-2)\frac{q\sqrt{\delta}}{2} + \frac{\delta q^2}{4} + (q-2)\varepsilon + 3 \varepsilon \frac{\delta q^2}{4} \bigg] J_q \\ 
&+ \bigg(1+3\varepsilon + \frac{q-2}{q \sqrt{\delta}}\bigg)c(\delta) \|w\|_q^q + \bigg(1 + \frac{q}{4\varepsilon} +\frac{1}{4 \varepsilon}\frac{\mu^2}{(\mu-\omega_q)^2}\bigg) \|w\|_q^{q-2} \|f\|_q^2.
\end{align*}
Thus
\begin{align*}
&\mu \|w\|_q^q + I_q + \bigg[(q-2)(1-\frac{q\sqrt{\delta}}{2}) - \frac{\delta q^2}{4} \bigg]J_q - \bigg[(q-2)\varepsilon + 3 \varepsilon \frac{\delta q^2}{4} \bigg] J_q \\
& \leq \bigg(1+3\varepsilon +\frac{q-2}{q \sqrt{\delta}} \bigg)c(\delta)\|w\|_q^q + \bigg(1 + \frac{q}{4\varepsilon} +\frac{1}{4 \varepsilon}\frac{\mu^2}{(\mu-\omega_q)^2}\bigg) \|w\|_q^{q-2} \|f\|_q^2 .
\end{align*}
Since $I_q \geq J_q$ and $q-1 -(q-2)\frac{q \sqrt{\delta}}{2} - \frac{\delta q^2}{4} > 0$ due to\footnote{Set $\psi(q)=q-1-(q-2)q\frac{\sqrt{\delta}}{2}-q^2\frac{\delta}{4}, \;\; 2\leq q<\frac{2}{\sqrt{\delta}}$. Note that $\psi(2)=1-\delta>0$ and $\psi(2/\sqrt{\delta})=0$. 
Also $\psi$ increases on $[2,\frac{2}{\sqrt{\delta}}\frac{1+\sqrt{\delta}}{2+\sqrt{\delta}}]$ and decreases on $[\frac{2}{\sqrt{\delta}}\frac{1+\sqrt{\delta}}{2+\sqrt{\delta}},\frac{2}{\sqrt{\delta}}[$, and so $\psi>0$ on $[2,\frac{2}{\sqrt{\delta}}[$.} $q \in [2 , \frac{2}{\sqrt{\delta}}[$ we conclude that, for suitable $\varepsilon$, there are $C_1 =C_1(\varepsilon, \delta, q) >0$ and $C_2 = (\varepsilon, \delta, q, \omega_q) < \infty$ such that
\[
(\mu -\tilde{\omega}_\varepsilon)\|w\|_q^q +  C_1 J_q \leq C_2 \|w\|_q^{q-2} \|f\|_q^2,
\]
where $\tilde{\omega}_\varepsilon= \big(1+3\varepsilon + \frac{q-2}{q \sqrt{\delta}}\big)c(\delta).$

The latter, the Sobolev and the Young inequalities combined imply that
\[
\|\nabla u_n \|_{q j} \leq K_2 \big(\mu -\lambda_0 \big)^{\frac{1}{q}-\frac{1}{2}} \|f\|_q \text{ whenever } \mu > \lambda_0 := \tilde{\omega}_0 > \omega_q .
\]
It is also seen that $\|\nabla u_n\|_q \leq C_2^\frac{1}{2}(\mu -\lambda_0)^{-\frac{1}{2}} \|f\|_q.$ Since $u_n \rightarrow u$ strongly in $L_+^q$ and hence in $L_+^{qj},$  and since $(C_c^\infty)_+$ is dense in $L_+^{qj},$ we have $u \in W^{1,q j}$, and the inequality $\|\nabla u \|_{qj} \leq K_2 \big(\mu -\lambda_0 \big)^{\frac{1}{q}-\frac{1}{2}} \|f\|_q$ holds for all $f \in L_+^q.$

We have established the assertions of the theorem for all $f \in L^q_+,$ and hence for all $f \in L^q$ (with $K_l \rightarrow 4 K_l, l= 1, 2).$
\end{proof}

\begin{remarks}

1. There is an obvious analogue of Theorem \ref{thm:markgrad} for $\Omega \subset \mathbb R^d$ with the additional assumption $\partial \Omega \in C^{0,1}$ in case of $\Delta=\Delta_N.$

\smallskip

2. Only for $d=3$ and $q=2$ the result of Theorem \ref{thm:markgrad} is ``trivial". Namely, the assumption $b \in \mathbf F_\delta, \; \delta < 1,$ implies that $b \cdot \nabla$ is Rellich's perturbation of $-\Delta$ in $L^2$, and so $-\Lambda_2(b) = \Delta - b\cdot \nabla$ of domain $W^{2,2}$. 
(Indeed, define $T=b\cdot\nabla(\mu-\Delta)^{-1}$, $\mu \geq \lambda$, and note that 
$$
\|T\|_{2 \rightarrow 2} \leq \||b|(\mu-\Delta)^{-\frac{1}{2}}\|_{2 \rightarrow 2}\|\nabla(\mu-\Delta)^{-\frac{1}{2}}\|_{2 \rightarrow 2} \leq \sqrt{\delta}.
$$
Thus, by Theorem \ref{thm:markfeast}, the Neumann series for $(\mu+\Lambda_2(b))^{-1}=(\mu-\Delta)^{-1}(1+T)^{-1}$.)

Hence, for $\Real\zeta \geq \lambda, \; (\zeta + \Lambda_2(b))^{-1} : L^2 \rightarrow W^{1,2j}, \; 2j= 6 =2 d.$ However, already for $d=4,$ $W^{2,2} \subset W^{1,2j}, \; 2 j = d$ but not crucial $2 j > d$.

\smallskip

3. Under the same assumptions on $\delta$ as in Theorem \ref{thm:markgrad}, a stronger regularity of the elements of $D(\Lambda_q(b))$ has been recently established in \cite[Theorem 1]{Ki4}. Namely, if $b \in \mathbf{F}_\delta$, $\delta<1$, then for every $q \in \big[2,\frac{2}{\sqrt{\delta}}[$ the formal differential expression $\Delta - b \cdot \nabla$ has an operator realization $-\Lambda_q(b)$ on $L^q$ as the generator  of a positivity preserving, $L^\infty$ contraction, quasi contraction $C_0$ semigroup such that
 $D(\Lambda_q)\subset \mathcal W^{1+\frac{2}{s},q}, s>q.$ In particular, by the Sobolev Embedding Theorem, if $\delta<1\wedge\big(\frac{2}{d-2}\bigr)^2$, then there exists $q>d-2$ such that $D\bigl(\Lambda_q(b)\bigr) \subset C^{0,\gamma}$, $\gamma<1-\frac{d-2}{q}$.
 
\smallskip

4. Let $|b| \in L^d + L^\infty, d\geq 3.$ Then, by the Sobolev Embedding Theorem, for any $q \in ]\frac{d}{2}, d[,$ there exists a constant $\lambda_{d,b} > 0,$ such that
\[
(\mu + \Lambda_q(b))^{-1} : L^q \rightarrow W^{2,q} \subset C^{0, 2-\frac{d}{q}} \text{ for all } \mu > \lambda_{d,b}.
\]
In particular,
\[
(\mu + \Lambda_q(b))^{-1} : L^q \cap L^d \rightarrow C^{0, \alpha} \text{ for any } \alpha < 1.
\]
By Theorem \ref{thm:markgrad}, the last embedding also holds for $b \in \mathbf F_0 = \bigcap_{\delta > 0} \mathbf F_\delta.$

\smallskip

5.~In fact, for $b \in [L^{d,\infty}]^d$ ($\subsetneq \mathbf{F}_\delta$ with $\delta=c_d \|b\|_{d,\infty}$, see Example \ref{ex0} above), one has the following characterization of $W^{2,q}$ smoothness of $u=(\mu+\Lambda_q(b))^{-1}f$, $f \in L^q$, $1<q<d$. 
Using that
$$
|\nabla (-\Delta)^{-1}(x,y)| \leq m_d^* (-\Delta)^{-\frac{1}{2}}(x,y), \quad x,y \in \mathbb R^d \quad \text{ with } m_d^*:=\frac{d-2}{2} \sqrt{\pi}\frac{\Gamma\bigl(\frac{d-2}{2}\bigr)}{\Gamma\bigl(\frac{d-1}{2}\bigr)},
$$
we have
$$
\|b \cdot \nabla (\zeta - \Delta)^{-1}\|_{q \rightarrow q} \leq \|b \cdot \nabla (-\Delta)^{-1}\|_{q \rightarrow q}\|\Delta (\zeta - \Delta)^{-1} \|_{q \rightarrow q} \leq 2 m_d^* \||b|(-\Delta)^{-\frac{1}{2}}\|_{q \rightarrow q}, \quad \Real \zeta>0.
$$
Let $\Omega_d=\pi^{\frac{d}{2}}\Gamma(\frac{d}{2}+1)$ denote the volume of the unit ball in $\mathbb R^d$.
Then
\begin{align*}
&\||b|(-\Delta)^{-\frac{1}{2}}\|_{q \rightarrow q} \qquad (\text{we apply \cite[Prop.\,2.5]{KPS}})\\ 
& \leq \|b\|_{d,\infty}\Omega_d^{-\frac{1}{d}}\||x|^{-1}(-\Delta)^{-\frac{1}{2}}\|_{q \rightarrow q} \qquad (\text{we apply \cite[Lemma 2.7]{KPS}}) \\
&= \|b\|_{d,\infty}\Omega_d^{-\frac{1}{d}}R_{d,q}, \qquad R_{d,q}:=\frac{1}{2}\frac{\Gamma\big(\frac{d}{2q}-\frac{1}{2}\big)\Gamma\big(\frac{d}{2}-\frac{d}{2q}\big)}{\Gamma\big(\frac{d+1}{2}-\frac{d}{2q}\big)\Gamma\big(\frac{d}{2q}\big)}.
\end{align*}
We conclude that if $\|b\|_{d,\infty}<\frac{1}{2}(m_d^*\Omega_d^{-\frac{1}{d}}R_{d,q})^{-1}$, then $b \cdot \nabla$ is Rellich's perturbation of $-\Delta$, and so $\Lambda_q(b)=-\Delta + b \cdot \nabla$, $D(\Lambda_q(b)) = W^{2,q}$, generates a holomorphic semigroup in $L^q$. (Note that $q=d$ is unavailable for $b \in [L^{d,\infty}]^d$.)

\smallskip

6. Let $d \geq 3$ and $b(x) = c |x|^{-2}x, \; |c| < 1.$ Clearly, $b \in \mathbf F_\delta$ with $\delta = \frac{4 c^2}{(d-2)^2}$ satisfies the assumptions of Theorem \ref{thm:markgrad} if $d \geq 4.$ For $d=3$ however $|c|$ has to be strictly less than $\frac{1}{2}.$

For this vector field, the proof of Theorem \ref{thm:markgrad} can be modified to take advantage of the fact that ${\rm div\,}b$ exists and is a form-bounded potential:
\end{remarks}

\begin{corollary}
\label{hardy_drift_thm1}
Let $\Omega=\mathbb R^d, \;d \geq 3$. Let $b(x):=-c|x|^{-2}x,$ $c=\frac{d-2}{2}\sqrt{\delta}$, $\delta<\infty$. Then

{\rm 1)} The interval of contraction solvability for $\Lambda_q(b)\supset -\Delta +b\cdot\nabla$ is $]1,\infty[$.

{\rm 2)} $d\geq 4.$ Let $u=(\mu + \Lambda_q(b))^{-1}f$, $\mu>0,$ $f \in L^q.$ Then, for any $\delta <\infty,$ there exist $q>d-2$ and constants $K_l=K_l(\delta, q), \; l=1, 2,$ such that 
\begin{align*}
\|\nabla u \|_q & \leq K_1 \mu^{-\frac{1}{2}} \|f\|_q ;\\
\|\nabla u \|_{qj} & \leq K_2 \mu^{\frac{1}{q}-\frac{1}{2}} \|f\|_q.
\end{align*}

{\rm 3)} $d=3.$ Then the assertion of 2) holds for any $q \geq 2$ and  $\sqrt{\delta}< \frac{2}{q}$.

In particular, both in 2) and 3), $(\mu + \Lambda_q(b))^{-1}: L^q \rightarrow C^{0,1- \frac{d}{qj}}.$ 
\end{corollary}

\begin{remark*}
We need $q \geq 2$ due to our choice of the test function $\phi$ in the proof below. 
\end{remark*}

\begin{proof}Set $|x|_\varepsilon^2:=|x|^2+\varepsilon, \;\varepsilon>0,$ $b_\varepsilon(x):=-c|x|_\varepsilon^{-2}x$, $u\equiv u_\varepsilon=(\mu + \Lambda_q(b_\varepsilon))^{-1}f$, $w:=\nabla u$. We follow the proof of Theorem \ref{thm:markgrad}.
We have
$$
\mu \langle |w|^q\rangle + I_q + (q-2)J_q = c\langle |x|_\varepsilon^{-2} x\cdot w, \phi \rangle + \langle f, \phi\rangle,
$$
where $\phi:=-\nabla \cdot (w|w|^{q-2})$.
Integrating by parts, we have
$$
\langle |x|_\varepsilon^{-2} x\cdot w, \phi\rangle = Z + \langle |x|_\varepsilon^{-2} x\cdot \nabla |w|, |w|^{q-1}\rangle,
$$
where
$$
Z:=\langle |x|_\varepsilon^{-2} |w|^q\rangle - 2\langle |x|_\varepsilon^{-4}(x\cdot w)^2|w|^{q-2}\rangle, 
$$
and so
$$
\langle |x|_\varepsilon^{-2} x\cdot w, \phi\rangle \leq  \langle |x|_\varepsilon^{-2}|w|^{q}\rangle + \langle  |x|_\varepsilon^{-2} x\cdot \nabla |w|, |w|^{q-1}\rangle,
$$
we obtain the inequality
\[
\mu \langle |w|^q\rangle + I_q + (q-2)J_q \leq  c\langle |x|_\varepsilon^{-2}|w|^{q}\rangle + c\langle |x|_\varepsilon^{-2} x\cdot \nabla |w|, |w|^{q-1}\rangle + \langle f, \phi \rangle.
\]
Noticing that
\[
\langle |x|_\varepsilon^{-2} x\cdot \nabla |w|, |w|^{q-1}\rangle = -\frac{d-2}{q}\langle |x|_\varepsilon^{-2} |w|^q \rangle - \frac{2\varepsilon}{q}\langle |x|_\varepsilon^{-4} |w|^q \rangle, 
\]
$I_q \geq J_q$ and
\[
\langle |x|_\varepsilon^{-2} |w|^q \rangle =\langle |x|^{-2} |w|^q \rangle -\varepsilon \langle |x|_\varepsilon^{-2}|x|^{-2} |w|^q \rangle, 
\]
we have
\begin{align*}
\mu \langle |w|^q\rangle + (q-1)J_q &\leq  c\bigg(1-\frac{d-2}{q}\bigg)\langle |x|^{-2}|w|^{q}\rangle - c\bigg(1-\frac{d-2}{q}\bigg)\varepsilon\langle |x|_\varepsilon^{-2}|x |^{-2} |w|^q\rangle + \langle f, \phi \rangle\\
&\biggl(\text{here we are using Hardy's inequality  $\langle |x|^{-2}|w|^q\rangle \leq \frac{q^2}{(d-2)^2}J_q$}\biggr)\\
&\leq c\bigg(1-\frac{d-2}{q}\bigg)\frac{q^2}{(d-2)^2} J_q - c\bigg(1-\frac{d-2}{q}\bigg)\varepsilon\langle |x|_\varepsilon^{-2}|x |^{-2} |w|^q\rangle + \langle f, \phi \rangle.
\end{align*}
Therefore, for $q>d-2$ if $d \geq 4$, and $q \geq 2$ if $d=3$,
\[
\mu \langle |w|^q\rangle + \bigg[q-1-c\frac{q(q -d +2)}{(d-2)^2}\bigg]J_q \leq \langle f, \phi \rangle.
\]

Now, arguing as in the proof of Theorem \ref{thm:markgrad} (see estimates 5)-8) there), we bound $\langle f, \phi \rangle$  as follows:
$$
\langle f, \phi \rangle \leq \varepsilon J_q + C(\varepsilon)\|w\|_q^{q-2}\|f\|_q^2, \quad 0<C(\varepsilon)<\infty,
$$
where $\varepsilon>0$ is to be chosen sufficiently small.

Finally, applying the Sobolev and the Young inequalities as in the last step of the proof of Theorem \ref{thm:markgrad}, we obtain the required estimates on $\|u\|_{W^{1,q}}$, $\|u\|_{W^{1,qj}}$.
\end{proof}

\begin{corollary}
\label{hardy_drift_thm2}
Let $\Omega=\mathbb R^d, \;d \geq 5$. Let $b(x):=c|x|^{-2}x$, $c=\frac{d-2}{2}\sqrt{\delta},$ $\sqrt{\delta}<\frac{d-3}{d-2}$. There exist $q>2 \vee (d-2) $ and  constants $K_l=K_l(\delta, q), \; l=1, 2,$ such that
 $u=(\mu + \Lambda_q(b))^{-1}f$, $\mu>0$,
$f \in L^q$, satisfies
\begin{align*}
\|\nabla u \|_q & \leq K_1 \mu^{-\frac{1}{2}} \|f\|_q ;\\
\|\nabla u \|_{qj} & \leq K_2 \mu^{\frac{1}{q}-\frac{1}{2}} \|f\|_q.
\end{align*}
In particular, $(\mu + \Lambda_q(b))^{-1}: L^q \rightarrow C^{0,1- \frac{d}{qj}}$. {\rm(}For dimensions $d=3,4$, see Theorem \ref{thm:markgrad}.\rm{)}  
\end{corollary}

\begin{proof}
Modifying the proof of Theorem \ref{thm:markgrad}, we have
\begin{equation*}
\mu \langle |w|^q\rangle + I_q + (q-2)J_q =- c\langle |x|_\varepsilon^{-2} x \cdot w, \phi \rangle + \langle f, \phi\rangle,
\end{equation*}
 Integrating by parts, we have
$
\langle |x|_\varepsilon^{-2} x\cdot w, \phi\rangle = Z + \langle |x|_\varepsilon^{-2} x\cdot \nabla |w|, |w|^{q-1}\rangle,
$
where, recall, $Z=\langle |x|_\varepsilon^{-2} |w|^q\rangle - 2\langle |x|_\varepsilon^{-4}(x\cdot w)^2|w|^{q-2}\rangle,$ and
$
\langle |x|_\varepsilon^{-2} x\cdot \nabla |w|, |w|^{q-1}\rangle = -\frac{d-2}{q}\langle |x|_\varepsilon^{-2} |w|^q \rangle - \frac{2\varepsilon}{q}\langle |x|_\varepsilon^{-4} |w|^q \rangle. 
$
Clearly,
\begin{align*}
&\langle |x|_\varepsilon^{-4}(x\cdot w)^2|w|^{q-2}\rangle \leq \langle |x|_\varepsilon^{-4} |x|^4 |x|^{-2}|w|^q\rangle\\
&= \langle |x|^{-2}|w|^q\rangle-2\varepsilon\langle |x|_\varepsilon^{-2} |x|^{-2}|w|^q\rangle +\varepsilon^2 \langle |x|_\varepsilon^{-4} |x|^{-2}|w|^q\rangle,\\
&\varepsilon \langle |x|_\varepsilon^{-4} |x|^{-2}|w|^q\rangle\leq \langle |x|_\varepsilon^{-2} |x|^{-2}|w|^q\rangle.
\end{align*}
 Therefore,
\begin{align*}
&\mu \langle |w|^q\rangle + I_q + (q-2)J_q \\
& \leq c \frac{q+d-2}{q}\langle |x|^{-2}|w|^q\rangle - c\varepsilon\frac{q+d-2}{q}\langle |x|_\varepsilon^{-2} |x|^{-2} |w|^q\rangle + \langle f, \phi \rangle\\
&\leq c\frac{q(q+d-2)}{(d-2)^2}J_q  + \langle f, \phi \rangle,
\end{align*}
and so
\[
\mu \langle |w|^q\rangle + \biggl[q-1-c\frac{q(q+d-2)}{(d-2)^2}\bigg]J_q \leq   \langle f, \phi \rangle.
\]
It is seen that if $\sqrt{\delta}<\frac{d-3}{d-2}$ then $q-1-c\frac{q(q+d-2)}{(d-2)^2}>0$ for some $q>d-2$.
The rest of the proof repeats the end of the proof of Corollary \ref{hardy_drift_thm1}.
\end{proof}

\begin{remark*}
Consider the formal differential operator $L=-\Delta + b\cdot\nabla$, $b(x):=\frac{d-2}{2}\sqrt{\delta}|x|^{-2}x$. If $\delta>4$, then the Dirichlet problem for $L$ in $\{|x|<1\}$ has two distinct bounded weak solutions, moreover, one of these solutions does not satisfy the maximum principle; see remark 3 after Theorem \ref{thm:markeast}. In view of Corollary \ref{hardy_drift_thm2}, this observation can not be used to justify discarding $b(x)=\frac{d-2}{2}\sqrt{\delta}|x|^{-2}x$ regardless of the value of $\delta$
(and thus the whole class $\mathbf{F}_\delta - [L^d + L^\infty]^d$) as an argument for 
the ``optimality'' of the assumption $b \in [L^d + L^\infty]^d$ 
 (a recurring theme in the literature).
 
In \cite[Theorems 3.23, 3.24]{Met}, the authors take full advantage of the specifics of the operator $-\Delta + c|x|^{-2}x \cdot \nabla$ with $c \in \mathbb R$ and construct its  realization generating a positivity preserving holomorphic semigroup in $L^p$ if:

(a) $c=\frac{d-2}{2}\sqrt{\delta}$, $0<\delta<4$, $p \in ]\frac{2}{2-\frac{d-2}{d}\sqrt{\delta}},\infty[$;

(b) $c=\frac{d-2}{2}\sqrt{\delta}$, $4 \leq \delta<\infty$, $p \in ]\frac{d}{2},\infty[$;

(c) $c=-\frac{d-2}{2}\sqrt{\delta}$, $0<\delta<\infty$, $ p \in ]1,\infty[$.

They also explicitly describe the domain of the generator. In cases (a), (c) the authors obtain $W^{1,p}(\mathbb R^d)$ and $W^{2,p}(\mathbb R^d)$ characterization of any $u$ in the domain of the generator:

In the assumptions of (a), $|\nabla u| \in L^p$ if $p \in ]\frac{2}{2-\frac{2}{d}-\frac{d-2}{d}\sqrt{\delta}},d[$, and $\nabla_i\nabla_k u \in L^p$, $i,k=1,\dots,d$, if $p \in ]\frac{2}{2-\frac{d-2}{d}\sqrt{\delta}},\frac{d}{2}[$;

In the assumptions of (c), $|\nabla u| \in L^p$ if $p \in ]\frac{2}{2-\frac{2}{d}+\frac{d-2}{d}\sqrt{\delta}},d[$, and $\nabla_i\nabla_k u \in L^p$, $i,k=1,\dots,d$, if $p \in ]\frac{2}{2+\frac{d-2}{d}\sqrt{\delta}},\frac{d}{2}[$.

\textit{This, however, does not allow to conclude that $|\nabla u| \in L^r$ for some $r>d$}, as in Corollaries \ref{hardy_drift_thm1}  and \ref{hardy_drift_thm2}, and hence that $u \in C^{0,\gamma}$ for some $\gamma>0$ (which, in turn, is our ultimate goal).  See also remark 3 after the proof of Theorem \ref{thm:markgrad}.
\end{remark*}

We generalize and detailize Corollary \ref{hardy_drift_thm2} as follows.

\begin{corollary}
\label{G_thm}
Let $\Omega=\mathbb R^d, \;d\geq3.$ Assume that $b \in \mathbf{F}_\delta$, $\delta<4$, $b^2 \in L^1+L^\infty,$ and $G \equiv G(b), G_{ik}:=\nabla_k b_i$, satisfies the inequality
\begin{equation}
\label{G_cond}
\tag{$\star$}
|\langle Gh,h\rangle| \leqslant \delta_1 \langle |\nabla |h||^2 \rangle + c_1(\delta_1)\langle |h|^2 \rangle, \quad h \in C_0^\infty(\mathbb R^d,\mathbb R^d)
\end{equation}
for some $0<\delta_1 \leq 1-\frac{\sqrt{\delta}}{2}$, $0 \leq c_1(\delta_1)<\infty$.
Let $u=(\mu + \Lambda_q(b))^{-1}f$,
$f \in L^q$, $\mu>\omega_q$, $q \in ]2 \vee q^-,q^+[$, where $$q^{\mp}:=\frac{2-\sqrt{\delta} \mp \sqrt{(2-\sqrt{\delta})^2-4\delta_1}}{\delta_1}.$$ There exist constants $\lambda_0 =\lambda_0(\delta, q)$ and $K_l=K_l(\delta, q), \; l=1, 2,$ such that, for all $\mu > \lambda_0 \vee \omega_q,$ 
\begin{align*}
\|\nabla u \|_q & \leq K_1 \big(\mu -\lambda_0 \big)^{-\frac{1}{2}} \|f\|_q ;\\
\|\nabla u \|_{qj} & \leq K_2 \big(\mu -\lambda_0 \big)^{\frac{1}{q}-\frac{1}{2}} \|f\|_q.
\end{align*}
In particular, if $2-\sqrt{\delta} + \sqrt{(2-\sqrt{\delta})^2-4\delta_1}>(2 \vee d-2)\delta_1$ then there exists $q \in ]2 \vee d-2,q^+[$ such that $(\mu + \Lambda_q(b))^{-1}: L^q \rightarrow C^{0,1- \frac{d}{qj}}$.
\end{corollary}

\begin{proof}
Set $b_n:=E_n b$, where $E_n:=e^\frac{\Delta}{n}.$ Then $b_n \in \mathbf{F}_\delta$ (see the arguments in section 3.4) and  $E_nG(b)$ satisfies \eqref{G_cond} with the same $\delta_1$ for all $n.$
Indeed, 
\begin{align*}
|\langle [E_n G(b)] h,h \rangle |
\leq |\langle G(b) \tilde{h}, \tilde{h} \rangle|, \qquad \tilde{h}_i:=\big(E_n h_i^2\big)^{\frac{1}{2}},
\end{align*}
and by \eqref{G_cond},
\begin{align*}
 |\langle G(b) \tilde{h}, \tilde{h} \rangle| \leq  \delta_1 \big\| \big|\nabla \big(E_n |h|^2\big)^{\frac{1}{2}}\big|^2 \big\|_1 + c_1(\delta_1)\| E_n |h|^2\|_1.
\end{align*}
Noticing that $$\big\langle \big|\nabla \big(E_n (|h|^2)\big)^{\frac{1}{2}}\big|^2 \big\rangle = \bigg\|\frac{E_n(|h|\nabla |h|)}{\sqrt{E_n |h|^2}} \bigg\|_2^2 \leq \|E_n |\nabla |h||^2\|_1 = \|\nabla |h|\|_2^2, \quad \text{ and } \quad  \| E_n |h|^2\|_1=\|h\|_2^2,$$
we obtain the required.

Now, we have
\begin{equation*}
\mu \langle |w|^q\rangle + I_q + (q-2)J_q + \langle b_n \cdot w, \phi \rangle = \langle f, \phi\rangle.
\end{equation*}
Integrating by parts, we obtain
\[
\langle b_n \cdot w, \phi\rangle = \langle [E_nG(b)]w,w|w|^{q-2}\rangle + \langle b_n \cdot \nabla |w|, |w|^{q-1}\rangle.
\]
Thus,
\begin{align*}
|\langle b_n \cdot w, \phi\rangle| &\leq \delta_1 \langle \big|\nabla |w|^{\frac{q}{2}}\big|^2 \rangle + \alpha \langle b_n^2,|w|^{q}\rangle + \frac{1}{4\alpha} J_q \quad \big(\alpha=\frac{1}{q\sqrt{\delta}}\big)\\
 &\leq \biggl[ \frac{q^2}{4}\delta_1  + \frac{q^2}{4}\delta \alpha  + \frac{1}{4\alpha}\biggr] J_q + \alpha(c_1(\delta_1)+ c(\delta)) \|w\|_q^q\\
&=\bigg[ \frac{q^2}{4}\delta_1  + \frac{q}{2}\sqrt{\delta}\bigg] J_q + \frac{c_1(\delta_1)+c(\delta)}{q\sqrt{\delta}} \|w\|_q^q.
\end{align*}

Applying $I_q \geq J_q$, we obtain
\begin{equation*}
\biggl(\mu - \frac{c_1(\delta_1)+c(\delta)}{q\sqrt{\delta}}\biggr)\langle |w|^q\rangle + \biggl(q-1 - \frac{q^2}{4}\delta_1  - \frac{q}{2}\sqrt{\delta} \biggr)J_q \leq  \langle f, \phi\rangle.
\end{equation*}
(It is seen that if $q \in ]2 \vee q^-,q^+[$, then $q-1 - \frac{q^2}{4}\delta_1  - \frac{q}{2}\sqrt{\delta}>0$.)
The rest of the proof repeats the end of the proof of Corollary \ref{hardy_drift_thm1}.
\end{proof}

\begin{remarks}
1.~The requirement $b^2 \in L^1 + L^\infty$ is not essential: one can get rid of it by defining $b_n$ as $b_n:=e^{\varepsilon_n\Delta} \zeta_n b$, where $ \zeta_n(x) = \eta_n(|x|)$,
$$
\eta_n(t):=\left\{
\begin{array}{ll}
1, & \text{ if } t< n, \\
2 - \frac{t}{n}, & \text{ if } n \leq t \leq 2 n,\\
0, & \text{ if } 2 n < t,
\end{array}
\right.
$$
and $\varepsilon \downarrow 0$ are chosen sufficiently small. (Albeit this works for $\delta_1<1-\frac{\sqrt{\delta}}{2}$. The latter does not affect the result since the interval $]q^-,q^+[$ is open.) 

2.~It is easy to modify the proof above to work on an arbitrary open set $\Omega \subset \mathbb R^d$.

3.~For the vector field $b(x)=c|x|^{-2}x$, $c>0$, one has $h \cdot Gh=c\bigl(|x|^{-2}|h|^2-2(x \cdot h)^2|x|^{-4}\bigr)$, and so $|\langle Gh,h \rangle| \leq c\langle |x|^{-2}|h|^2 \rangle $. Thus, for this vector field the conditions of Corollary \ref{G_thm} are satisfied  with $\delta_1=\frac{4}{(d-2)^2}c$, $\delta=\frac{4}{(d-2)^2}c^2$. In particular, one recovers the assertion of Corollary \ref{hardy_drift_thm2}.
\end{remarks}

In the forthcoming papers \cite{KiS2}, \cite{KiS3} we extend Theorem \ref{thm:markgrad} to the operators $ -\nabla \cdot a \cdot \nabla + b\cdot\nabla$ and $-a\cdot\nabla^2 + b\cdot\nabla.$ Here we only mention the following

\begin{theorem*}[{\cite{KiS2}}]
Let $\Omega =\mathbb R^d, d \geq 3, c > -1, a(x) = I + c |x|^{-2} x \otimes x$. Let $b \in \mathbf{F}_\delta(-\Delta),$ or $b \in \mathbf{F}_\delta(A),$ where $A=A_D.$ 

{\rm (\textit{i}) [Divergence form operator]} If $\sqrt{\delta} \in ]0, 1 \wedge \frac{2}{d-2}[$ and $|c|$ is sufficiently small, or
$$c \in \biggl]-\frac{1}{2+\frac{2}{d-3}}, 2(d-3)(d-1)\biggr[ \quad (d \geq 4), \qquad c \in \bigl]-1/9,1/4\bigr[ \quad (d=3)$$ and $\delta$ is sufficiently small, then, for $q>d-2$ sufficiently close to $d-2$ the operator realization $\Lambda_q(a,b)$ of $-\nabla \cdot a \cdot \nabla + b \cdot \nabla$ is well defined, and there exist constants $\lambda_0 =\lambda_0(c, \delta, q)$ and $K_l=K_l(c, \delta, q), \; l=1, 2,$ such that $u = (\mu + \Lambda_q(a,b) )^{-1} f$ $(\mu \in \rho(-\Lambda_q), f \in L^q)$ satisfies the inequalities
\begin{equation}
\tag{$\ast$}
\begin{array}{rl}
\|\nabla u \|_q  & \leq  K_1(\mu-\lambda_0)^{-\frac{1}{2}} \|f\|_q, \\[1mm]
\|\nabla u \|_{q j} &  \leq   K_2(\mu-\lambda_0)^{\frac{1}{q}-\frac{1}{2}} \|f\|_q, \quad j=\frac{d}{d-2}.
\end{array}
\end{equation}
In particular,  $(\mu + \Lambda_q(a,b))^{-1}: L^q \rightarrow C^{0,1- \frac{d}{qj}}$, for all $\mu > \lambda_0 \vee \omega_q$.

{\rm (\textit{ii}) [Non-divergence form operator]} If $\sqrt{\delta} \in ]0, 1 \wedge \frac{2}{d-2}[$ and $|c|$ is sufficiently small, or $$c \in \biggl]-\frac{1}{1+\frac{1}{4}\frac{(d-4)^2}{(d-3)(2d-5)}},\frac{d-3}{2} \biggr[ \quad (d \geq 4), \qquad c \in ]-1,1/3[ \quad (d=3)$$ and $\delta$ is sufficiently small, then, for all $q>d-2$ sufficiently close to $d-2$ the operator realization $\Lambda_q(a,(\nabla a)+b)$ of $-a \cdot \nabla^2 + b \cdot \nabla$ in $L^q$ is well defined,
and $(\ast)$ holds for
$u = \bigl(\mu + \Lambda_q(a,(\nabla a) + b)\bigr)^{-1}$ $(\mu \in \rho(-\Lambda_q), f \in L^q).$
\end{theorem*}

\subsection{$L^r$-strong Feller semigroup on $C_\infty$ corresponding to $-\Delta +b\cdot \nabla$, $b \in \mathbf F_\delta$} 
\label{feller1_sect}

Armed with Theorem \ref{thm:markease} and Theorem \ref{thm:markgrad}, we establish

\begin{theorem}
\label{thm:markfeller}
Let $d \geq 3$, $b : \mathbb{R}^d \rightarrow \mathbb{R}^d,$ $b \in \mathbf F_\delta$, $\delta \in \big] 0, 1 \wedge \frac{4}{(d-2)^2} \big [.$ Then:

{\rm(\textit{i})} The limit 
\[
e^{-t\Lambda_{C_\infty}(b)}:=s \mbox{-} C_\infty \mbox{-} \lim_n e^{-t \Lambda_{C_\infty}(b_n)} \quad (t \geq 0)
\] 
exists and determines a positivity preserving contraction $C_0$ semigroup on $C_\infty.$ Here $\{b_n\}$ is defined in section \ref{est_sect}, $\Lambda_{C_\infty}(b_n):=-\Delta+b_n\cdot \nabla$, 
$D(\Lambda_{C_\infty}(b_n))=(1-\Delta)^{-1}C_\infty$,

\smallskip

{\rm(\textit{ii})}~{\rm[}The $L^r$-strong Feller property {\rm]}~$\bigl( (\mu+\Lambda_{C_\infty}(b))^{-1} \upharpoonright L^r \cap C_\infty\bigr)^{\clos}_{L^r \rightarrow C_\infty} \in \mathcal B(L^r,  C^{0,1- \frac{d}{rj}})$ whenever $r \in \big]\frac{2}{2-\sqrt{\delta}} \vee (d-2), \frac{2}{\sqrt{\delta}} \big[$ and $\mu > \omega_q.$  

\smallskip

{\rm(\textit{iii})}~The
integral kernel of $e^{-t\Lambda_{C_\infty}(b)}$ determines the
transition probability function of a Hunt process.
\end{theorem}

\begin{proof}We will need the following auxiliary results.
Set
$$
g:=u_m- u_n, \quad u_n = (\mu + \Lambda_r(b_n))^{-1} f, \quad f \in L^1 \cap L^\infty, \quad \mu > \frac{c(\delta)}{\delta}.
$$

\begin{lemma}
\label{lem1}
There are positive constants $C=C(d), k=k(\delta)$ such that
\[
\|g \|_{r j} \leq \big( C \delta \|\nabla u_n \|_{q j}^2 \big)^\frac{1}{r} \big(r^{2k} \big)^\frac{1}{r} \|g\|_{x^\prime (r-2)}^{1-\frac{2}{r}},
\]
where $q \in \big]\frac{2}{2-\sqrt{\delta}} \vee (d-2), \frac{2}{\sqrt{\delta}}\big [, \; 2 x = q j, \; j=\frac{d}{d-2}, \; x^\prime := \frac{x}{x-1}$ and $x^\prime(r-2) > \frac{2}{2-\sqrt{\delta}}.$ 
\end{lemma}
\begin{proof}
 Note that $g$ satisfies the equation
\[
(\mu + \Lambda_q(b_m)) g = F, \quad \quad F = (b_n - b_m) \cdot \nabla u_n .
\]
Let $\psi = g |g|^{r-2}, v =g |g|^\frac{r-2}{2}.$ Taking the scalar product of the equation by $\psi,$ we have
\[
\mu \|v\|_2^2 + \frac{4}{r r^\prime} \|\nabla v\|_2^2 = - \frac{2}{r} \langle v, b_m \cdot \nabla v \rangle + \langle (b_n - b_m) \cdot \nabla u_n, v |v|^{1-\frac{2}{r}} \rangle.
\]
By the quadratic estimates,
\begin{align*}
|\langle v, b_m \cdot \nabla v \rangle| & \leq \varepsilon \|b_m v \|_2^2 + (4\varepsilon)^{-1} \| \nabla v \|_2^2\\
&\leq (\varepsilon \delta + (4 \varepsilon)^{-1} ) \|\nabla v \|_2^2 + \varepsilon c(\delta) \|v\|_2^2\\
& = \sqrt{\delta} \|\nabla v \|_2^2 +(2 \sqrt{\delta})^{-1} c(\delta)\|v\|_2^2 \quad \quad (\varepsilon = (2 \sqrt{\delta})^{-1} ),\\ 
|\langle (b_n - b_m) \cdot \nabla u_n, v |v|^{1-\frac{2}{r}} \rangle| & \leq \langle (|b_n| + |b_m|) |v|, |v|^{1-\frac{2}{r}} | \nabla u_n | \rangle \\
& \leq \eta \delta \|\nabla v \|_2^2 + \eta c(\delta) \|v\|_2^2 +\eta^{-1} \| |v|^{1-\frac{2}{r}} |\nabla u_n| \|_2^2 \quad (\eta >0),
\end{align*}
we obtain the inequality
\[
\bigg[\mu - \bigg(\frac{1}{r} \frac{1}{\sqrt{\delta}} + \eta \bigg) c(\delta) \bigg ] \|v\|_2^2 + \bigg( \frac{4}{r r^\prime} -\frac{2}{r}\sqrt{\delta} -\eta \delta \bigg) \|\nabla v\|_2^2 \leq \eta^{-1} \||v|^{1-\frac{2}{r}} |\nabla u_n\|_2^2.
\]
Since $r > \frac{2}{2 -\sqrt{\delta}} \Leftrightarrow \frac{2}{r^\prime} - \sqrt{\delta} > 0,$ we choose $k > 1$ so large that $\frac{4}{r r^\prime} - \frac{2}{r} \sqrt{\delta} = \frac{2}{r} \big(\frac{2}{r^\prime} - \sqrt{\delta} \big)> 2 r^{-k}.$
Fix $\eta$ by $\eta \delta = \frac{4}{r r^\prime} - \frac{2}{r} \sqrt{\delta} - r^{-k} \;\; ( \geq r^{-k} ).$ Thus
\[
\bigg [ \mu - \frac{c(\delta)}{\delta}\bigg (\frac{4}{r r^\prime} -\frac{1}{r} \sqrt{\delta} - r^{-k} \bigg) \bigg ] \|v\|_2^2 + r^{-k} \| \nabla v \|_2^2 \leq \delta r^k \| |v|^{1-\frac{2}{r}} |\nabla u_n| \|_2^2 .
\]
Our choice of $\mu$ ( $\mu > \frac{c(\delta)}{\delta}$ ) ensures that the expression contained in square brackets is strictly positive. Thus
\[
\| \nabla v \|_2^2 \leq \delta r^{2 k} \| \nabla u_n \|^2_{2 x} \|v \|^{2(1-\frac{2}{r})}_{2 x^\prime (1-\frac{2}{r})} .
\]
Finally, applying the uniform Sobolev inequality $c_d \|v \|^2_{2 j} \leq \| \nabla v \|_2^2$, we end the proof of Lemma \ref{lem1}.
\end{proof}

\begin{lemma}
\label{lem2}
In the notation of Lemma \ref{lem1}, for any $r_0 > \frac{2}{2-\sqrt{\delta}}$ 
$$
\|g\|_\infty \leq B \|g \|_{r_0}^\gamma, \quad \mu \geq 1 + \lambda_0 \vee \frac{c(\delta)}{\delta},
$$
where $\gamma = \big( 1 - \frac{x^\prime}{j} \big) \big( 1 - \frac{x^\prime}{j} + \frac{2 x^\prime}{r_0}\big)^{-1} > 0$, and $B=B(d, \delta,K_2)<\infty$ is a constant {\rm(}$K_2$ is the constant in Theorem \ref{thm:markgrad}{\rm)}.
\end{lemma}

\begin{proof} Let $D := C \delta \sup_n \|\nabla u_n \|_{q j}^2 .$ By Theorem \ref{thm:markgrad}, $D < \infty.$ We iterate the inequality
\[
\|g\|_{rj} \leq D^\frac{1}{r} (r^\frac{1}{r})^{2k} \|g\|_{x^\prime(r-2)}^{1-\frac{2}{r}}\\
\tag{$\star$}
\]
as follows. Successively setting $x^\prime (r_1 -2) = r_0, \; x^\prime (r_2 - 2) = j r_1, \; x^\prime (r_3 - 2) = j r_2, \dots$ so that
$r_n = (t -1)^{-1} \big( t^n ( \frac{r_0}{x^\prime} +2) - t^{n-1} \frac{r_0}{x^\prime} -2 \big),$ where $t = \frac{j}{x^\prime} > 1,$ we get from $(\star)$
\[
\|g\|_{r_n j} \leq D^{\alpha_n} \Gamma_n \|g\|_{r_0}^{\gamma_n} ,
\]
where
\begin{align*}
\alpha_n = &\frac{1}{r_1}\bigg(1 - \frac{2}{r_2} \bigg) \bigg(1 - \frac{2}{r_3} \bigg) \dots \bigg(1 - \frac{2}{r_n} \bigg) + \frac{1}{r_2}\bigg(1 - \frac{2}{r_3} \bigg) \bigg(1 - \frac{2}{r_4} \bigg)\dots \bigg(1 - \frac{2}{r_n} \bigg)\\
& + \dots + \frac{1}{r_{n-1}} \bigg(1 - \frac{2}{r_n} \bigg) + \frac{1}{r_n} ;\\
\gamma_n = & \bigg(1 - \frac{2}{r_1} \bigg) \bigg(1 - \frac{2}{r_2} \bigg) \dots \bigg(1 - \frac{2}{r_n} \bigg) ;\\
\Gamma_n =& \bigg[r_n^{r_n^{-1}} r_{n-1}^{r_{n-1}^{-1}(1-2 r_n^{-1})} r_{n-2}^{r_{n-2}^{-1} (1-2 r_{n-1}^{-1}) (1-2 r_n^{-1})} \dots r_1^{r_1^{-1} (1-2 r_2^{-1}) \dots (1- 2 r_n^{-1}) } \bigg]^{2k} .
\end{align*}
Since $\alpha_n = t^n - r_n^{-1}(t-1)^{-1}$ and $\gamma_n = r_0 t^{n-1}(x^\prime r_n)^{-1},$
\[
\alpha_n \leq \alpha \equiv \bigg( \frac{r_0}{x^\prime} +2 -  \frac{r_0}{j} \bigg)^{-1}, 
\]
and
\[
\inf_n \gamma_n > \gamma = \big( 1 - \frac{x^\prime}{j} \big) \big( 1 - \frac{x^\prime}{j} + \frac{2 x^\prime}{r_0}\big)^{-1} > 0, \quad \quad \sup_n \gamma_n < 1.
\]
Note that $\|g\|_{r_0} \rightarrow 0$ as $n, m \uparrow \infty$ since  $r_0 \in I_c^o$, and so $\|g\|_{r_0}^{\gamma_n} \leq \|g\|_{r_0}^{\gamma}$ for all large enough $n, m.$

Finally, since
\[
\Gamma_n^\frac{1}{2 k} = r_n^{r_n^{-1}} r_{n-1}^{t r_n^{-1}} r_{n-2}^{t^2 r_n^{-1}} \dots r_1^{t^{n-1}r_n^{-1}} \quad \text{ and } \quad b t^n \leq r_n \leq a t^n,
\]
where $a= r_1(t-1)^{-1}, \; b = r_1 t^{-1},$ we have
\begin{align*}
\Gamma_n^\frac{1}{2 k} & \leq (a t^n)^{(bt^n)^{-1}} (a t^{n-1})^{(bt^{n-1})^{-1}} \dots (a t)^{(b t)^{-1}} \\
& = \bigg[ a^{t^n -t^{-n}(t-1)^{-1}} t^{\sum_{i=1}^n i t^{-i}} \bigg]^\frac{1}{b} \leq \bigg[ a^{(t-1)^{-1}} b^{t(t-1)^2} \bigg]^\frac{1}{b}.
\end{align*}
The proof of Lemma \ref{lem2} is completed.
\end{proof}

\begin{remark*}
The fact that $\gamma>0$ is the main concern of the iterative procedure.
\end{remark*}

\begin{lemma}
\label{lem3}
Let $U_n := (\mu + \Lambda_{r_0}(b_n) )^{-1} F , \;  \mu > \frac{c(\delta)}{\delta}, \; F:= b_n \cdot \nabla (\mu - \Delta)^{-1} f, \; f \in C_c^1.$
There are constants $0 < \tilde{\gamma} \leq 1,$ $\tilde{B}$ and $\hat{B}$ independent of $n$ such that
\[
\| U_n \|_\infty \leq \tilde{B} \|U_n \|_{r_0}^{\tilde{\gamma}},
\]
\[ \|\mu U_n \|_\infty \leq \hat{B} \|\mu U_n \|_{r_0}^{\tilde{\gamma}}
\]
whenever $r_0 > \frac{2}{2-\sqrt{\delta}}.$
\end{lemma}
\begin{proof}
Proceeding exactly as in the proof of Lemma \ref{lem1}, we obtain the inequalities
\[
\| U_n \|_{r j} \leq (C \delta \|\nabla (\mu -\Delta)^{-1} f \|^2_{q j} )^\frac{1}{r}(r^{2 k})^\frac{1}{r} \|U_n \|_{x^\prime (r-2)}^{1-\frac{2}{r}},
\]
\[
\|\mu U_n \|_{r j} \leq (C \delta \|\nabla f \|^2_{q j} )^\frac{1}{r}(r^{2 k})^\frac{1}{r} \|\mu U_n \|_{x^\prime (r-2)}^{1-\frac{2}{r}};
\]
their iteration provides the required result.
\end{proof}

\begin{lemma}
\label{lem4}
In the notation of Lemma \ref{lem3},
we have $$
\|\mu U_n \|_r \leq \bigg( \frac{8}{r}\bigg(\frac{2}{r^\prime} - \sqrt{\delta}\bigg ) \bigg)^{-\frac{1}{2}} \bigg(\mu - \frac{c(\delta)}{\delta} \bigg)^{-\frac{1}{2}} \|\nabla f\|_r
$$
whenever $r > \frac{2}{2-\sqrt{\delta}}.$
\end{lemma}
\begin{proof}
Proceeding again as in Lemma \ref{lem1}, we obtain the inequality ($\eta > 0$)
\[
\bigg[\mu - \bigg(\frac{1}{r} \frac{1}{\sqrt{\delta}} + \eta \bigg) c(\delta) \bigg ] \|v\|_2^2 + \bigg( \frac{4}{r r^\prime} -\frac{2}{r}\sqrt{\delta} -\eta \delta \bigg) \|\nabla v\|_2^2 \leq (4 \eta)^{-1} \||v|^{1-\frac{2}{r}} | f_\mu|\|_2^2,
\]
where $v := U_n |U|_n^\frac{r-2}{2}$ and $f_\mu := \nabla (\mu - \Delta)^{-1}f.$
Setting here $\eta \delta = \frac{4}{r r^\prime} -\frac{2}{r}\sqrt{\delta}$ and noticing that
\[
\mu - \bigg(\frac{1}{r} \frac{1}{\sqrt{\delta}} + \eta \bigg) c(\delta)  = \mu - \frac{c(\delta)}{\delta} \bigg( \frac{4}{r r^\prime} -\frac{\sqrt{\delta}}{r} \bigg)  \geq \mu -\frac{c(\delta)}{\delta} 
\]
we have
\[
 \frac{8}{r} \bigg( \frac{2}{r^\prime} -\sqrt{\delta} \bigg) \bigg(\mu -\frac{c(\delta)}{\delta} \bigg) \|v\|_2^2 \leq \delta \|v\|_2^{2(1-\frac{2}{r})} \|f_\mu \|_r^2.
\]
It remains to note that $\|f_\mu \|_r \leq \mu^{-1} \|\nabla f \|_r.$ 
\end{proof}

\begin{lemma}
\label{lem5}
$s \mbox{-} C_\infty \mbox{-} \lim_{\mu \uparrow \infty} \mu (\mu + \Lambda_{C_\infty}(b_n))^{-1} = 1$ uniformly in $ n.$
\end{lemma}

\begin{proof} 
We only need to show that
\[
\lim_{\mu \uparrow \infty} \sup_n \|\mu\big[ (\mu + \Lambda_r(b_n))^{-1} - (\mu -\Delta)^{-1} \big] f\|_\infty =0 \quad \text{ for all } f \in C_c^1.
\]
Indeed, since $-\big[(\mu + \Lambda_r(b_n))^{-1} - (\mu -\Delta)^{-1} \big] f = (\mu + \Lambda_r(b_n) )^{-1} b_n \cdot \nabla (\mu - \Delta)^{-1} f = U_n,$ we obtain by Lemma \ref{lem3} and Lemma \ref{lem4} that
$$
\|\mu U_n \|_\infty \leq \hat{B} \|\mu U_n \|_{r_0}^{\tilde{\gamma}} \leq \dot{B} \bigg(\mu- \frac{c(\delta)}{\delta} \bigg)^{-\frac{\tilde{\gamma}}{2}} \|\nabla f\|_{r_0}^{\tilde{\gamma}},
$$
which yields the required.
\end{proof}

We are in position to complete the proof of Theorem \ref{thm:markfeller}.
(\textit{i}) follows from Lemmas \ref{lem2}, \ref{lem5} and Theorem \ref{thm:markease}
by applying the Trotter Approximation Theorem (Appendix \ref{trotter_sect}). 
(\textit{ii}) is Theorem \ref{thm:markgrad}.
The proof of (\textit{iii}) is standard.

The proof of Theorem \ref{thm:markfeller} is completed.
\end{proof}
\begin{remarks}
1.\,Theorem \ref{thm:markfeller} is valid for any $\{b_n\} \subset C^\infty \cap \mathbf F_\delta, \; b_n \overset{s}\rightarrow b \;\mathcal L^d$ a.e.

2.\,For a parabolic variant of the above iteration procedure see \cite{Ki}.
\end{remarks}

\bigskip

\section{$b \cdot \nabla$ is $-\Delta$ weakly form-bounded}
\label{weak_fbd_sect}

Let $L^p=L^p(\mathbb R^d,\mathcal L^d), \;p \in [1,\infty],$ be the standard (complex) Lebesgue spaces. 

Consider the following classes of
vector fields:

\begin{definition*}

(1) A  $b: \mathbb{R}^d \rightarrow  \mathbb{R}^d$ belongs to the Kato class $\mathbf{K}^{d+1}_\delta$ (write $ b \in \mathbf{K}^{d+1}_\delta$) if $|b| \in L^1_{\loc}$ and there exists $\lambda = \lambda_\delta > 0$ such that
\[
\| b (\lambda - \Delta)^{-\frac{1}{2}} \|_{1 \rightarrow 1} \leq \delta.
\]

(2) A $b: \mathbb{R}^d \rightarrow  \mathbb{R}^d$ belongs to $\mathbf{F}_\delta^{\scriptscriptstyle 1/2}\equiv \mathbf{F}_\delta^{\scriptscriptstyle 1/2}(-\Delta)$, the class of \textit{weakly} form bounded vector fields (write $b \in \mathbf{F}_\delta^{\scriptscriptstyle 1/2}$) if 
$|b| \in L^1_{\loc}$ and there exists $\lambda = \lambda_\delta > 0$ such that
\[
\| |b|^\frac{1}{2} (\lambda - \Delta)^{-\frac{1}{4}} \|_{2 \rightarrow 2} \leq \sqrt{\delta}.
\]
\end{definition*}

\begin{example} 
\label{ex1}
1. The inclusion 
$|b| \in L^p+L^\infty\;\; (p>d) \Rightarrow b \in \mathbf{K}^{d+1}_0:=\bigcap_{\delta>0}\mathbf{K}^{d+1}_\delta$,
follows from H\"{o}lder's inequality.

2.~We have: $$b(x):=e\mathbf{1}_{|x_1|<1}|x_1|^{s-1}, \quad \frac{1}{2}<s<1,$$
where $e=(1,\dots,1) \in \mathbb R^d$, $x=(x_1,\dots,x_d)$, 
is in $\mathbf{K}^{d+1}_{0} - \mathbf{F}_{\delta_2}$ for any $\delta_2>0$.

In turn,  $b(x):=\sqrt{\delta} \frac{d-2}{2}|x|^{-2}x \in \mathbf{F}_{\delta} - \mathbf{K}^{d+1}_{\delta_1}$ for any $\delta, \delta_1>0$.

Thus,
$ \mathbf{K}^{d+1}_{0} - \mathbf{F}_\delta \neq \varnothing,$ and $\mathbf{F}_{\delta_1} - \mathbf{K}^{d+1}_\delta \neq \varnothing$ for any $\delta, \delta_1 > 0.$

3.~An example of a $b \in \mathbf{K}^{d+1}_\delta - \mathbf{K}^{d+1}_0$, $\delta>0$, can be obtained as follows.
 
Fix $e \in \mathbb R^d$, $|e|=1$.
Let $z_n:=(2^{-n},0,\dots,0) \in \mathbb R^d$, $n \geq 1$. Set
$$
b(x):=e F(x), \quad F(x):=\sum_{n=1}^\infty 8^{n} \mathbf{1}_{B(z_n,8^{-n})}(x), \quad x \in \mathbb R^d,
$$
where $B(z_n,8^{-n})$ is the open ball of radius $8^{-n}$ centered at $z_n$ and $\mathbf{1}_{B(z_n,8^{-n})}$ is its indicator. 
 
Then $b \in \mathbf{K}^{d+1}_\delta - \mathbf{K}^{d+1}_0$ for appropriate $\delta>0$.

 4. The class $\mathbf{F}_\delta^{\scriptscriptstyle 1/2}$  is the largest: 
\begin{equation}
\label{prop_incl}
\begin{array}{c}
\mathbf{K}_\delta^{d+1} \subsetneqq \mathbf{F}_\delta^{\scriptscriptstyle 1/2}, \qquad 
\mathbf{F}_{\delta_1} \subsetneqq \mathbf{F}_\delta^{\scriptscriptstyle 1/2}
\quad \text{ for  
$\delta = \sqrt{\delta_1},$} \\[4mm]
\biggl( b \in  \mathbf{F}_{\delta_1}^{~} \text{ and } \mathsf{f} \in \mathbf{K}^{d+1}_{\delta_2}  \biggr) \Longrightarrow \biggl( b + \mathsf{f} \in \mathbf{F}^{\scriptscriptstyle 1/2}_{\delta}, \; \sqrt{\delta} = \sqrt[4]{\delta_1} + \sqrt{\delta_2}\biggr)
\end{array}
\end{equation}
Indeed, for $b \in \mathbf{K}_\delta^{d+1}$,
$\||b| (\lambda - \Delta)^{-\frac{1}{2}} \|_{1 \rightarrow 1} \leq \delta$.  By duality, $\| (\lambda - \Delta)^{-\frac{1}{2}} |b| \|_{\infty} \leq \delta$, and so by interpolation,
$
\||b|^{\frac{1}{2}}(\lambda - \Delta)^{-\frac{1}{2}}|b|^{\frac{1}{2}}\|_{2 \rightarrow 2} \leq \delta.
$
Therefore, $b \in \mathbf{F}_\delta^{\scriptscriptstyle 1/2}$. 
The second inclusion $\mathbf{F}_{\delta_1} \subsetneqq \mathbf{F}_\delta^{\scriptscriptstyle 1/2}$, $\delta=\sqrt{\delta_1}$ is the consequence of the Heinz inequality \cite{He}.
The last assertion now follows from $$b \in \mathbf{F}_{\sqrt{\delta}_1}^{\scriptscriptstyle 1/2}, \mathsf{f} \in \mathbf{F}_{\delta_2}^{\scriptscriptstyle 1/2} \Rightarrow b+\mathsf{f} \in \mathbf{F}_\delta^{\scriptscriptstyle 1/2},$$ where we have used $(|b|+|\mathsf{f}|)^{\frac{1}{2}} \leq |b|^{\frac{1}{2}} + |\mathsf{f}|^{\frac{1}{2}}$.

5.~For the examples of vector fields in the class $\mathbf{F}_\delta$ see Example \ref{ex0} in the beginning of section \ref{fbd_sect}.

\end{example}

\begin{remarks}

1.~The classes $\mathbf{F}_{\delta}$, $\mathbf{K}_\delta^{d+1}$ cover singularities of $b$ of critical order (i.e.~sensitive to multiplication by a constant\footnote{For instance, the uniqueness of weak solution to the Cauchy problem for $\partial_t-\Delta + b\cdot\nabla$ can fail if $b \in \mathbf{F}_\delta$ is replaced with $cb~(\in \mathbf{F}_{c^2\delta})$ for a sufficiently large constant $c$. \cite[Example 5]{KS}.}\,) at isolated points or along hypersurfaces, 
respectively, as follows from Example \ref{ex0}(3) and Example \ref{ex1}(2,3). 
The classes $\mathbf{K}^{d+1}_0$, $\mathbf{F}_0$ and, thus, $[L^d+L^\infty]^d$, do not contain vector fields having critical-order singularities.

2.~The Kato class $\mathbf{K}_\delta^{d+1}$, with $\delta>0$ sufficiently small (yet allowed to be non-zero), is recognized as the class `responsible' for
the Gaussian upper and lower bounds on the fundamental solution of $\partial_t-\Delta + b\cdot \nabla$. The latter allow to construct an associated Feller semigroup (in $C_b$).
The class $\mathbf{F}_\delta$, $\delta<4$, is recognized as the class `responsible' for dissipativity
of $\Delta - b\cdot\nabla$ in $L^p$, $p \geqslant 2/(2-\sqrt{\delta})$, needed to run the iterative procedure of section \ref{feller1_sect} (taking $p \rightarrow \infty$, assuming additionally $\delta<\min\{1,(2/(d-2))^2\}$), which
produces
an associated Feller semigroup in $C_\infty$.
We emphasize that, in general, the Gaussian bounds are not valid if $b \in [L^d]^d$, while $b \in \mathbf{K}_0^{d+1}$, in general, destroys the $L^p$-dissipativity.
\end{remarks}

\begin{figure}[t]
\label{fig1}

\begin{equation*}
\bfig
\node a1(0,0)[\qquad \lbrack L^p+L^\infty \rbrack^d~(p>d)]
\node a2(0,500)[\lbrack L^d+L^\infty\rbrack^d]
\node a3(0,1000)[\mathbf{F}_0]
\node a4(0,1500)[\mathbf{F}_{\delta^2}]
\node a55(800, 1000)[\mathbf{C}_s]
\node a56(800, 1500)[\mathbf{W}_s]
\node a5(800,500)[\lbrack L^{d,\infty}+L^\infty \rbrack^d]
\node b1(-400,1500)[\mathbf{K}_\delta^{d+1}]
\node b0(-800,1000)[\mathbf{K}_0^{d+1}]
\node c1(0,2000)[\mathbf{F}_\delta^{\scriptscriptstyle 1/2}]
\arrow[a1`a2;]
\arrow[a2`a3;]
\arrow[a3`a4;]
\arrow[a2`a5;]
\arrow[a1`b0;]
\arrow[b0`b1;]
\arrow[a4`c1;]
\arrow[b1`c1;]
\arrow[a5`a55;]
\arrow[a55`a56;]
\arrow[a56`a4;\ast]
\efig
\end{equation*}

\footnotesize{General classes of vector fields $b:\mathbb R^d \rightarrow \mathbb R^d$ studied in literature in connection with operator $-\Delta +b \cdot \nabla$. Here $\rightarrow$ stands for strict inclusion of vector spaces, and  $\overset{\ast}{\rightarrow}$ reads: if  $b \in \mathbf{W}_s$ ($s>1$), then $b \in \mathbf{F}_{\delta^2}$ with $\delta=\delta(\|b^2\|_{W_s})<\infty$.}
\caption{}
\end{figure}

\subsection{A variant of the Hille-Lions approach.\;$L^2$-theory} 
\label{hl_sect}

Let
$\mathcal W^{\alpha,p}=\mathcal W^{\alpha,p}(\mathbb R^d,\mathcal L^d)$, $\alpha>0$, be the Bessel potential space endowed with norm $\|u\|_{p,\alpha}:=\|g\|_p$,  
$u=(1-\Delta)^{-\frac{\alpha}{2}}g$, $g \in L^p$, and $\mathcal W^{-\alpha,p'}$, $p'=\frac{p}{p-1}$, the anti-dual of $\mathcal W^{\alpha,p}$. For a comprehensive account of Bessel potential spaces see \cite[sect.\,2.6]{Zi}.

\begin{theorem}
\label{thm:markweak}
Let $d \geq 3$, $b:\mathbb R^d \rightarrow \mathbb R^d$. Assume that $b \in \mathbf F^{\scriptscriptstyle 1/2}_\delta$, $\delta < 1.$ Then there is an operator realization $-\Lambda$ of $\Delta - b\cdot \nabla$ in $L^2$ generating a quasi bounded holomorphic semigroup $e^{-t\Lambda}$. We have:

{\rm(\textit{i})} The resolvent set $\rho(-\Lambda)$ contains the half-plane $\mathcal O := \{\zeta \in \mathbb C \mid \Real\, \zeta \geq \lambda =\lambda_\delta \}$. The resolvent admits the representation 
\begin{align*}
(\zeta + \Lambda)^{-1} = (\zeta - \Delta)^{-\frac{3}{4}} (1 + H^* S)^{-1} (\zeta - \Delta)^{-\frac{1}{4}}, 
\end{align*}
where
$H := |b|^\frac{1}{2} (\zeta -\Delta)^{-\frac{1}{4}}$, $b^\frac{1}{2} := |b|^{-\frac{1}{2}} b$,
$S := b^\frac{1}{2} \cdot \nabla (\zeta - \Delta)^{-\frac{3}{4}}$,
and $\|H^* S \|_{2\rightarrow 2} \leq \delta$;

{\rm(\textit{ii})}
$$
(\zeta + \Lambda)^{-1}=(\zeta - \Delta)^{-1} - (\zeta - \Delta)^{-\frac{3}{4}} H^*(1 + S H^*)^{-1} S (\zeta -\Delta)^{-\frac{1}{4}};
$$

{\rm(\textit{iii})} 
$$
\|(\zeta + \Lambda)^{-1} \|_{2\rightarrow 2}  \leq (1- \delta)^{-1} |\zeta |^{-1} \quad (\zeta \in \mathcal O);
$$

{\rm(\textit{iv})} $\Lambda$ is related to $-\Delta + b\cdot\nabla$ as follows. If $f \in D(\Lambda)$, then 
\[
\langle \Lambda f, g \rangle = \langle \nabla f, \nabla g \rangle + \langle b^\frac{1}{2} \cdot \nabla f, |b|^\frac{1}{2} g \rangle, \quad (g \in W^{1,2});
\]

{\rm(\textit{v})}  If $f \in D(\Lambda),$ then $b \cdot \nabla f \in L^1_\loc, \; D(\Lambda) \subset \mathcal W^{2,1}_\loc$, and
\[
 \langle \Lambda f, \phi \rangle = \langle f, - \Delta \phi \rangle + \langle b \cdot \nabla f, \phi \rangle \quad \quad ( \phi \in C_c^\infty).
\]
\end{theorem}

\begin{proof}
\textbf{1}.~Set $\mathcal H_0:=L^2$. Define $A := \lambda - \Delta$ of domain $D(A) = W^{2,2}, \; \mathcal H_\alpha := \big( D(A^\alpha), \langle f, g \rangle_{\mathcal H_\alpha} =\langle A^\alpha f, A^\alpha g \rangle \big) \; (\alpha \geq 0)$ and $J_\lambda := (\lambda - \Delta)^{-\frac{1}{4}},$ the $\frac{1}{4}$ power of Bessel's potential. Clearly, $ \mathcal H_\alpha = \mathcal W^{2 \alpha, 2}$ and $\mathcal H_{-\frac{1}{4}} = \mathcal W^{-\frac{1}{2}, 2}.$ Consider the chain of Hilbert spaces
\[
\mathcal H_1 \subset \mathcal H_\frac{3}{4} \subset \mathcal H_\frac{2}{4} \subset \mathcal H_\frac{1}{4}\subset \mathcal H_0 \subset \mathcal H_{-\frac{1}{4}}.
\]
Then $J_\lambda : \mathcal  H_l \rightarrow \mathcal H_{l+\frac{1}{4}}, \; l = -\frac{1}{4}, 0, \frac{1}{4}, \frac{2}{4}, \frac{3}{4}$ are bijections,
\[
\mathcal H_\frac{3}{4} \subset \mathcal H_\frac{1}{4} \subset \mathcal H_{-\frac{1}{4}} 
\]
is the standard triple of Hilbert spaces (so that $\mathcal H_{-\frac{1}{4}}^* = \mathcal H_\frac{3}{4}$ with respect to $\langle f, g \rangle_{\mathcal H_\frac{1}{4}}$).

 By $\langle f, g \rangle_\frac{1}{4}, \; f \in \mathcal H_{-\frac{1}{4}}, \; g \in \mathcal H_\frac{3}{4}$ we denote the pairing between $\mathcal H_{-\frac{1}{4}}$ and  $\mathcal H_\frac{3}{4}.$ Then
 \[
 \langle f, g \rangle_\frac{1}{4} = \langle f, g \rangle_{\mathcal H_\frac{1}{4}} \text{ whenever } f \in \mathcal H_\frac{1}{4}.
 \]
By $\hat{A}$ we denote the extension of $A - \lambda$ to a bounded map from $\mathcal H_\frac{3}{4}$ into $\mathcal H_{-\frac{1}{4}}.$ Then 
\[
|\langle (\zeta  + \hat A) f, f \rangle_\frac{1}{4} | \geq \|f \|^2_{\mathcal H_\frac{3}{4}} \quad \quad  ( f \in \mathcal H_\frac{3}{4}, \; \Real \zeta \geq \lambda),
\]
and so $\zeta + \hat A$ is a bijection; $\; \|\zeta + \hat A \|_{\mathcal H_\frac{3}{4} \rightarrow \mathcal H_{-\frac{1}{4}}} \geq 1.$   Clearly
\[
D(A) = \{ f \in \mathcal H_\frac{3}{4} \mid \hat A f \in \mathcal H_0 \} \text{ and } A^{-1} = (\lambda + \hat A)^{-1} \upharpoonright {\mathcal H_0} = J_\lambda^4.
\]
The operator $\hat B := b \cdot \nabla: \mathcal H_\frac{3}{4} \rightarrow \mathcal H_{- \frac{1}{4}}$ is bounded. Indeed, according to $b \in \mathbf F_\delta^{\scriptscriptstyle 1/2},$
\[
b^\frac{1}{2} \cdot \nabla : \mathcal H_\frac{3}{4} \rightarrow \mathcal H_0, \; |b|^\frac{1}{2} : \mathcal H_0 \rightarrow \mathcal H_{-\frac{1}{4}}, \; \hat B \in \mathcal B (\mathcal H_\frac{3}{4}, \mathcal H_{- \frac{1}{4}}) \text{ with } \|\hat B \|_{\mathcal H_\frac{3}{4} \rightarrow \mathcal H_{-\frac{1}{4}}} \leq \delta.
\]
Thus $ |\langle (\zeta + \hat A + \hat B) f, f \rangle_\frac{1}{4} | \geq (1 - \delta) \|f\|_{\mathcal H_\frac{3}{4}}^2,$ and so $\zeta + \hat \Lambda := \zeta + \hat A + \hat B$ is a bijection.

 Set $\hat R_\zeta := (\zeta + \hat \Lambda)^{-1}$ and $\hat F_\zeta := (\zeta + \hat A)^{-1} (=J^4_\zeta).$ Clearly,
\begin{equation}
\tag{$p_1$}
\hat{R}_\zeta = \hat{R}_\eta + (\eta - \zeta)\hat{R}_\zeta \hat{R}_\eta, \qquad \zeta, \eta \in \mathcal O ; \label{hl_1}
\end{equation}
\[
 \hat R_\zeta = \hat F_\zeta - \hat F_\zeta \hat B \hat R_\zeta = \hat F_\zeta - \hat F_\zeta \hat B \hat F_\zeta + \hat F_\zeta \hat B \hat F_\zeta \hat B \hat F_\zeta - \dots;
 \]
\[
\hat F_\zeta \hat B \hat F_\zeta = J_\zeta^3 J_\zeta \hat B J_\zeta^3 J_\zeta;
\]
\[
\hat F_\zeta \hat B \hat F_\zeta \upharpoonright {\mathcal H_0} = J_\zeta^3 J_\zeta |b|^\frac{1}{2} b^\frac{1}{2} \cdot \nabla J_\zeta^3 J _\zeta \upharpoonright {\mathcal H_0} = J_\zeta^3 H^* S J_\zeta;\]
\begin{equation}
\tag{$p_2$}
R_\zeta := \hat{R}_\zeta \upharpoonright {\mathcal H_0} = J_\zeta^3 (1 + H^* S)^{-1} J_\zeta. \label{hl_2}
\end{equation}
Since $|J_\zeta g | \leq J_\lambda |g|,$ and hence
\begin{align*}
\|H^* S f \|_{\mathcal H_0} & \leq \|H\|_{\mathcal H_0 \rightarrow \mathcal H_0 } \||b|^\frac{1}{2} J_\lambda |\nabla J_\zeta^2 f| \|_{\mathcal H_0} \\
& \leq \delta \|\nabla J_\zeta^2 f| \|_{\mathcal H_0} \leq \delta \|f\|_{\mathcal H_0}, \; f \in \mathcal H_0,
\end{align*}
it follows from \eqref{hl_2} that
\begin{equation}
\label{hl_3}
\tag{$p_3$}
 \|R_\zeta \|_{\mathcal H_0 \rightarrow \mathcal H_0} \leq (1 - \delta)^{-1} |\zeta|^{-1}. 
 \end{equation}
We conclude from \eqref{hl_1} that $R_\zeta$ is a pseudo-resolvent, and from \eqref{hl_2} that its null-set is $\{0\}.$ Therefore, $R_\zeta$ is the resolvent of some closed operator $\Lambda$ in $\mathcal H_0$, and $\Lambda = R_\zeta^{-1} -\zeta$ (Appendix \ref{hille_theory_sect}, Theorem \ref{hille_thm1}). It is also seen that $D(\Lambda) := \hat{R}_\lambda \mathcal H_0$, $\Lambda f := \hat \Lambda f, \; f \in D(\Lambda).$

 Next, let us show that $\Lambda$ is a densely defined operator. Indeed, by the construction, $\mathcal H_0$ is a dense subspace of $\mathcal H_{-\frac{1}{4}},$ and hence $\hat{R}_\lambda \mathcal H_0$ is a dense subspace of $\mathcal H_\frac{3}{4}$ and of $\mathcal H_0$ for $\mathcal H_\frac{3}{4}$ is a dense subspace of $\mathcal H_0.$ Thus $D(\Lambda)$ is dense in $\mathcal H_0.$ 

 Taking into account \eqref{hl_3}, we conclude that $-\Lambda$ is the generator of a quasi bounded holomorphic semigroup. Thus (\textit{i}) and (\textit{iii}) are established. (\textit{ii}) is an easy consequence of (\textit{i}).

\smallskip

\textbf{2}.~In order to justify (\textit{iv}) and (\textit{v}), define $P_\zeta(b) := - J_\zeta \hat B J^3_\zeta \upharpoonright L^2 = -H^* S \; \; (\Real \zeta \geq \lambda).$ 
Let $\{b_n\}$ be given by 
\begin{equation}
\label{b_n_2}
b_n:=\mathbf{1}_n b, \quad n=1,2,\dots,
\end{equation}
where $\mathbf{1}_n$ is the indicator of $\{x \in \mathbb R^d \mid |b(x)| \leq n\}$.
Then $b_n \in \mathbf F_\delta^{\scriptscriptstyle 1/2}$, and so 
$\Lambda(b_n)$ is well defined. 
Since for every $f \in L^2$
\begin{align*}
\|P_\zeta (b)f - P_\zeta (b_n) f \|_2 & \leq \|H^* \|_{2 \rightarrow 2} \|(\mathbf 1 - \mathbf 1_n) |b|^\frac{1}{2} | \nabla J_\zeta^\frac{3}{4} f| \|_2 \\
& \leq \sqrt{\delta} \|(\mathbf 1 - \mathbf 1_n) K \|_2, \quad \quad K = |b|^\frac{1}{2} | \nabla J_\zeta^\frac{3}{4} f| \in L^2,
\end{align*}
it is seen that
\begin{equation}
\label{hl_4}
s \mbox{-} L^2 \mbox{-} \lim_n (\zeta + \Lambda(b_n))^{-1} = s \mbox{-} L^2 \mbox{-} \lim_n J_\zeta^3 (1 + P_\zeta(b_n))^{-1} J_\zeta = J_\zeta^3 (1 + P_\zeta (b))^{-1} J_\zeta = (\zeta + \Lambda(b))^{-1}.
\end{equation}
Let $f \in L^2, \;g \in W^{1,2}, \; \phi \in C_c^\infty.$ Clearly, $b^\frac{1}{2} \cdot \nabla (\zeta + \Lambda(b))^{-1} f \in L^2, \; |b|^\frac{1}{2} \in L^2_\loc,$ and
\begin{align*}
\langle \Lambda(b_n)(\zeta + \Lambda(b_n))^{-1} f, g \rangle & \rightarrow \langle \Lambda(b)(\zeta + \Lambda(b))^{-1} f, g \rangle;
\end{align*}
\begin{align*}
\langle \Lambda(b_n)(\zeta + \Lambda(b_n))^{-1} f, g \rangle & = \langle \nabla (\zeta + \Lambda(b_n))^{-1} f, \nabla g \rangle +  \langle b_n^\frac{1}{2} \cdot \nabla (\zeta + \Lambda(b_n))^{-1} f, |b|_n^\frac{1}{2} g \rangle \\
& \rightarrow \langle \nabla (\zeta + \Lambda(b))^{-1} f, \nabla g \rangle + \langle b^\frac{1}{2} \cdot \nabla (\zeta + \Lambda(b))^{-1} f, |b|^\frac{1}{2} g \rangle;
\end{align*}
\begin{align*}
\langle \Lambda(b_n)(\zeta + \Lambda(b_n))^{-1} f, \phi \rangle & = \langle (\zeta + \Lambda(b_n))^{-1} f, -\Delta \phi \rangle + \langle b_n^\frac{1}{2} \cdot \nabla (\zeta + \Lambda(b_n))^{-1} f, |b|_n^\frac{1}{2} \phi \rangle \\
  & \rightarrow \langle (\zeta + \Lambda(b))^{-1} f, - \Delta \phi \rangle + \langle b^\frac{1}{2} \cdot \nabla (\zeta + \Lambda(b))^{-1} f, |b|^\frac{1}{2} \phi \rangle\\
& = \langle (\zeta + \Lambda(b))^{-1} f, - \Delta \phi \rangle + \langle b \cdot \nabla (\zeta + \Lambda(b))^{-1} f, \phi \rangle.
\end{align*}
Thus $\langle u, \Delta \phi \rangle = \langle - \Lambda u + b \cdot \nabla u, \phi \rangle$ whenever $u \in D(\Lambda).$ This means that $D(\Lambda) \subset \mathcal W^{2,1}_\loc.$
\end{proof}

\subsection{The Hille-Trotter approach. $L^2$-theory}
\label{alt_proof_sect}
\label{ht_sect}
We give an alternative proof of Theorem \ref{thm:markweak}.

Let $\{b_n\}$ be given by \eqref{b_n_2} in section \ref{hl_sect}. Since  $b_n \in \mathbf F_\delta^{\scriptscriptstyle 1/2},$ we have $\|b_n \cdot \nabla (\eta - \Delta)^{-1} \|_{2 \rightarrow 2} < 1$ whenever $\eta > n^2.$ Therefore, by the Miyadera Perturbation Theorem, $- \tilde \Lambda (b_n) := \Delta - b_n \cdot \nabla$ of domain $W^{2,2}$ generates a $C_0$ semigroup in $L^2$.

We can construct $\Lambda(b)$ by showing first that the resolvent set of $-\tilde \Lambda(b_n)$ contains $\mathcal O = \{ \zeta \mid \Real \zeta \geq \lambda \},$ and that there is a constant $c$ such that, for all $n = 1, 2, \dots,$
\[
\|(\zeta + \tilde \Lambda(b_n))^{-1} \|_{2 \rightarrow 2} \leq c |\zeta|^{-1}, \; \zeta \in \mathcal O .
\]
We accomplish this as follows. Define
$$
\Theta(\zeta, b_n) := J_\zeta^3 (1 + P_\zeta(b_n))^{-1} J_\zeta, \quad \zeta \in \mathcal O,
$$
where $P_\zeta(b_n) := J_\zeta |b|_n^\frac{1}{2} \; b^\frac{1}{2}\cdot\nabla J_\zeta^\frac{3}{4} \in \mathcal B (L^2)$. We prove consecutively that the operator-valued function $\Theta(\zeta, b_n) $
 possesses the following properties: 
\[
\Theta(\eta, b_n) f = (\eta + \tilde \Lambda(b_n))^{-1} f, \; f \in L^2, \text{ whenever } \eta > n^2 \vee \lambda; \tag{$p_1$}
\]
\[
\|\Theta(\zeta, b_n) \|_{2 \rightarrow 2} \leq c |\zeta|^{-1} \text{ for some constant } c \text{ and all } n=1,2, \dots ; \tag{$p_2$}
\]
\begin{equation}
\label{p3}
\Theta(\zeta, b_n) - \Theta(\eta, b_n) = (\eta - \zeta) \Theta(\zeta, b_n) \Theta(\eta, b_n), \;\; \;\eta \in \mathcal O. \tag{$p_3$}
\end{equation}
Note that $(p_1)$ follows from the definitions of $\Theta(\zeta, b_n), \;\tilde \Lambda(b_n),$ and from the obvious equality 
\[
\Theta(\eta, b_n) f = J_\eta^4 f - J_\eta^4 b_n \cdot \nabla J_\eta^4 f + J_\eta^4 b_n \cdot \nabla J_\eta^4 b_n \cdot \nabla J_\eta^4 f + \dots = (\eta + \tilde{\Lambda}(b_n))^{-1} f, \qquad f \in L^2,
\]
while $(p_2)$ follows from the definition of $\Theta(\zeta, b_n).$ $(p_3)$ says that $\Theta(\zeta, b_n)$ is a pseudo-resolvent. But then the range of $\Theta(\zeta, b_n)$ equals to the range of $\Theta(\eta, b_n)$ for all $\eta \in \mathcal O,$ and hence is dense in $L^2$ by $(p_1)$.
Thus the properties $(p_1),(p_3)$ mean that
\begin{equation}
\label{theta_id}
\tag{$\diamond$}
\Theta(\zeta, b_n)=(\zeta+\tilde \Lambda(b_n))^{-1}, 
\end{equation}
and hence $\mathcal O \subset \rho(-\tilde \Lambda(b_n))$ and the semigroup $e^{-t \tilde \Lambda(b_n)}$ is holomorphic (due to $(p_2)$).

Finally, on the basis of the Trotter Approximation Theorem, by proving that $\mu \Theta(\mu,b_n) \overset{s} \rightarrow 1$ as $\mu \uparrow \infty$ uniformly in $n$, we conclude, using $\Theta(\zeta, b_n) \overset{s} \rightarrow \Theta(\zeta,b)$ (see \eqref{hl_4}), that $\Theta(\zeta, b)$ is indeed the resolvent of an operator $-\Lambda(b),$ which generates a holomorphic ($\| \Theta(\zeta,b) \|_{2 \rightarrow 2} \leq c |\zeta|^{-1}$) semigroup.

The proof that $\mu \Theta(\mu,b_n) \overset{s} \rightarrow 1$ as $\mu \uparrow \infty$ uniformly in $n$ is carried out as follows. Set $H_n = |b|_n^\frac{1}{2} J_\mu$ and $S_n = b_n^\frac{1}{2} \cdot \nabla J_\mu^3.$ Since $s \mbox{-} L^2 \mbox{-} \lim_{\mu \uparrow \infty} \mu J^4_\mu = 1$ and $\sup_n\| \Theta(\mu, b_n) \|_{2 \rightarrow 2} \leq c \mu^{-1},$
\[
\lim_{\mu \uparrow \infty} \mu \sup_n \|J_\mu^3 H_n^*(1 + S_n H_n^*)^{-1} S_n J_\mu f \|_2 = 0 \quad \quad (f \in C_c^\infty)
\]
needs only to be proved (see assertion (\textit{ii}) of the theorem).
Noticing that
\[
\|S_n J_\mu f \|_2 = \|b_n^\frac{1}{2} \cdot \nabla J_\mu^4 f \|_2 \leq \sqrt{\delta} \mu^{-\frac{3}{4}} \| \nabla f \|_2 \text{ and } \| H_n^* (1 + S_n H_n^*)^{-1} \|_{2 \rightarrow 2} \leq \sqrt{\delta} (1- \delta)^{-1},
\]
we obtain
\[
\|J_\mu^3 H_n^*(1 + S_n H_n^*)^{-1} S_n J_\mu f \|_2 \leq \delta(1-\delta)^{-1} \mu^{-\frac{3}{2}} \| \nabla f \|_2.
\]
The desired convergence follows.
 
It remains to prove $(p_3)$.

\begin{proof} [Proof of $(p_3)$] Set $F_\zeta := (\zeta -\Delta)^{-1}$ and define
\[
N_\zeta^k:= (-1)^k F_\zeta b_n \cdot \nabla F_\zeta \dots b_n \cdot \nabla F_\zeta, \qquad 1 \leq k := \# \;  b_n \text{'s}, \qquad N_\zeta^0:= F_\zeta.\]
Obviously,
\[
\Theta(\zeta, b_n) = \sum_{k=0}^\infty N_\zeta^k \text{ ( absolutely convergent series in } L^2 ),
\]
\[
\Theta(\zeta, b_n) \Theta(\eta, b_n) = \sum_{l=0}^\infty \sum_{i=0}^l N_\zeta^i N_\eta^{l-i}, \quad \quad \zeta, \eta \in \mathcal O. \tag{$\star$}
\]
Set $\#$ = ``number of''. Define
\begin{align*}
I_{l,m}^k (\zeta, \eta) & := F_\zeta b_n \cdot \nabla F_\zeta \dots b_n \cdot \nabla F_\zeta F_\eta b_n \cdot \nabla F_\eta \dots b_n \cdot \nabla F_\eta, \\
& l := \# \; \zeta \text{'s}, \qquad m := \# \; \eta \text{'s}, \qquad k := \# \; b_n \text{'s}.
\end{align*}
Substituting the identity $F_\zeta F_\eta = (\eta - \zeta)^{-1} \big( F_\zeta - F_\eta \big)$ inside the product
\[
N_\zeta^k N_\eta^m  = (-1)^{k+m} F_\zeta b_n \cdot \nabla F_\zeta \dots b_n \cdot \nabla F_\zeta F_\eta b_n \cdot \nabla F_\eta \dots b_n \cdot \nabla F_\eta, 
\]
we obtain $N_\zeta^k N_\eta^m = (\eta -\zeta)^{-1} (-1)^{k+m} \bigg[ I_{k+1,m}^{k+m} - I_{k, m+1}^{k+m} \bigg].$ Therefore,
\begin{align*}
(\eta - \zeta) \sum_{i=0}^l N_\zeta^i N_\eta^{l-i} = &  (-1)^l \bigg[I^l_{1,l} - I^l_{0,l+1} + I^l_{2,l-1} - I^l_{1,l} + \dots + I^l_{l+1,0} - I^l_{l.1} \bigg] \\
= & (-1)^l \big( I^l_{l+1,0} - I^l_{0,l+1} \big).
\end{align*}
Substituting the last identity in the RHS of $(\star)$ we obtain
\[
(\eta-\zeta)\Theta(\zeta, b_n) \Theta(\eta, b_n)  = \sum_{l=0}^\infty (-1)^l \big( I^l_{l+1,0} - I^l_{0,l+1} \big) = \Theta(\zeta,b_n) - \Theta(\eta,b_n).
\] 
The proof of ($p_3$) is completed.
\end{proof}

It follows from \eqref{theta_id} that $$-\Lambda(b_n)=-\tilde{\Lambda}(b_n)\,\,\,(=\Delta - b_n\cdot\nabla,\,D(-\Lambda(b_n))=W^{2,2}).$$ 
The latter and \eqref{hl_4} yield

\begin{corollary}
\label{cor}
 In the assumptions of Theorem \ref{thm:markweak} we have
\begin{equation*}
(\zeta + \Lambda(b_n))^{-1} \overset{s}{\rightarrow} (\zeta+\Lambda(b))^{-1} \quad (\zeta \in \mathcal O).
\end{equation*}
\end{corollary}

\subsection{($L^p \rightarrow L^q)$ estimates} 
\label{LpLq_est_sect}

\begin{theorem}
\label{thm:markembed}
Let $d \geq 3$, $b:\mathbb R^d \rightarrow \mathbb R^d$, $b \in \mathbf F^{\scriptscriptstyle 1/2}_\delta$, $\delta < 1$. 
Let $\Lambda(b)$ be the operator realization of $-\Delta + b\cdot\nabla$ in $L^2$ constructed in Theorem \ref{thm:markweak}. Let $\{b_n\}$ be given by \eqref{b_n_2} in section \ref{hl_sect}.
There is a family $\{ e^{-t \Lambda_r(b)}, t \geq 0 \}_{2 \leq r < \infty}$ of consistent, positivity preserving, $L^\infty$ contraction $C_0$ semigroups such that $e^{-t \Lambda_r(b)} \upharpoonright L^r \cap L^2 = e^{-t \Lambda(b)} \upharpoonright L^r \cap L^2;$
\[
s \mbox{-} L^r \mbox{-} \lim_n e^{-t \Lambda_r(b_n)} = e^{-t \Lambda_r(b)} \quad \text{{\rm(}uniformly on every compact interval of $t${\rm)}}; \tag{\textit{i}}
\]
\[
\|e^{-t \Lambda_r} f \|_q \leq c_r e^{t \omega_r} t^{-\frac{d}{2}\big(\frac{1}{r} - \frac{1}{q} \big)} \|f \|_r, \; \; c_r = c_r(\delta,c_2), \; \; \omega_r = \frac{2 \lambda_\delta}{r} \quad \quad (f \in L^r, \; 2 \leq r < q \leq \infty). \tag{\textit{ii}}
\]
\end{theorem}
\begin{proof}

Since $b_n \in \mathbf F_\delta^\frac{1}{2},$ $\Lambda(b_n)$ is well defined. On the other hand $\|b_n \cdot \nabla (\eta - \Delta)^{-1} \|_{2 \rightarrow 2} < 1$ whenever $\eta > n^2.$ Therefore, by the Miyadera Perturbation Theorem, $- \tilde \Lambda (b_n) := \Delta - b_n \cdot \nabla$ of domain $W^{2,2}$ generates the $C_0$ semigroup $e^{-t\tilde \Lambda(b_n)}$ in $L^2,$ which preserves positivity and is $L^\infty$ contraction, and so is $e^{-t\Lambda(b_n)}$ since, for $\eta > n^2 \vee \lambda$ and $ f \in L^2, \; (\eta + \Lambda(b_n))^{-1} f = (\eta + \tilde{\Lambda}(b_n))^{-1} f.$
The convergence $(\zeta + \Lambda(b_n))^{-1} \overset{s}\rightarrow (\zeta + \Lambda(b))^{-1}$ yields (\textit{i}) for $r=2,$ and hence for each $r > 2.$

Proof of (\textit{ii}) Our strategy is as follows: 1) We prove that, for some $r \in ]2, \infty[,$
\[
\| e^{-t \Lambda} \|_{2 \rightarrow r} \leq C_T t^{-\frac{d}{2} \big(\frac{1}{2} - \frac{1}{r} \big)}  \quad \quad ( 0 < t < T < \infty ). \tag{$\star$} 
\]
2) Using the extrapolation between ($\|e^{-t\Lambda} f \|_\infty \leq \|f\|_\infty, \; f \in L^2 \cap L^\infty$) and $(\star)$, we conclude that
\[
\|e^{-t\Lambda} \|_{2 \rightarrow \infty} \leq \tilde{C}_T t^{-\frac{d}{4}}
\]
(see Appendix \ref{appendix_B}). 
Then (\textit{ii}) follows from the Riesz Interpolation Theorem.

Since 2) is straightforward, we need only to prove $(\star).$ Without loss of generality we may suppose that $\lambda \geq 1.$ Set $\Gamma_0 = \lambda - \Delta$ and $\Gamma = \lambda + \Lambda.$ Due to $|(\eta + \Gamma_0)^{-\beta} f| \leq \Gamma_0^{-\beta} |f|$ and $\|(\eta + \Gamma_0)^{-\beta} \|_{2 \rightarrow 2} \leq (1+ |\eta|)^{-\beta} \; (0 < \beta < 1, \; \Real \eta \geq 0),$ we have, using Theorem \ref{thm:markweak}(\textit{i}),
\[
\|(\eta + \Gamma)^{-1} \|_{2 \rightarrow 2} \leq (1-\delta)^{-1} \|(\eta + \Gamma_0)^{-\frac{3}{4}} \|_{2 \rightarrow 2} \|(\eta + \Gamma_0)^{-\frac{1}{4}} \|_{2 \rightarrow 2} \leq (1-\delta)^{-1} (1+ |\eta|)^{-1}.
\]
Also, for any $r \in ]2, \infty [$ and $\mu \geq 0, \;\; \|(\mu + \Gamma_r)^{-1} \|_{r \rightarrow r} \leq c_r (1+\mu)^{-1}.$

Let $\Gamma_r^{-\beta}$ and $\Gamma_r^\beta := \big( \Gamma_r^{-\beta} \big)^{-1}$ be the fractional powers of $\Gamma_r.$ Below we will need the following well known results:
\[
\Gamma_r^{-\beta} = \frac{\sin \pi \beta}{\pi} \frac{1}{1-\beta} \int_0^\infty \mu^{1-\beta} (\mu + \Gamma_r)^{-2} d\mu;
\]

\[
\|\Gamma^\beta (\mu + \Gamma)^{-1} \|_{2 \rightarrow 2} \leq c \; (1+ \mu)^{-1 +\beta};
\]

$$
\|\Gamma^\beta e^{-t \Gamma} \|_{2 \rightarrow 2} \leq c(\beta) t^{-\beta}
$$
(see e.g. \cite[Ch.\,4]{KZPS}).

Fix $\beta = \frac{1}{4}, \; r = \frac{2 d}{d-1}.$ Let $F_t = \Gamma^{2 \beta} e^{-t \Gamma} f, \; f \in L^2 \cap L^r.$ We have
\[
\|e^{-t \Gamma_r} f \|_r = \| \Gamma_r^{-2 \beta} F_t \|_r \leq \frac{2}{\pi} \int_0^\infty \mu^{1-2 \beta} \|(\mu + \Gamma_r)^{-2} F_t \|_r d \mu.
\]
By the embedding $(\mu +\Gamma_0)^{-\beta} L^2 \subset L^r,$
\[
\| (\mu + \Gamma_r)^{-2} F_t \|_r \leq c(d) \|(\mu + \Gamma_0)^{-\frac{1}{2}} ( 1 + P_\mu (b))^{-1} (\mu + \Gamma_0)^{-\frac{1}{4}} (\mu + \Gamma)^{-1} F_t \|_2,
\]
and hence $\|(\mu + \Gamma_r)^{-2} F_t \|_r \leq c(d)(1-\delta)^{-1} \mu^{-3 \beta} \|(\mu + \Gamma)^{-1} F_t \|_2 .$

Thus
\begin{align*}
\| e^{-t \Gamma_r} f \|_r & \leq C \int_0^\infty \mu^{-\beta} \|(\mu + \Gamma)^{-1} F_t \|_2 \; d \mu\\
&\leq C_1 \bigg ( \int_0^\frac{1}{t} \mu^{-\beta} (1 + \mu)^{2 \beta -1} d \mu \;\|f\|_2 + \int_\frac{1}{t}^\infty \mu^{-\beta-1} d \mu \; \|\Gamma^{2 \beta} e^{-t \Gamma} f \|_2 \bigg) \\
& \leq C_2 \bigg ( \int_0^\frac{1}{t} \mu^{-\frac{3}{4}} d \mu + \int_\frac{1}{t}^\infty \mu^{-\frac{5}{4}} d \mu \; t^{-2 \beta} \bigg) \|f\|_2 = 4 C_2 \; t^{-\beta} \|f\|_2,
\end{align*}
which ends the proof of $(\star).$
\end{proof}

\subsection{$L^r$-theory and $\mathcal W^{1+\frac{1}{q},r}$-estimates on solutions to $(\mu - \Delta + b \cdot \nabla ) u = f$,  $b \in \mathbf F^{\scriptscriptstyle 1/2}_\delta$} 

\label{Lr_sect2}

As in the case of $b \in \mathbf F_\delta$, it is reasonable to expect that there is a quantitative dependence between the value of $\delta$ and the smoothness of $D(\Lambda_r(b))$.

Set
\[
m_d:= \pi^\frac{1}{2}(2 e)^{-\frac{1}{2}} d^\frac{d}{2} (d-1)^{-\frac{d-1}{2}}, \quad
c_r := \frac{r r^\prime}{4}, \quad \kappa_d := \frac{d}{d-1}, 
\]
\[
 r_\mp:= \frac{2}{1 \pm \sqrt{1-m_d \delta}}.
\]
\begin{theorem}
\label{thm:markgrad2}
Let $d\geq 3, \; b:\mathbb R^d \rightarrow \mathbb R^d.$ Assume that $b \in \mathbf F_\delta^{\scriptscriptstyle 1/2}$, $m_d \delta < 1.$ Then $\bigl(e^{-t \Lambda_r(b)}, r \in [2,\infty[ \bigr)$ extends by continuity to a quasi bounded $C_0$ semigroup in $L^r$ for all $r \in ]r_-, \infty [.$ 

For every $r \in I_s := ]r_-, r_+[$ we have:
 
{\rm (\textit{i})} The resolvent set $\rho(-\Lambda_r(b))$ contains the half-plane $\mathcal O := \{ \zeta \in \mathbb C \mid \Real \zeta \geq \kappa_d \lambda_\delta \},$ and the resolvent admits the representation
 \[
 (\zeta + \Lambda_r(b))^{-1} = \Theta_r(\zeta, b), \quad \zeta \in \mathcal O,
 \]
 where
 \[
 \Theta_r(\zeta, b) := (\zeta - \Delta)^{-1} - Q_r (1 + T_r)^{-1} G_r,
 \]
 the operators $Q_r, G_r, T_r \in \mathcal B(L^r),$
\begin{equation}
\label{_star}
 \|Q_r\|_{r\rightarrow r} \leq C_1 |\zeta|^{-\frac{1}{2}-\frac{1}{2 r}}, \quad \|G_r\|_{r\rightarrow r} \leq C_2 |\zeta|^{-\frac{1}{2 r^\prime}}, \quad \|T_r\|_{r\rightarrow r} \leq m_d c_r \delta, \tag{$\star$}
\end{equation}
 \[
 G_r \equiv G_r(\zeta, b):= b^\frac{1}{r} \cdot \nabla (\zeta - \Delta)^{-1}, \quad b^\frac{1}{r}:= |b|^{\frac{1}{r}-1} b,
 \]
 $Q_r, \; T_r$ are the extensions by continuity of densely defined (on $\mathcal E := \bigcup_\epsilon e^{-\epsilon |b|} L^r )$ operators
 \[
 Q_r \upharpoonright \mathcal E \equiv Q_r(\zeta, b) \upharpoonright \mathcal E := (\zeta-\Delta)^{-1} |b|^\frac{1}{r^\prime} , \quad T_r\upharpoonright \mathcal E \equiv T_r(\zeta, b)\upharpoonright\mathcal E := b^\frac{1}{r} \cdot \nabla (\zeta - \Delta)^{-1} |b|^\frac{1}{r^\prime} .
 \]

 {\rm (\textit{ii})} It follows from (i) that $e^{-t \Lambda_r(b)}$ is holomorphic: there is a constant $C_r$ such that
 \[
 \|(\zeta + \Lambda_r(b))^{-1} \|_{r\rightarrow r} \leq C_r |\zeta|^{-1}, \quad \zeta \in \mathcal O.
 \]

  {\rm (\textit{iii})} For all $p < r < q$ and $\zeta \in \mathcal O,$ define
 \begin{align*}
 & G_r(p) \equiv G_r(p, \zeta, b):= b^\frac{1}{r}\cdot \nabla (\zeta-\Delta)^{-\frac{1}{2}-\frac{1}{2 p}}, \\
 & Q_r(q) \equiv Q_r(q, \zeta, b):= (\zeta-\Delta)^{-\frac{1}{2 q^\prime}} |b|^\frac{1}{r^\prime} \text{ on } \mathcal E.
 \end{align*}
 Then $G_r(p) \in \mathcal B(L^r). \; Q_r(q)$ extends by continuity to a bounded operator in $L^r.$ Its extension we denote again by $Q_r(q)$. 
 
\smallskip 
 
 {\rm (\textit{iv})} For every $\zeta \in \mathcal O,$
 \begin{align*}
& \Theta_r(\zeta, b) = (\zeta-\Delta)^{-1} - (\zeta-\Delta)^{-\frac{1}{2}-\frac{1}{2 q}} Q_r(q) (1 + T_r)^{-1} G_r(p) (\zeta-\Delta)^{-\frac{1}{2 p^\prime}} ; \\
& \Theta_r(\zeta, b) \text{ extends by continuity to an operator in } \mathcal B(\mathcal W^{-\frac{1}{p^\prime},r}, \mathcal W^{1+\frac{1}{q},r}).
 \end{align*}
 
 {\rm (\textit{v})} By {\rm (\textit{i})} and {\rm (\textit{iv})}, $D(\Lambda_r(b)) \subset \mathcal W^{1+\frac{1}{q},r} \;(q> r).$ In particular, if $m_d \delta < 4 \frac{d-2}{(d-1)^2},$ then
there exists $r \in I_s, \; r > d - 1,$ such that $D(\Lambda_r(b)) \subset C^{0, \gamma}, \gamma < 1 - \frac{d-1}{r}.$ 

\smallskip
{\rm (\textit{vi})} $D(\Lambda_r(b)) \subset \mathcal W_\loc^{2,1}$ and
\[
\langle \Lambda_r(b) u, v \rangle = \langle u,-\Delta v \rangle + \langle b \cdot \nabla u, v \rangle, \quad u \in D(\Lambda_r(b)), \quad v \in C^\infty_c(\mathbb R^d) .
\] 
\smallskip

 {\rm (\textit{vii})} Let $\{b_n\}$ be given by \eqref{b_n_2}, then $$e^{-t \Lambda_r(b)}=s\mbox{-}L^r\mbox{-}\lim_n e^{- t \Lambda_r(b_n)} \quad \text{{\rm(}uniformly on every compact interval of $t \geq 0${\rm)}, }
$$

\smallskip
 {\rm (\textit{viii})} $(e^{-t \Lambda_r(b)}, t>0)$ preserves positivity and
\[
\|e^{-t \Lambda_r(b)} f\|_p \leq \tilde{C}_r t^{-\frac{d}{2}(\frac{1}{r} - \frac{1}{p})}\|f\|_r, \; 0 < t \leq 1, \; r_-< r < p\leq \infty, \; f\in L^r. 
\] 
\end{theorem}

\begin{proof}
1.~Suppose that \eqref{_star} is established. 

Then $\Theta_r(\zeta, b)\upharpoonright L^2 \cap L^r= \Theta_2(\zeta, b)\upharpoonright L^2 \cap L^r, \; \zeta \in \mathcal O,$ and (\textit{i}),  (\textit{vi}) follow immediately from Theorem \ref{thm:markweak}. 

 (\textit{vii}) is a simple consequence of Corollary \ref{cor} and the bound $\|e^{-t \Lambda_{r_0}} \|_{r_0 \rightarrow r_0}\leq C_{r_0}=C_{r_0}(\delta,r_--r_0) \; (t\leq 1, \; r_-<r_0 <2).$ The latter is also needed for the proof of (\textit{viii}) (cf. Theorem \ref{thm:markembed}).
 
(\textit{iv}) is an obvious consequence of (\textit{i}) $+$ (\textit{iii}). 

\smallskip

It remains to prove \eqref{_star} and (\textit{iii}). 

2.~\textit{Proof of \eqref{_star}.} Let $r \in ]1, \infty[.$ We will need

\smallskip

$(\mathbf{a})$ $\mu \geq \lambda \Rightarrow \| |b|^\frac{1}{r} (\mu - \Delta)^{-\frac{1}{2}} \|_{r\rightarrow r} \leq (c_r \delta)^\frac{1}{r} \mu^{-\frac{1}{2 r^\prime}}$ (recall $c_r = \frac{r r^\prime}{4}$).

Indeed, in $L^2$ define $A = (\mu - \Delta)^\frac{1}{2}, \; D(A) = W^{1,2}.$ Since $-(A - \mu^\frac{1}{2})$ is a symmetric Markov generator,  for any $r \in ]1, \infty[,$
\[
0 \leq u \in D(A_r) \Rightarrow v := u^\frac{r}{2} \in D(A^\frac{1}{2}) \text{ and } c_r^{-1} \|A^\frac{1}{2}v \|_2^2 \leq \langle A_r u, u^{r-1} \rangle.
\]
Now, let $u = A_r^{-1} |f|, \; f\in L^r.$ Clearly, $\|u\|_r \leq \mu^{-\frac{1}{2}} \|f\|_r.$ Since $b \in \mathbf F_\delta^{\scriptscriptstyle 1/2}$ we have
\[
(c_r \delta)^{-1} \| |b|^\frac{1}{2}v\|_2^2 \leq \langle A_r u, u^{r-1} \rangle = \langle f, u^{r-1} \rangle,
\]
and so $\| |b|^\frac{1}{r} u \|_r^r \leq c_r \delta \|f\|_r \|u\|_r^{r-1}, \; \| |b|^\frac{1}{r}A_r^{-1} f\|_r^r \leq c_r \delta \mu^{-\frac{r-1}{2}} \|f\|_r^r.$ $(\mathbf{a})$ is proved.

$(\mathbf{b})$ $\mu \geq \lambda \Rightarrow \| |b|^\frac{1}{r} (\mu - \Delta)^{-\frac{1}{2}} |b|^\frac{1}{r^\prime} f\|_r \leq c_r \delta  \|f\|_r, \; f \in \mathcal E.$

Indeed, let $u = A_r^{-1} |b|^\frac{1}{r^\prime} |f|, \; f \in \mathcal E.$ Then arguing as in $(\mathbf{a})$ we have
\[
\| |b|^\frac{1}{r} u \|_r^r \leq c_r \delta \|f\|_r \| |b|^\frac{1}{r} u \|_r^{r-1}, \text{ or } \| |b|^\frac{1}{r}u \|_r \leq c_r \delta \|f\|_r.
\]
Thus $(\mathbf{b})$ is proved.

$(\mathbf{c}) \quad \mu \geq \lambda \Rightarrow \|(\mu-\Delta)^{-\frac{1}{2}} |b|^\frac{1}{r^\prime} f\|_r \leq (c_{r^\prime} \delta)^\frac{1}{r^\prime} \mu^{-\frac{1}{2 r}} \|f\|_r, \; f \in \mathcal E.$ 

Indeed, $(\mathbf{c})$ follows from $(\mathbf{a})$ by duality.

\smallskip

We are in position to complete the proof of \eqref{_star}. Using $(\mathbf{b})$ and the pointwise bound
\begin{equation}
\label{a1}
|\nabla (\zeta - \Delta)^{-1}(x,y)| \leq m_d (\lambda - \Delta)^{-\frac{1}{2}}(x,y), \; \Real \zeta \geq \kappa_d \lambda \tag{$A_1$}
\end{equation}
(see proof of \eqref{a1} below), we have, for every $r \in I_s$ and $\mu = \lambda,$
\[
\| T_r f \|_r \leq m_d \| |b|^\frac{1}{r}(\lambda - \Delta)^{-\frac{1}{2}} |b|^\frac{1}{r^\prime} |f| \|_r \leq m_d c_r \delta \|f\|_r, \;\; f \in \mathcal E.
\] 
Note that $m_d c_r \delta < 1$ due to $r \in I_s.$

Next, using $(\mathbf{a})$ and the pointwise bound
\begin{equation}
\label{a3}
|\nabla (2 \kappa_d \zeta -\Delta)^{-1}(x, y)| \leq 2^\frac{d}{2} m_d(|\zeta|-\Delta)^{-\frac{1}{2}}(x, y), \; \Real \zeta > 0 \tag{$A_3$}
\end{equation}
(see proof of \eqref{a3} below), we have, for every $r \in I_s, \; \zeta \in \mathcal O$ and $\mu = |\zeta|,$
\[
\|G_r(2\kappa_d \zeta, b)\|_{r\rightarrow r} \leq 2^\frac{d}{4}(c_r \delta)^\frac{1}{r} m_d |\zeta|^{-\frac{1}{2 r^\prime}}.
\]
Now, using the identity $(\zeta-\Delta)^{-1} = (2 \kappa_d \zeta - \Delta)^{-1}\big( 1 + (2\kappa_d -1) \zeta (\zeta-\Delta)^{-1} \big),$ we obtain
\[
\|G_r(\zeta, b)\|_{r\rightarrow r} \leq 2^\frac{d+4}{4} \kappa_d m_d (c_r \delta)^\frac{1}{r} |\zeta|^{-\frac{1}{2 r^\prime}} \equiv C_2 |\zeta|^{-\frac{1}{2 r^\prime}}.
\]
Similarly, using $(\mathbf{c})$ and the pointwise bound
\begin{equation}
\label{a4}
|(2 \zeta -\Delta)^{-\frac{1}{2}}(x, y)| \leq 2^\frac{d+1}{4} (|\zeta|-\Delta)^{-\frac{1}{2}}(x, y), \; \Real \zeta > 0 \tag{$A_4$}
\end{equation}
(see proof of \eqref{a4} below), we have, for every $r \in I_s$ and $\Real \zeta \geq \lambda,$
\begin{align*}
\|Q_r(2 \zeta, b)f\|_r & \leq \|(2 \zeta -\Delta)^{-\frac{1}{2}}\|_{r\rightarrow r} \|(2 \zeta - \Delta)^{-\frac{1}{2}} |b|^\frac{1}{r^\prime} f \|_r\\
& \leq |2 \zeta|^{-\frac{1}{2}} 2^\frac{d+1}{4} (c_{r^\prime} \delta)^\frac{1}{r^\prime} |\zeta|^{-\frac{1}{2 r}} \|f\|_r, \; f \in \mathcal E. 
\end{align*}
Finally, using the identity $(\zeta-\Delta)^{-1} = \big( 1 + \zeta (\zeta-\Delta)^{-1} \big)(2 \zeta - \Delta)^{-1},$ we obtain
\begin{align*}
\|Q_r(\zeta, b)f\|_r & \leq \|1+\zeta (\zeta-\Delta)^{-1}\|_{r\rightarrow r}\|Q_r(2 \zeta, b)f\|_r\\
& \leq 2^\frac{d+3}{4} (c_{r^\prime} \delta)^\frac{1}{r^\prime} |\zeta|^{-\frac{1}{2}-\frac{1}{2 r}} \|f\|_r, \\
& = C_1 |\zeta|^{-\frac{1}{2}-\frac{1}{2 r}} \|f\|_r, \qquad f \in \mathcal E,
\end{align*}
completing the proof of \eqref{_star}.

3.~\textit{Proof of {\rm(\textit{iii})}.} We have to deduce the bounds on $\|Q_r(q)\|_{r\rightarrow r}, \|G_r(p)\|_{r\rightarrow r}.$ Let $\Real \zeta \geq \lambda$ and $r < q.$ Using $(\mathbf{c})$ and the formula
\[
(z -\Delta)^{-\alpha} = \frac{\sin \pi \alpha}{\pi} \int_0^\infty t^{-\alpha} (t +z -\Delta)^{-1} d t \qquad 0 < \alpha <1, \; \Real z > 0,      
\]
we obtain
\begin{align*}
\|Q_r(q)f\|_r & \leq \|(\Real \zeta -\Delta)^{-\frac{1}{2 q^\prime}} |b|^\frac{1}{r^\prime} |f| \|_r\\
& \leq \|(\lambda-\Delta)^{-\frac{1}{2 q^\prime}} |b|^\frac{1}{r^\prime} |f| \|_r\\
& \leq s_{q^\prime} \int_0^\infty t^{-\frac{1}{2 q^\prime}} \|(t+\lambda -\Delta)^{-1} |b|^\frac{1}{r^\prime} |f| \|_r d t \qquad \bigg(s_{q^\prime} = \frac{\sin \frac{\pi}{2 q^\prime}}{\pi} \bigg)\\
& \leq s_{q^\prime} \int_0^\infty t^{-\frac{1}{2 q^\prime}} (t+\lambda)^{-\frac{1}{2}}\|(t+\lambda -\Delta)^{-\frac{1}{2}} |b|^\frac{1}{r^\prime} |f| \|_r d t\\
& \leq s_{q^\prime} (c_{r^\prime} \delta)^\frac{1}{r^\prime}\int_0^\infty t^{-\frac{1}{2 q^\prime}} (t+\lambda)^{-\frac{1}{2}-\frac{1}{2 r}} \|f\|_r d t = K_{1,q} \|f\|_r, \quad f\in \mathcal E,
\end{align*}
where $K_{1,q} < \infty$ due to $q > r.$

Let $\zeta \in \mathcal O$ and $1\leq p < r.$ Similarly, using $(\mathbf{a})$ and the pointwise bound
\begin{equation}
\label{a2}
|\nabla(\zeta-\Delta)^{-1+\frac{1}{2 r}}(x, y)| \leq m_{d,r} (\kappa_d^{-1} \Real \zeta -\Delta)^{-1+\frac{1}{2 r}}(x, y) \tag{$A_2$}
\end{equation}
(see proof of \eqref{a2} below), we obtain
\begin{align*}
\|G_r(p)f\|_r & \leq m_{d,p}\| |b|^\frac{1}{r}(\lambda -\Delta)^{-\frac{1}{2p}} |f| \|_r\\
& \leq m_{d,p} s_p (c_r \delta)^\frac{1}{r}\int_0^\infty t^{-\frac{1}{2p}}(t+\lambda)^{-\frac{1}{2}-\frac{1}{2 r^\prime}} d t \|f\|_r =K_{2,p} \|f\|_r, \quad f \in L^r,
\end{align*}
where $K_{2,p} < \infty$ due to $p < r.$ The proof of (\textit{iii}) is completed.
\end{proof}

4.~We now give  a proof of \eqref{a1}-\eqref{a4}.

\textit{Proof of} \eqref{a1}. Let $\alpha \in ]0, 1[.$ Set $c(\alpha)= \sup_{\zeta > 0} \zeta e^{-(1-\alpha)\zeta^2}\big(=(2e(1-\alpha))^{-\frac{1}{2}} \big).$ Then, for all $\zeta > 0,$
\[
\zeta e^{-\zeta^2} \leq c(\alpha) e^{-\alpha \zeta^2} \tag{$\star$}
\]
Using the formula $(\zeta-\Delta)^{-\beta} = \frac{1}{\Gamma(\beta)} \int_0^\infty e^{-\zeta t} t^{\beta-1} (4\pi t)^{-\frac{d}{2}} e^{-\frac{|x-y|^2}{4t}} d t, \;\; 0 < \beta \leq 1,$ first with $\beta=1,$ and then with $\beta =\frac{1}{2},$ we have
\begin{align*}
|\nabla(\zeta-\Delta)^{-1}(x,y)| & \leq \int_0^\infty e^{-t \Real\zeta}(4\pi t)^{-\frac{d}{2}} \frac{|x-y|}{2 t}e^{-\frac{|x-y|^2}{4t}} d t\\
& \leq c(\alpha) \int_0^\infty e^{-t \Real\zeta} t^{-\frac{1}{2}} (4\pi t)^{-\frac{d}{2}} e^{-\alpha \frac{|x-y|^2}{4t}} d t \quad \bigg(\text{by } (\star) \text{ with } \zeta = \frac{|x-y|}{2 \sqrt{t}} \bigg)\\
& = c(\alpha) \alpha^{-\frac{d-1}{2}} \Gamma(1/2)(\alpha \Real\zeta-\Delta)^{-\frac{1}{2}}(x,y). 
\end{align*}
The minimum of $c(\alpha)\alpha^{-\frac{d-1}{2}} \Gamma(\frac{1}{2})$ in $\alpha \in ]0, 1[$ is attained at $\alpha=\frac{d-1}{d}(= \kappa_d^{-1} ),$ and equals $m_d.$

The proof of \eqref{a2} is similar.

\textit{Proof of} \eqref{a3}. First suppose that $\Imag \zeta \leq 0.$ By Cauchy's theorem,
\[
\int_0^\infty e^{-t \zeta}(4\pi t)^{-\frac{d}{2}} e^{-\frac{|x-y|^2}{4t}} d t = \int_0^\infty e^{- \zeta r e^{i\frac{\pi}{4}}} e^{-i\frac{\pi}{4}\frac{d}{2}}(4\pi r)^{-\frac{d}{2}}e^{-\frac{|x-y|^2}{4re^{i\frac{\pi}{4}}}} e^{i\frac{\pi}{4}} d r
\]
(i.e.\,we have changed the contour of integration from $\{t \mid t \geq 0 \}$ to $\{r e^{i\frac{\pi}{4}} \mid r \geq 0 \}).$ Thus
\[
|\nabla(\zeta-\Delta)^{-1}(x, y)| \leq \int_0^\infty \bigg| e^{- \zeta r e^{i\frac{\pi}{4}}} \bigg| (4\pi r)^{-\frac{d}{2}} \bigg|\frac{x-y}{2 r} \bigg|  \bigg|e^{-\frac{|x-y|^2}{4re^{i\frac{\pi}{4}}}} \bigg| d r,
\]
\[
\bigg| e^{- \zeta r e^{i\frac{\pi}{4}}} \bigg| \leq e{-r \frac{1}{\sqrt{2}}}(\Real\zeta-\Imag \zeta), \;\; \bigg|e^{-\frac{|x-y|^2}{4re^{i\frac{\pi}{4}}}} \bigg| \leq e^{-\frac{|x-y|}{4 r}\frac{1}{\sqrt{2}}}, \;\; \Real\zeta-\Imag \zeta \geq |\zeta|,
\]
\begin{align*}
|\nabla(\zeta-\Delta)^{-1}(x, y)| & \leq \int_0^\infty e^{-\frac{r}{\sqrt{2}}|\zeta|} (4\pi r)^{-\frac{d}{2}} \frac{|x-y|}{2 r} e^{-\frac{|x-y|^2}{4r}} d r,\\
& \leq \frac{2^\frac{d}{2 m_d}}{\Gamma(\frac{1}{2})}\int_0^\infty e^{-\frac{r}{2 \kappa_d}|\zeta|} (4\pi r)^{-\frac{d}{2}} \frac{|x-y|}{2 r} e^{-\frac{|x-y|^2}{4r}} d r,\\
& = 2^\frac{d}{4}m_d \big(\kappa_d^{-1}2^{-1}|\zeta|-\Delta)^{-\frac{1}{2}}(x, y), 
\end{align*}
and so \eqref{a3} for $\Imag \zeta \leq 0$ is proved. The case $\Imag \zeta > 0$ is treated analogously.

\textit{Proof of} \eqref{a4}. First suppose that $\Imag \zeta \leq 0.$ By Cauchy's theorem,
\[
\int_0^\infty e^{-t \zeta} t^{-\frac{1}{2}} (4\pi t)^{-\frac{d}{2}} e^{-\frac{|x-y|^2}{4t}} d t = \int_0^\infty e^{- \zeta r e^{i\frac{\pi}{4}}} r^{-\frac{1}{2}} e^{-i\frac{\pi}{8}}  e^{-i\frac{\pi}{4}\frac{d}{2}}(4\pi r)^{-\frac{d}{2}}e^{-\frac{|x-y|^2}{4re^{i\frac{\pi}{4}}}} e^{i\frac{\pi}{4}} d r,
\]
so we estimate as above:
\begin{align*}
|(\zeta-\Delta)^{-\frac{1}{2}}(x, y)| & \leq \int_0^\infty e^{-\frac{r}{\sqrt{2}}|\zeta|} r^{-\frac{1}{2}} (4\pi r)^{-\frac{d}{2}} e^{-\frac{|x-y|^2}{4r\sqrt{2}}} d r,\\
& = 2^\frac{d+1}{4} (2^{-1}|\zeta|-\Delta)^{-\frac{1}{2}}(x, y).
\end{align*}
The case $\Imag \zeta > 0$ is treated analogously.

\begin{remarks}
1.~In the proof of $G_r(p)$, $Q_r(q) \in \mathcal B(L^p)$
we appeal to the $L^p$ inequalities between the operator $(\lambda-\Delta)^{\frac{1}{2}}$ and the ``potential'' $|b|$ (Appendix \ref{markov_sect}). This is the reason for the symmetry of the interval $I_s$ in spite of $-\Delta + b \cdot \nabla$ being a non-symmetric operator.

2.~In the proof of Theorem \ref{thm:markgrad2} carried out for the Kato class $\mathbf{K}^{d+1}_\delta$ the interval $I_s$ transforms into $[1,\infty[$, and the dependence of the properties of $D(\Lambda_r(b))$ on $\delta$ gets lost. The latter indicates the smallness of $\mathbf{K}^{d+1}_\delta$
as a subclass of $\mathbf{F}_\delta^{\scriptscriptstyle 1/2}$.

3.~We obtain, using \cite[Cor.~2.9]{KPS} (see examples in the beginning of the section):
 $$
 |x|^{-2}x \in \mathbf{F}^{\scriptscriptstyle \frac{1}{2}}_{\delta},
\quad \sqrt{\delta}=2^{-\frac{1}{2}} \frac{\Gamma\bigl(\frac{d-1}{4} \bigr)}{\Gamma\bigl(\frac{d+1}{4} \bigr)},
$$ 
$$
|x|^{-2}x \in \mathbf{F}_{\delta_1}, \quad \sqrt{\delta_1}=\frac{2}{d-2},
 $$ 
and so $\delta<\sqrt{\delta}_1$.

4.~Theorem \ref{thm:markgrad2}, compared to Theorem \ref{thm:markgrad}, covers a larger class of vector fields, and at the same time establishes stronger smoothness properties of $D(\Lambda_r(b))$: $D(\Lambda_r(b)) \subset \mathcal W^{1+\frac{1}{q},\,r}$, $q>r$, $r \in I_s$, while in Theorem \ref{thm:markgrad} $D(\Lambda_r(b)) \subset W^{1,\frac{dr}{d-2}}$, $r \in ]\frac{2}{2-\sqrt{\delta}}, \frac{2}{\sqrt{\delta}}[$.

Nevertheless, in spite of the inclusion $\mathbf{F}_{\delta_1} \subsetneqq \mathbf{F}_\delta^{\scriptscriptstyle 1/2}$, $\delta=\sqrt{\delta_1}$, cf.~\eqref{prop_incl}, the difference in the admissible values of $\delta$, $\delta_1$ shows that these classes \textit{are essentially incomparable.}

5.~In order to define $\Lambda(b,V) \supset -\Delta + b\cdot\nabla- V$, $0 \leq V,$ with $b\in\mathbf F_{\delta_b}$ it is enough to assume that $\|V^\frac{1}{2}(\lambda -\Delta)^{-\frac{1}{2}}\|_{2\rightarrow 2}\leq \sqrt{\delta_V}$ with $\delta_b + \delta_V < 1.$  However, if $b \in \mathbf F_{\delta_b}^{\scriptscriptstyle 1/2},$ then   
$\Lambda(b,V)$ can be defined only for $V$ such that $\|V^{\frac{3}{4}}(\lambda-\Delta)^{-\frac{3}{4}}\|_{2 \rightarrow 2} \leq (\delta_V)^\frac{3}{4}$ with $\delta_b + \delta_V < 1.$

\end{remarks}

\subsection{$L^r$-strong Feller semigroup on $C_\infty$ corresponding to $-\Delta +b\cdot \nabla$,  $b \in \mathbf F^{\scriptscriptstyle 1/2}_\delta$} 

\label{feller2_sect}

In Theorem \ref{thm:markgrad2} one can use the following approximation of $b$ by smooth vector fields:
\begin{equation}
\label{b_n_smooth_2}
b_n:=e^{\varepsilon_n \Delta}\mathbf{1}_n b, \quad \varepsilon_n \downarrow 0,
\end{equation}
where $\mathbf{1}_n$ is the indicator of $\{x \in \mathbb R^d: |x|\leq n, |b(x)| \leq n\}$
(alternatively, one can use the K.\,Friedrichs' mollifier). 

For any $\tilde{\delta}>\delta$ we can select a sequence $\varepsilon_n \downarrow 0$ such that
$b_n
 \in \mathbf{F}_{\tilde{\delta}}^{\scriptscriptstyle 1/2}$ with the same $\lambda=\lambda_\delta$ (see the argument in the proof of Theorem \ref{thm:marksolve}).

Since the assumptions on $\delta$ involve strict inequalities only, we may assume without loss of generality that $b_n$ defined by \eqref{b_n_smooth_2} are in $\mathbf{F}_{\delta}^{\scriptscriptstyle 1/2}$ with the same $\lambda_\delta$.

\begin{theorem}
\label{thm:markfeller2}
Let $d \geq 3$, $b:\mathbb R^d \rightarrow \mathbb R^d$,
$b \in \mathbf{F}_\delta^{\scriptscriptstyle \frac{1}{2}}$, $
m_d \delta <4\frac{d-2}{(d-1)^2}.$ Fix $r \in \big]d-1,\frac{2}{1-\sqrt{1-m_d\delta}}\big[$. By $\mathcal S$ denote the L.~Schwartz space of test functions. Then:

{\rm(\textit{i})}$$e^{-t\Lambda_{C_\infty}(b)}:=\bigl(e^{-t\Lambda_r(b)} \upharpoonright \mathcal S \bigr)_{C_\infty \rightarrow C_\infty}^{\clos} \quad \text{{\rm(}after a change on a set of measure zero{\rm)}}, \quad t \geq 0,$$
determines a positivity preserving contraction $C_0$ semigroup on $C_\infty$ {\rm(}Feller semigroup{\rm)}, where the semigroup $e^{-t\Lambda_{r}(b)}$ has been constructed in Theorem \ref{thm:markgrad2}.

\medskip

{\rm(\textit{ii})}~{\rm[}The $L^r$-strong Feller property\,{\rm]}~$\bigl( (\mu+\Lambda_{C_\infty}(b))^{-1} \upharpoonright L^r \cap C_\infty\bigr)^{\clos}_{L^r \rightarrow C_\infty} \in \mathcal B(L^r,  C^{0,\alpha})$,  $\mu>0$, $\alpha < 1-\frac{d-1}{r}$.

\smallskip

{\rm(\textit{iii})}~The
integral kernel $e^{-t\Lambda_{C_\infty}(b)}(x,y)$ of $e^{-t\Lambda_{C_\infty}(b)}$ determines the
transition probability function of a Hunt process.

\smallskip

{\rm(\textit{iv})} Let $\{b_n\}$ be given by \eqref{b_n_smooth_2}, then
\begin{equation*}
e^{-t\Lambda_{C_\infty}(b)}=\text{\small $s\text{-}C_\infty\text{-}$}\lim_n e^{-t\Lambda_{C_\infty}(b_n)} \quad \text{{\rm(}uniformly on every compact interval of $t \geq 0${\rm)}},
\end{equation*} 
where $\Lambda_{C_\infty}(b_n):=-\Delta+b_n\cdot \nabla$, 
$D(\Lambda_{C_\infty}(b_n))=(1-\Delta)^{-1}C_\infty$.

\end{theorem}

\begin{proof}
(\textit{i}), (\textit{ii}). Let $\Theta_{r}(\mu,b)$ be the operator-valued function introduced in Theorem \ref{thm:markgrad2}.

1) For every $\mu \geq \kappa_d\lambda$,
$\Theta_r(\mu,b) \mathcal S \subset C_\infty$, and
$\|\Theta_{r}(\mu,b)f\|_{\infty} \leq  \mu^{-1}\|f\|_\infty$, $f \in \mathcal S$.

\noindent Indeed, by Theorem \ref{thm:markgrad2}(\textit{v}), since $r>d-1$, 
$\Theta_r(\mu,b) L^r \subset C_\infty$, which yields the first assertion.
Since $e^{-t\Lambda_r(b)}$ is an $L^\infty$-contraction, the second assertion follows.

In view of 1), we can define ($\mu \geq \kappa_d\lambda$)
\begin{equation*}
\Theta_{C_\infty}(\mu,b):=\bigl(\Theta_{r}(\mu,b) \upharpoonright \mathcal S \bigr)_{C_\infty \rightarrow C_\infty}^{\clos} \in \mathcal B(C_\infty)\;\text{ (after a change on a set of measure zero)}.
\end{equation*}

2) $\mu \Theta_{C_\infty}(\mu,b) \overset{s}{\rightarrow} 1 \text{  as $\mu \uparrow \infty$ in $C_\infty$}$.

\noindent Indeed,
since $\|\mu\Theta_{C_\infty}(\mu,b)\|_{\infty \rightarrow \infty} \leq  1$, and $\mathcal S$ is dense in $C_\infty$, it suffices to 
prove that
$$
\mu \Theta_{r}(\mu,b)f \overset{s}{\rightarrow} f \text{  as $\mu \uparrow \infty$ in $C_\infty$}, \quad \text{ for every } f \in \mathcal S.
$$
Put $\Theta_r \equiv \Theta_r(\mu,b)$, $T_r \equiv T_r(\mu,b)$.
Since $\mu(\mu-\Delta)^{-1}f \overset{s}{\rightarrow} f$ in $C_\infty$,
it suffices to show that $\|\mu \Theta_{r}f- \mu(\mu-\Delta)^{-1}f\|_\infty \rightarrow 0
$. For each $f \in \mathcal S$ there is $h \in \mathcal S$ such that $f=(\lambda-\Delta)^{-\frac{1}{2}}h$, where $\lambda=\lambda_\delta>0$.
Let $q>r$.
Write
$$
\Theta_{r}f- (\mu-\Delta)^{-1}f =
-(\mu-\Delta)^{-\frac{1}{2}-\frac{1}{2q}} Q_{r}(q)   
\bigl(1+T_{r} \bigr)^{-1} b^{\frac{1}{r}} (\lambda-\Delta)^{-\frac{1}{2}} \cdot (\mu-\Delta)^{-1}\nabla h.
$$
Using estimates $\|(1+T_r)^{-1}\|_{r \rightarrow r} \leq (1-m_dc_r\delta )^{-1}$, $\|Q_{r}(q)\|_{r \rightarrow r} \leq K_{1,q}<\infty$ (cf.~proof of Theorem \ref{thm:markgrad2}(\textit{iii})) and
$$
\|(\mu-\Delta)^{-\frac{1}{2}-\frac{1}{2q}}\|_{r \rightarrow \infty} \leq c \mu^{-\frac{1}{2}+\frac{d}{2r}-\frac{1}{2q}}, \quad c<\infty,
$$
we obtain
$$
\|\Theta_{r}f - (\mu-\Delta)^{-1}f\|_\infty 
\leq C \mu^{-\frac{1}{2}+\frac{d}{2r} -\frac{1}{2q}} \mu^{ -1}  \|\nabla h\|_r.
$$
Since $r>d-1$, choosing $q$ sufficiently close to $r$, we obtain  $$-\frac{1}{2}+\frac{d}{2r}-\frac{1}{2q}-1<-1,$$ so $\mu \Theta_{r} - \mu(\mu-\Delta)^{-1} \overset{s}{\rightarrow} 0$ in $C_\infty$. The proof of 2) is completed.

$\Theta_r(\mu,b)$ satisfies the resolvent identity for $\mu \geq
\kappa_d\lambda$ (Theorem \ref{thm:markgrad2}(\textit{i})), and so does
$\Theta_{C_\infty}(\mu,b) \upharpoonright \mathcal S$. Therefore, $\Theta_{C_\infty}(\mu,b)$ is a pseudo-resolvent for $\mu \geq
\kappa_d\lambda$. The latter and 2) yield (by Theorems \ref{hille_thm1} and \ref{hille_thm2}, Appendix \ref{hille_theory_sect}):
$\Theta_{C_\infty}(\mu,b)$ is the resolvent of a densely defined closed operator $\Lambda_{C_\infty}(b)$. By 1),
$\|(\mu+\Lambda_{C_\infty}(b))^{-1}\|_{\infty \rightarrow \infty} \leq  \mu^{-1}$, 
so $-\Lambda_{C_\infty}(b)$ is the generator of a contraction $C_0$ semigroup $e^{-t\Lambda_{C_\infty}(b)}.$ Clearly, $e^{-t\Lambda_{C_\infty}(b)}$ is positivity preserving, so we have proved (\textit{i}). Now (\textit{ii}) follows from Theorem \ref{thm:markgrad2}(\textit{iv}).
The proof of (\textit{iii}) is standard.

\smallskip

(\textit{iv}) Note that
\begin{equation*}
(\mu+\Lambda_{C_\infty}(b_n))^{-1} \upharpoonright \mathcal S = \Theta_r(\mu,b_n) \upharpoonright \mathcal S, \quad n=1,2,\dots, \quad \mu \geq \kappa_d\lambda,
\end{equation*}
The latter, combined with
\begin{equation}
\label{feller_ii}
\tag{$\star$}
\Theta_r(\mu,b_n) f \overset{s}{\rightarrow} \Theta_r(\mu,b) f \text{ in $C_\infty$}, \quad \mu \geq \kappa_d\lambda, \; f \in \mathcal S,
\end{equation}
yields 
$
(\mu+\Lambda_{C_\infty}(b_n))^{-1} \overset{s}{\rightarrow} (\mu+\Lambda_{C_\infty}(b))^{-1}$ in $C_\infty$, $\mu \geq \kappa_d\lambda$ $\Rightarrow$ (\textit{iv}).

\smallskip

\begin{proof}[Proof of \eqref{feller_ii}] It suffices to prove that
\[
(\mu-\Delta)^{-\frac{1}{2}-\frac{1}{2q}}Q_r(q,b_n)(1+T_r(b_n))^{-1}G_r(b_n) \overset{s}\rightarrow (\mu-\Delta)^{-\frac{1}{2}-\frac{1}{2q}}Q_r(q,b)(1+T_r(b))^{-1}G_r(b) \text{ on $\mathcal S$ in $C_\infty$}.
\] 
We choose $q$ close to $d-1$ so that $(\mu-\Delta)^{-\frac{1}{2}-\frac{1}{2q}}L^r \subset C_\infty$.
Thus it suffices to prove that
\[
G_r(b_n) \overset{s}\rightarrow G_r(b), \; (1+T_r(b_n))^{-1} \overset{s}\rightarrow (1+T_r(b))^{-1}, \; Q_r(q,b_n) \overset{s}\rightarrow Q_r(q,b) \text{ in $L^r$}.
\]
In turn, since $(1+T_r(b_n))^{-1}-(1+T_r(b))^{-1}= (1+T_r(b_n))^{-1}(T_r(b)-T_r(b_n))(1+T_r(b))^{-1},$ it suffices to prove that $T_r(b_n)\overset{s}\rightarrow T_r(b).$ Finally,
\[
T_r(b_n)-T_r(b)= T_r(b_n)- b_n^{\frac{1}{r}} \cdot\nabla (\mu-\Delta)^{-1}|b|^{\frac{1}{r'}}+ b_n^{\frac{1}{r}} \cdot\nabla (\mu-\Delta)^{-1}|b|^{\frac{1}{r'}} - T_r(b),
\]
and hence we have to prove that
\[
 b_n^{\frac{1}{r}} \cdot \nabla  (\mu-\Delta)^{-1}|b|^{\frac{1}{r'}} - T_r(b):=J_n^{(1)} \overset{s}\rightarrow  0 \text{ and } T_r(b_n)- b_n^{\frac{1}{r}} \cdot \nabla (\mu-\Delta)^{-1}|b|^{\frac{1}{r'}}:= J_n^{(2)} \overset{s}\rightarrow 0.
 \]
 Now, by the Dominated Convergence Theorem, $G_r(b_n) \overset{s}\rightarrow G_r(b),$ $J_n^{(1)}|_\mathcal{E} \overset{s}\rightarrow 0.$ Also
 \begin{align*}
\| J_n^{(2)}f \|_r= & \|G_r(b_n)(|b_n|^{\frac{1}{r'}}-|b|^{\frac{1}{r'}})f \|_r \\
& \leq \|G_r(b_n)\|_{r \rightarrow r} \|(|b_n|^{\frac{1}{r'}}-|b|^{\frac{1}{r'}})f \|_r \\
& \leq m_d(1+\delta)\mu^{-\frac{1}{2r'}} \|(|b_n|^{\frac{1}{r'}}-|b|^{\frac{1}{r'}})f \|_r, \quad (f \in \mathcal{E}).
\end{align*}
Thus, $ J_n^{(2)}|_ \mathcal{E} \overset{s}\rightarrow 0.$ Since $\|J_n^{(2)}\|_{r \rightarrow r}, \|J_n^{(1)}\|_{r \rightarrow r} \leq m_d \delta,$ we conclude that $T_r(b_n)\overset{s}\rightarrow T_r(b).$
It is clear now that $ Q_r(q, b_n) \overset{s}\rightarrow Q_r(q, b)$.

The proof of \eqref{feller_ii} is completed.
\end{proof}

The proof of Theorem \ref{thm:markfeller2} is completed.
\end{proof}

\begin{remarks}
1.~The assertion (\textit{iv}) of Theorem \ref{thm:markfeller2} holds for any $\{b_n\} \subset C^1(\mathbb R^d, \mathbb R^d) \cap \mathbf F_\delta^{\scriptscriptstyle 1/2}$, $b_n \rightarrow b \;\; \mathcal L^d$ a.e., in particular, for the $b_n$'s given by \eqref{b_n_smooth_2}, but not for $b_n$'s as in Theorem \ref{thm:markfeller}.

2.~Theorem \ref{thm:markgrad2} allows us to move the proof of convergence in $C_\infty$  to $L^r, \; r >d-1,$ a space having much weaker topology (locally). The same idea has been realized in the proof of Theorem \ref{thm:markfeller}.

3.~In comparison with the construction of a Feller semigroup in section \ref{feller1_sect}, here the relative ease of the construction stems from the fact that one already has the limiting object, i.e.\,$\Theta_p(\mu,b)$, $p>d-1$, while 
in section \ref{feller1_sect} one has to work with Cauchy's sequences.

4.~Theorem \ref{thm:markgrad2} and Theorem \ref{thm:markfeller2} admit straightforward generalization (a) to the operator $(-\Delta)^{\frac{\alpha}{2}} + b \cdot \nabla$, $1<\alpha<2$, and (b) to the operator $-\nabla \cdot a \cdot \nabla + b \cdot \nabla$ with $a \in (H_u)$ uniformly H\"{o}lder continuous.
\end{remarks}

\appendix

\section{Monotone Convergence Theorem for sesquilinear forms}

\label{MCT_sect}

Let $\mathcal{H}$ be a (complex) Hilbert space with the inner product $\langle f, g \rangle$ and norm $\| f \| = \langle f, f \rangle^\frac{1}{2}.$ Let $\mathcal{T}$ denote the family of all closed, symmetric, non-negative, densely defined sesquilinear forms in $\mathcal{H}.$ If $t \in \mathcal{T},$ then there exist a unique self-adjoint operator $T \geq 0$ such that
\begin{align*}
& t[u,v] = \langle T u,v \rangle, \quad \quad u \in D(T)\;\; \big(\subset D(t)\;\big), \quad v \in D(t).
\end{align*}
Then
$$
t[u,v] = \big \langle T^\frac{1}{2} u, T^\frac{1}{2} v \rangle, \quad D(T^\frac{1}{2})=D(t).
$$
In this case we say that $T$ is associated with $t$ and write $T \leftrightarrow t$ or/and $t = t_T.$

Let $a, b \in \mathcal{T}$ be such that $D(a) \cap D(b)$ is dense in $\mathcal{H}.$ Then $a + b \in \mathcal{T}.$ Thus, for $a, b, a+b \in \mathcal{T}$ and $A \leftrightarrow a, \; B \leftrightarrow b,$ we write $A\dot{+} B \leftrightarrow a + b.$ $A \dot{+} B$ is called the form sum of $A$ and $B.$ $A \dot{+} B,$ a self-adjoint extension of the algebraic sum $A + B,$ possesses some exclusive properties described below.

We endow $\mathcal{T}$ with a semi-order $\prec :$
\[
a \prec b \;\; \Leftrightarrow \;\; D(a) \supset D(b) \text{ and } a[u] \leq b[u] \;\;(u \in D(b)).
\]
Here and below $a[u] \equiv a[u,u].$ If $a,b \in \mathcal{T}$ and $A \leftrightarrow a, B \leftrightarrow b,$ we write $A \leq B$ if and only if $a \prec b.$ In this case $(\lambda + B)^{-1} \leq (\lambda + A)^{-1}$ $(\lambda > 0 )$ in the sense that
\[
\| (\lambda + B)^{-\frac{1}{2}}f \| \leq \| (\lambda + A)^{-\frac{1}{2}}f \| \;\;\; (f \in \mathcal{H} ).
\]

\begin{theorem}[{Convergence from below, see e.g.\,\cite[Ch.\,VIII, sect.\,3]{Ka} or \cite{DM}}]
\label{thm:markconb}
Let $\{a_n \}_{n=1}^\infty \subset \mathcal{T}$ be such that
\[
a_1 \prec a_2 \prec \dots .
\]
Define $a$ by
\[
a[u] := \lim_n a_n[u], \;\;\; D(a) := \bigg \{ u \in \bigcap_{n=1}^\infty D(a_n) \mid \text{ the finite } \lim_n a_n[u] \text{ exists } \bigg \}.
\]
Suppose that $D(a)$ is dense in $\mathcal{H}.$ Then $a \in \mathcal{T}$ and, for all $ u, v \in D(a),$
\[
a_n [u, v] \rightarrow a[u, v].
\]
Let $A \leftrightarrow a$ and $A_n \leftrightarrow a_n.$ Then
\begin{align*}
& (\lambda + A_n)^{-1} \; \overset{s} \rightarrow (\lambda + A)^{-1} \quad \quad (\Real \lambda > 0),\\
& (\lambda + A_n)^\frac{1}{2}u \; \overset{s} \rightarrow (\lambda + A)^\frac{1}{2}u \quad \quad ( u \in D(a), \; \lambda > 0).
\end{align*}  
\end{theorem}

\begin{remark*} This theorem is extremely useful for constructing operator realizations of formal differential expressions. See e.g.\;section \ref{markov_gen_sect}. Also, it is a proper tool for perturbation theory of self-adjoint operators.
\end{remark*}

\section{The criteria of Phillips and Stampacchia}There is a characterization of Markov semigroups in terms of their generators.

\label{phillips_sect}

Let $X$ be a set and $\mu$ a measure on $X.$ 
Recall that a $C_0$ semigroup $T^t$, $t \geq 0$, of contractions
on $L^p=L^p(X,\mu)$, $p \in [1,\infty[$, is called Markov if, for each $t > 0$,
\begin{equation}\label{eqn:i_}
	T^t L^p_+ \subset L^p_+, \tag{i}
\end{equation}
\begin{equation}\label{eqn:ii_}
	(f \in L^p, |f| \leq 1) \Rightarrow |T^t f| \leq 1. \tag{ii}
\end{equation}

Let $-A$ denote the generator of $T^t$.
Let us introduce the following conditions:
\begin{equation}\label{eqn:iprime}
	[Af, f^+] \geq 0 
	\;\;\; (f \in \mathcal{D}(A) \cap \Real L^p) \quad \text{(R.\,Phillips)}, \tag{i$^\prime$}
\end{equation}
\begin{equation}\label{eqn:iiprime}
	\Real [Af, f - f_\wedge] \geq 0
	\;\;\; (f \in \mathcal{D}(A)) \quad \text{(G.\,Stampacchia)}. \tag{ii$^\prime$}
\end{equation}
Here $\sgn z:=z/|z|$ if $z \neq 0$, and $\sgn 0:=0$; $f^+:=f\vee 0;$ $f_\wedge := (|f| \wedge 1) \sgn f;$ $[f, g] := \langle f, |g|^{p-1} \sgn g \rangle \|g\|^{2-p}_p$, $f, g \in L^p$, the semi-inner product in $L^p$. Here and elsewhere $f \wedge g \equiv \inf \{f,g \}$,
$f \vee g \equiv \sup \{f,g \}$.

It follows from Proposition 1 and Proposition 2 below (see \cite[Sect.\,1]{LS}) that
\[
	\eqref{eqn:i_} + \eqref{eqn:ii_} \Leftrightarrow \eqref{eqn:iprime} + \eqref{eqn:iiprime}
\]

This equivalence is useful first of all for actual verification of \eqref{eqn:i},\eqref{eqn:ii} in the case when $A$ is an operator realization of a formal partial differential expression.
\begin{proposition}
\label{prop1} 
Let $e^{-tA}$, $t \geq 0$, be a contraction $C_0$ semigroup in $L^p$ such that $e^{-tA} \Real L^p \subset \Real L^p$. Then
\[
	\eqref{eqn:i_} \Leftrightarrow \eqref{eqn:iprime}.
\]
\end{proposition}

\begin{proposition}
\label{prop2} 

Let $e^{-tA}$, $t \geq 0$, be a contraction $C_0$ semigroup in $L^p$. Then
\begin{equation*}
	\|e^{-tA}v\|_\infty \leq \|v\|_\infty \;\;\; (v \in L^p \cap L^\infty; t \geq 0) 
\end{equation*}
\textit{if and only if}
\begin{equation*}
	\Real [Af, f - f_\wedge] \geq 0 \;\;\; (f \in \mathcal{D}(A)).
\end{equation*}
\end{proposition}

\smallskip

\section{Trotter's Approximation Theorem}

\label{trotter_sect}

Consider a sequence $\{e^{-tA_k}\}_{k=1}^\infty$ of $C_0$ semigroups on a (complex) Banach space $Y$.

\begin{theorem}[{H.F.\,Trotter \cite[Ch.\,IX, sect.\,2]{Ka}}]
Let
$\sup_k\|(\mu+A_k)^{-m}\|_{Y \rightarrow Y} \leq M(\mu-\omega)^{-m}$, $m=1,2,\dots$, $\mu>\omega$, and 
$s\mbox{-}\lim_{\mu \rightarrow \infty}\mu(\mu+ A_k)^{-1}=1$ uniformly in $k$, and let $s\mbox{-}\lim_{k}(\zeta+ A_k)^{-1}$
exist for some $\zeta$ with ${\rm Re\,} \zeta>\omega$. Then there is a $C_0$ semigroup $e^{-tA}$ such that
$$
(z+ A_k)^{-1} \overset{s}{\rightarrow} (z+ A)^{-1} \quad \text{ for every } {\rm Re}\, z>\omega,
$$
and
$$
e^{-tA_k} \overset{s}{\rightarrow} e^{-tA}
$$
uniformly in any finite interval of $t \geq 0$.
\end{theorem}

The first condition of the theorem is satisfied if e.g.\,$\sup_k\|(z+A_k)^{-1}\|_{Y \rightarrow Y} \leq C|z-\omega|^{-1}$, ${\rm Re\,}z>\omega$.

\smallskip

\section{Hille's theorems on pseudo-resolvents}

\label{hille_theory_sect}

Let $Y$ be a (complex) Banach space. A pseudo-resolvent $R_\zeta$ is a function defined on a subset $\mathcal O$ of the complex $\zeta$-plane with values in $\mathcal B(Y)$ such that
$$
R_\zeta - R_\eta=(\eta-\zeta)R_\zeta R_\eta, \quad \zeta, \eta \in \mathcal O.
$$
Clearly, $R_\zeta$ have common null-set.

\begin{theorem}
\label{hille_thm1}
If the null-set of $R_\zeta$ is $\{0\}$, then $R_\zeta$ is the resolvent of a closed linear operator $A$, the range of $R_\zeta$ coincides with $D(A)$, and $A=R_{\zeta}^{-1}-\zeta$.
 \end{theorem}
\begin{proof}
Put $A:=R_{\zeta}^{-1}-\zeta$. Since $R_{\zeta}$ is closed, so is 
$R_{\zeta}^{-1}$ and $A$. A straightforward calculation shows that $(\zeta+A)R_\zeta f=f$, $f \in Y$, and $R_\zeta(\zeta+A)g=g$, $g \in D(A)$, as needed.
\end{proof}

\begin{theorem}
\label{hille_thm2}
If there exists a sequence of numbers $\{\mu_k\} \subset \mathcal O$ such that $\lim_k|\mu_k|=\infty$ and $\sup_k \|\mu_kR_{\mu_k}\|_{Y \rightarrow Y}<\infty$, then the set $\{y \in Y: \lim_k \mu_k R_{\mu_k}y=y\}$ is contained in the closure of the range of $R_\zeta$.
\end{theorem}
\begin{proof}
Indeed, let $\lim_k \mu_k R_{\mu_k}y=y$. That is, for every $\varepsilon>0,$ there exists $k$ such that $\|y-\mu_k R_{\mu_k}y\|<\varepsilon$, so $y$ belongs to the closure of the range of $R_\zeta$.
\end{proof}

See \cite[Sect.\,5.2]{HP},  \cite[Ch.\,VIII, sect.\,4]{Yos}.

\smallskip

\section{$L^r$-inequalities for symmetric Markov generators}

\label{markov_sect}

Let $X$ be a set and $\mu$ a $\sigma$-finite measure on X. Let $T^t=e^{-t A}$, $t \geq 0$, be a symmetric Markov semigroup in $L^2(X, \mu)$. 

\begin{theorem}[{\cite[Theorem 2.1]{LS}}]
\label{thm:markovest}
If $f \in D(A_r)$ for some $r \in ]1, \infty[$, then $f_{(r)} := |f|^\frac{r}{2} \sgn f, |f|^{\frac{r}{2}}  \in D(A^\frac{1}{2})$ and
\[
\label{eqn:est1}
	\frac{4}{r r^\prime} \| A^\frac{1}{2} f_{(r)} \|^2_2 \leq
		\Real \langle A_r f, |f|^{r-1} \sgn f \rangle \leq \varkappa(r) \| A^\frac{1}{2} f_{(r)} \|^2_2,
		\tag{i}
\]
where $\varkappa(r) := \sup_{s \in [0, 1]} \{ (1 + s^\frac{1}{r}) (1 + s^\frac{1}{r'}) (1 + s^\frac{1}{2})^{-2} \}, \; r^\prime = r (r-1)^{-1}$;

\[
\label{eqn:est2}
	\big| \Imag \langle A_r f, |f|^{r-1} \sgn f \rangle \big| \leq \frac{|r-2|}{2\sqrt{r-1}}
		\Real \langle A_r f, |f|^{r-1} \sgn f \rangle;
		\tag{ii}
\]
If $0 \leq f \in D(A_r)$, then 
\[
\label{eqn:est3}
	\frac{4}{r r^\prime} \| A^\frac{1}{2} f^\frac{r}{2} \|^2_2 \leq
		\langle A_r f, f^{r-1} \rangle \leq \| A^\frac{1}{2} f^\frac{r}{2} \|^2_2;
		\tag{iii}
\]
If $r \in [2, \infty[$ and $f \in D(A) \cap L^\infty$, 
then $|f|^{\frac{r}{2}}$, $f_{(r)} \in D(A^\frac{1}{2})$ and
\[
\label{eqn:est1__}
	\frac{4}{r r^\prime} \| A^\frac{1}{2} f_{(r)} \|^2_2 \leq
		\Real \langle A f, |f|^{r-1} \sgn f \rangle \leq \varkappa(r) \| A^\frac{1}{2} f_{(r)} \|^2_2,
				\tag{iv}
		\]
\[
\label{eqn:est2__}
	\big| \Imag \langle A f, |f|^{r-1} \sgn f \rangle \big| \leq \frac{|r-2|}{2\sqrt{r-1}}
		\Real \langle A f, |f|^{r-1} \sgn f \rangle;
		\tag{v}
\]
If $r \in [2, \infty[$ and $0 \leq f \in D(A) \cap L^\infty$, then $f^\frac{r}{2} \in D(A^\frac{1}{2})$ and
\[
\label{eqn:est4}
	\frac{4}{r r^\prime} \| A^\frac{1}{2} f^\frac{r}{2} \|^2_2 \leq
		\langle A f, f^{r-1} \rangle \leq \| A^\frac{1}{2} f^\frac{r}{2} \|^2_2.
		\tag{vi}
\]
\end{theorem}

(The proof of (iii) works in the assumptions of (vi); the proof of (i), (ii) works in the assumptions of (iv), (v).)

\bigskip

\section{Extrapolation Theorem}

\label{appendix_B}

\begin{theorem}[{T.\,Coulhon-Y.\,Raynaud.} {\cite[Prop.\,II.2.1, Prop.\,II.2.2]{VSC}}.]
Let $U^{t,s}: L^1 \cap L^\infty \rightarrow L^1 + L^\infty$ be a two-parameter evolution family of operators:
\[U^{t,s} = U^{t,\tau}U^{\tau,s}, \quad 0 \leq s < \tau < t \leq \infty.
\]
Suppose that, for some $1 \leq p < q < r \leq \infty,$ $\nu > 0,$ $M_1$ and $M_2,$ the inequalities
\[
\| U^{t,s} f \|_p \leq M_1 \| f \|_p \quad \text{ and } \quad \| U^{t,s} f \|_r \leq M_2 (t-s)^{-\nu} \|  f \|_q
\]
are valid for all $(t,s)$ and $f \in L^1 \cap L^\infty.$ Then
\[
\| U^{t,s} f \|_r \leq M (t-s)^{-\nu/(1-\beta)} \| f \|_p ,
\]
where $\beta = \frac{r}{q}\frac{q-p}{r-p}$ and $M = 2^{\nu/(1-\beta)^2} M_1 M_2^{1/(1-\beta)}.$
\end{theorem}

\begin{proof} Set $2 t_s=t+s.$ The hypotheses and H\"older's inequality imply
\begin{align*}
\| U^{t, s} f \|_r & \leq M_2 (t-t_s)^{-\nu} \| U^{t_s,s} f \|_q \\
& \leq M_2 (t-t_s)^{-\nu} \| U^{t_s,s} f \|_r^\beta \;\| U^{t_s,s} f \|_p^{1-\beta} \\
& \leq M_2 M_1^{1-\beta} (t-t_s)^{-\nu} \| U^{t_s,s} f \|_r^\beta \;\| f \|_p^{1-\beta},
\end{align*}
and hence
\[
(t-s)^{\nu/(1-\beta)} \| U^{t,s} f \|_r/\| f \|_p \leq M_2 M_1^{1-\beta} 2^{\nu/(1-\beta)} \big [(t_s -s)^{-\nu/(1-\beta)} \| U^{t_s,s} f \|_r\;/\| f \|_p \big ]^\beta.
\]
Setting $R_{2 T}: = \sup_{t-s \in ]0,T]} \big [ (t-s)^{\nu/(1-\beta)} \| U^{t,s} f \|_r/\| f \|_p \big ],$ we obtain from the last inequality that $R_{2 T} \leq M^{1-\beta} (R_T)^\beta.$ But $R_T \leq R_{2T}$, and so $R_T \leq M.$
\end{proof}

\begin{corollary}
Let $U^{t,s}: L^1 \cap L^\infty \rightarrow L^1 + L^\infty$ be an evolution family of operators. Suppose that, for some $1 < p <q < r \leq \infty,$ $\nu > 0,$ $M_1$ and $M_2,$ the inequalities
\[
\| U^{t,s} f \|_r \leq M_1 \| f \|_r \quad \text{ and } \quad \| U^{t,s} f \|_q \leq M_2 (t-s)^{-\nu} \|  f \|_p
\]
are valid for all $(t,s)$ and $f \in L^1 \cap L^\infty.$ Then
\[
\| U^{t,s} f \|_r \leq M (t-s)^{-\nu/(1-\beta)} \| f \|_p ,
\]
where $\beta = \frac{r}{q} \frac{q-p}{r-p}$ and $M = 2^{\nu/(1-\beta)^2} M_1 M_2^{1/(1-\beta)}.$
\end{corollary}

\section{N.\,Meyers Embedding Theorem}

\begin{theorem}[{See 
 \cite[Theorem 2]{Me}}]
Let $a \in (H_u)$, $\sigma I \leq a(x) \leq \xi I$ $\mathcal L^d$ a.e.\,on $\mathbb R^d$. In $L^2=L^2(\mathbb R^d,\mathcal L^d)$ define $A$, the Dirichlet extension of $-\nabla\cdot a\cdot\nabla$. Then there exists $p>2$ determined by the condition $\|\nabla (1-\Delta)^{-\frac{1}{2}}\|^2_{p \rightarrow p}<\frac{\xi}{\xi-\sigma}$
such that $(1+A_p)^{-1}$ extends to 
$$
(1+A_p)^{-1} \in \mathcal B(\mathcal W^{-1,p},\mathcal W^{1,p}).
$$
\end{theorem}
\begin{proof}First, let $a \in (H_u) \cap [C^\infty]^{d \times d}$.
Set $\tau:=\xi I-a$. Then
$$
(1+A)^{-1}f=(1-\xi\Delta)^{-1}f-(1-\xi\Delta)^{-1} \nabla \cdot \tau \cdot \nabla (1+A)^{-1}f, \quad f \in C_c,
$$
$$
\|\nabla (1+A)^{-1}f\|_p \leq \|\nabla(1-\xi\Delta)^{-1}f\|_p + \|\nabla (1-\xi\Delta)^{-\frac{1}{2}}\|_{p \rightarrow p} \|(1-\xi\Delta)^{-\frac{1}{2}} \nabla \cdot \tau \cdot \nabla (1+A)^{-1}f\|_p.
$$
Let $\varphi \in L^{p'}$, $p>1$, $F:=\nabla (1-\xi\Delta)^{-\frac{1}{2}}\varphi$, $G:=\nabla (1+A)^{-1}f$.
Then, using $v \cdot \tau \cdot \bar{v} \leq (\xi-\sigma)|v|^2$, we obtain
\begin{align*}
&\big|\big\langle \varphi, (1-\xi\Delta)^{-\frac{1}{2}} \nabla \cdot \tau \cdot \nabla (1+A)^{-1}f \big\rangle\big| \leq (\xi-\sigma)
\langle |F|, |G|\rangle \leq (\xi-\sigma) \|F\|_{p'}\|G\|_p \\
& \leq \frac{\xi-\sigma}{\sqrt{\xi}}\|\nabla (\xi^{-1}-\Delta)^{-\frac{1}{2}}\|_{p' \rightarrow p'} \|\varphi\|_{p'}\|\nabla (1+A)^{-1}f\|_p.
\end{align*}
Therefore,
$$
\|(1-\xi\Delta)^{-\frac{1}{2}} \nabla \cdot \tau \cdot \nabla (1+A)^{-1}f\|_p \leq \frac{\xi-\sigma}{\sqrt{\xi}}\|\nabla (1-\Delta)^{-\frac{1}{2}}\|_{p' \rightarrow p'} \|\nabla (1+A)^{-1}f\|_p.
$$
We arrive at
$$
\|\nabla (1+A)^{-1}f\|_p \leq \|\nabla(1-\xi\Delta)^{-1}f\|_{\mathcal W^{-1,p} \rightarrow L^p} \|f\|_{\mathcal W^{-1,p}} + \frac{\xi-\sigma}{\xi}\|\nabla (1-\Delta)^{-\frac{1}{2}}\|^2_{p \rightarrow p} \|\nabla (1+A)^{-1}f\|_p,
$$
i.e.
$$
\|\nabla (1+A)^{-1}f\|_p \leq  \|\nabla(1-\xi\Delta)^{-1}f\|_{\mathcal W^{-1,p} \rightarrow L^p}\big(1-\kappa(p)\big)^{-1}  \|f\|_{\mathcal W^{-1,p}}, 
$$
where $\kappa(p):=\frac{\xi-\sigma}{\xi} \|\nabla (1-\Delta)^{-\frac{1}{2}}\|^2_{p \rightarrow p}$.

Let $a_n:=e^{\frac{\Delta}{n}}a$. Using that $(1+A_p(a_n))^{-1} \rightarrow (1+A_p(a))^{-1}$ strongly in $L^p$ and that $\nabla$ is closed, we arrive at the required estimate.
\end{proof}

\label{meyers_sect}

\end{document}